\documentclass{article}
\usepackage{amsmath}
\usepackage{amsthm}
\usepackage{amssymb}
\usepackage{hyperref}
\usepackage{verbatim}
\usepackage{graphicx}
\usepackage[usenames,dvipsnames,svgnames,table]{xcolor}
\usepackage{mathrsfs}  

\usepackage[hmargin=3cm,vmargin=3.5cm]{geometry}
\def\a{\alpha}
\def\b{\beta}
\def\ga{\gamma}
\def\Ga{\Gamma}
\def\de{\delta}
\def\De{\Delta}
\def\ep{\epsilon}

\def\la{\lambda}
\def\La{\Lambda}
\def\si{\sigma}

\def\om{\omega}
\def\Om{\Omega}

\def\th{\theta}

\def\nab{\nabla}
\def\varep{\varepsilon}

\def\DD{{\cal D}}

\def\II{{\cal I}}

\def\HH{{\cal H}}

\def\DD{{\cal D}}

\def\RR{{\cal R}}
\def\SS{{\cal S}}

\newcommand{\N}[0]{\mathbb{N}}
\newcommand{\F}[0]{\mathbb{F}}
\newcommand{\R}[0]{\mathbb{R}}
\newcommand{\Z}[0]{\mathbb{Z}}

\newcommand{\C}[0]{\mathbb{C}}

\newcommand{\T}[0]{\mathbb{T}}

\newcommand{\supp}{\mathrm{supp} \,}
\newcommand{\suppt}{\mathrm{supp}_t \,}

\newcommand{\fr}[2]{\frac{#1}{#2}}

\newcommand{\vect}[1]{\left[ \begin{array}{c} #1 \end{array} \right]}
\newcommand{\mat}[2]{\left[ \begin{array}{ #1} #2 \end{array} \right]}
\newcommand{\ALI}[1]{\begin{align*} #1 \end{align*}}
\newcommand{\tx}[1]{\mbox{#1}}

\newcommand{\lsm}[0]{\lesssim}

\newcommand{\pr}[0]{\partial}
\newcommand{\nb}{\nabla}

\newcommand{\co}[1]{\|#1\|_{C^0}}
\newcommand{\cda}[1]{\|#1\|_{\dot{C}^\a}}

\newcommand{\Ddt}[0]{\overline{D}_t}

\newcommand{\hxi}[0]{\widehat{\Xi}}          
\newcommand{\lhxi}[0]{\log \widehat{\Xi} \,}
\newcommand{\nhat}[0]{\widehat{N}}
\newcommand{\plhxi}[0]{(\log \widehat{\Xi})}
\newcommand{\va}[0]{\vec{a}}
\newcommand{\chka}[0]{\check{a}}
\newcommand{\vcb}[0]{\vec{b}}
\newcommand{\vcc}[0]{\vec{c}}
\newcommand{\vce}[0]{\vec{e}}
\DeclareMathAlphabet{\mathpzc}{OT1}{pzc}{m}{it}
\newcommand{\hh}[0]{\mathpzc{h}}
\newcommand{\bp}[0]{\bar{p}}
\newcommand{\hq}[0]{\hat{q}}

\newcommand{\VR}[0]{\mathring{V}}

\newcommand{\ever}[0]{\left( \fr{e_v}{e_R} \right)}
\newcommand{\wtld}[1]{\widetilde{#1}}

\newcommand{\ali}[1]{ \begin{align} #1 \end{align} }

\def\XXint#1#2#3{{\setbox0=\hbox{$#1{#2#3}{\int}$}
     \vcenter{\hbox{$#2#3$}}\kern-.5\wd0}}

\newcommand{\hltRed}[1]{#1} 

\newtheorem{thm}{Theorem}
\newtheorem{lem}{Lemma}[section]

\newtheorem{prop}{Proposition}[section]

\newtheorem{conj}{Conjecture}

\theoremstyle{definition}
\newtheorem{defn}{Definition}[section]

\theoremstyle{remark}

\title{ A Proof of Onsager's Conjecture }
\author{ Philip Isett\thanks{Department of Mathematics, University of Texas at Austin, Austin, TX  (\href{mailto:isett@math.utexas.edu}{isett@math.utexas.edu}).  The work of P. Isett is supported by the National Science Foundation under Award No. DMS-1402370 and DMS-1700312.} 
}
\date{ }

\begin{document}
\maketitle

\begin{abstract}
For any $\alpha < 1/3$, we construct weak solutions to the $3D$ incompressible Euler equations in the class $C_tC_x^\alpha$ that have nonempty, compact support in time on ${\mathbb R} \times {\mathbb T}^3$ and therefore fail to conserve the total kinetic energy.  This result, together with the proof of energy conservation for $\alpha > 1/3$ due to [Eyink] and [Constantin, E, Titi], solves Onsager's conjecture that the exponent $\alpha = 1/3$ marks the threshold for conservation of energy for weak solutions in the class $L_t^\infty C_x^\alpha$.  The previous best results were solutions in the class $C_tC_x^\alpha$ for $\alpha < 1/5$, due to [Isett], and in the class $L_t^1 C_x^\alpha$ for $\alpha < 1/3$ due to [Buckmaster, De Lellis, Sz\'{e}kelyhidi], both based on the method of convex integration developed for the incompressible Euler equations by [De Lellis, Sz\'{e}kelyhidi].  The present proof combines the method of convex integration and a new ``gluing approximation'' technique.  The convex integration part of the proof relies on the ``Mikado flows'' introduced by [Daneri, Sz\'{e}kelyhidi] and the framework of estimates developed in the author's previous work.  
\end{abstract}
\tableofcontents

\part{Introduction}
In this work, we consider weak solutions to the 3D incompressible Euler equations (posed on a periodic domain), which we write using the Einstein summation convention and in divergence form as
\ali{
\begin{split} \label{eq:eulerEqn}
\pr_t v^\ell + \nab_j (v^j v^\ell) + \nb^\ell p &= 0 \\
\nb_j v^j &= 0 
\end{split}
}
For continuous velocity and pressure fields $v : \R \times \T^3 \to \R^3$, $p : \R \times \T^3 \to \R$, 
being a weak solution to \eqref{eq:eulerEqn} is equivalent to \eqref{eq:eulerEqn} holding in the sense of distributions, or to the equations
\ali{
\fr{d}{dt} \int_\Om v(t,x) \, dx = \int_{\pr \Om} v(t,x) &(v \cdot n) \, dS + \int_{\pr \Om} p(t,x) n \, dS \label{eq:balanceMoment}\\
\int_{\pr \Om} v(t,x) \cdot n(x) \, dS &= 0 \label{eq:balanceMass}
}
holding (as continuous functions of $t \in \R$) for all smooth subregions $\Om \subseteq \T^3$, where $n = n(x)$ is the inward unit normal vector field on the boundary $\pr \Om$, and $dS = dS(x)$ is the surface measure on the boundary.  Equations \eqref{eq:balanceMoment}-\eqref{eq:balanceMass} express the balance of momentum and balance of mass for the portion of an incompressible fluid occupying the region $\Om$, and they are equivalent to \eqref{eq:eulerEqn} holding pointwise for solutions that are continuously differentiable.  More detailed discussions of the concept of a weak solution and its physical meaning can be found in \cite{deLSzeOnsagSurv}.

For $C^1$ solutions to \eqref{eq:eulerEqn} on a periodic domain, one can prove that any solution on a time interval $I$ is uniquely determined by its values $v(t_0, x)$ at a single initial time $t_0 \in I$, and that the total kinetic energy $:= \int_{\T^3} \fr{1}{2} |v|^2(t,x) dx$ is a constant function of time (i.e., $v$ conserves energy).  However, the simple proofs of these results do not apply to weak solutions, and in fact it has been known since the startling discovery of Sheffer \cite{scheff} \hltRed{and later works of \cite{shnNonUnq,shnDiss}} that general distributional solutions to \eqref{eq:eulerEqn} in the class $v \in L_{t,x}^2(\R \times \R^2)$ may fail to be unique, may fail to conserve energy, and may even have compact support \hltRed{or have strictly decreasing total kinetic energy}.  

A longstanding open question has been to determine what degree(s) of regularity must be assumed to guarantee uniqueness or conservation of energy for weak solutions to \eqref{eq:eulerEqn}.  A folklore conjecture is that uniqueness should fail when $v \in C^1$ is replaced by $v \in C_t C_x^\a$ for some $\a < 1$. 
Regarding the conservation of energy, one has the following conjecture, which originates from a 1949 paper by the physicist and chemist Lars Onsager \cite{onsag}:
\begin{conj}[Onsager's Conjecture, Positive Direction] \label{conj:onsPos} If $\a > 1/3$, then (on a periodic domain and a time interval $I$), every weak solution to \eqref{eq:eulerEqn} that satisfies the H\"{o}lder condition
\ali{
|v(t,x + \De x) - v(t,x)| &\leq C |\De x|^\a ,\quad \tx{ for all } t \in I, \De x \in \R^3 \label{eq:holderCondAlpha}
}
for some $C \geq 0$ must satisfy the conservation of energy (i.e. $\int_{\T^3} \fr{1}{2} |v|^2(t,x) \, dx$ is constant in time).
\end{conj}
\begin{conj}[Onsager's Conjecture, Negative Direction] \label{conj:onsag} For every $\a < 1/3$, there exist (periodic) weak solutions to \eqref{eq:eulerEqn} that satisfy \eqref{eq:holderCondAlpha} (in other words, $v \in L_t^\infty C_x^\a$) such that the conservation of energy fails (i.e., $\int_{\T^3} \fr{1}{2} |v|^2(t,x) dx$ fails to be constant in time).
\end{conj}
Onsager's interest in the possibility of Conjecture~\ref{conj:onsag} came from an effort to explain the primary mechanism driving ``anomalous dissipation of energy'' in turbulence in terms of ``energy cascades'' that are modeled by the advective term present in the incompressible Euler equations rather than the viscosity that is present in the Navier-Stokes equations.  He asserted that 
Conjecture~\ref{conj:onsPos} was true to emphasize that, if anomalous dissipation of energy were indeed possible for solutions to the Euler equations, one would have to consider solutions with low regularity.  (Onsager's notion of ``weak solution'' was based on an equivalent definition in terms of Fourier series.) For further discussion of Onsager's conjecture and its significance in turbulence theory we refer to \cite{deLSzeOnsagSurv, eyinkSreen, bardTitiSurv, shvOns}.  \hltRed{See also \cite{chesShvFrCtsmodel} and the references therein for work on model equations for the energy cascade in the Navier-Stokes equations.}

Following the proof of a slightly weaker version of Conjecture~\ref{conj:onsPos} by \cite{eyink}, the positive direction of Onsager's conjecture 
was proven by \cite{CET} using a very short argument.  The sharpest result available, obtained in \cite{ches}, proves conservation of energy for solutions in the class $L_t^3 B_{3, c(\N)}^{1/3} \cap C_t L_x^2$ (on either $\T^n$ or $\R^n$) where $B_{3, c(\N)}^{1/3}$ denotes the closure of $C^\infty_c$ in the Besov space\footnote{Functions in $B_{3, \infty}^{1/3}$ have, roughly speaking, 1/3 of a derivative in the spatial variables in a sense measured by an $L^3$ type norm, rather than the supremum type bound in \eqref{eq:holderCondAlpha}.  See \cite{ches} for a precise definition. } $B_{3, \infty}^{1/3}$.  This result allows for the possibility that the failure of energy conservation in Conjecture~\ref{conj:onsag} may also hold in the endpoint case $\a = 1/3$, and \cite{eyink, ches} provide examples that suggests that fluctuations in kinetic energy should indeed be possible for $\a = 1/3$.  We refer also to \hltRed{\cite{duchonRobert, shvEnSing,IOheat,chesLFShv,robinsonRodrigo2016integral} for extensions} of these results and alternative proofs.

The first results towards the negative direction of Onsager's conjecture came in a breakthrough series of papers by De Lellis and Sz\'{e}kelyhidi \cite{deLSzeCts, deLSzeHoldCts} wherein the authors proved that the failure of energy conservation in Conjecture~\ref{conj:onsag} is possible for solutions in $L_t^\infty C_x^\a$ if $\a < 1/10$.  To achieve this result, the authors adapted a method known as ``convex integration'' -- which has its origins in the work of Nash on constructing paradoxical $C^1$ isometric embeddings \cite{nashC1} -- to the (very different) setting of the incompressible Euler equations \eqref{eq:eulerEqn} (see the survey \cite{deLszehOnsagSurvEur} for a thorough discussion).  Their method 
involves explicitly constructing the velocity field $v$ by adding a series of increasingly high frequency, divergence free waves that are specially designed as perturbations of a family of stationary solutions to 3D Euler known as ``Beltrami flows''.  See \cite{choffDeLSzeCts2d,choff} for extensions to dimension $2$.

In \cite{isettThesis}, the author introduced improvements to the convex integration scheme of \cite{deLSzeCts, deLSzeHoldCts} to establish Conjecture~\ref{conj:onsag} in the range $\a < 1/5$.  (See also \cite{deLSzeBuck,buckDeLIsettSze} for a shorter proof that includes a result on the existence of anomalous dissipation.)  A central theme of the above improvements concerns how to deal with the transport of high frequency waves in the construction by a low frequency velocity field, and the importance of improved estimates for the advective derivative $\pr_t + v \cdot \nb$ as part of improving the regularity of the scheme.  In \cite{isettThesis}, the author also presented a conjectural ``Ideal Case Scenario'' that would imply Onsager's conjecture, and investigated the potential for convex integration to achieve this scenario if the method could be sufficiently improved.

Another direction of research aimed at improvements towards Conjecture~\ref{conj:onsag} in weaker topologies was initiated by the work of Buckmaster \cite{Buckmaster}, who constructed $C_tC_x^{1/5-\ep}$ solutions that fail to conserve energy such that for almost every $t \in \R$ the velocity field has Onsager critical spatial regularity $v(t,\cdot) \in C_x^{1/3-\ep}$.  Using a more involved construction, Buckmaster, De Lellis and Sz\'{e}kelyhidi \cite{buckDeLSzeOnsCrit} improved this result to obtain continuous solutions in the class $v \in L_t^1 C_x^{1/3 - \ep}$ (which means that \eqref{eq:holderCondAlpha} holds with $\a = 1/3 - \ep$ not for all $t \in I$, but with a constant $C(t)$ depending on time such that $\int_{I} |C(t)| dt < \infty$).  A possible target of this direction of research suggested in \cite{buckDeLSzeOnsCrit} could be to obtain solutions in a class such as $v \in L_t^3 C_x^{1/3 - \ep}$, as this class would be borderline with the $L^3$ type spaces $L_t^3 B_{3,c_0(\N)}^{1/3}$ in which one is able to prove energy conservation as in \cite{CET, ches}.  
However, obtaining improvements in the uniform topologies $L_t^\infty C_x^\a$ -- with respect to which Conjecture~\ref{conj:onsag} is formulated above -- appears to be far out of reach of these methods.




Our main theorem is the following, which implies a complete solution to the negative direction of Onsager's conjecture, Conjecture~\ref{conj:onsag}.
\begin{thm} \label{thm:main}  For any $\a < 1/3$, there is a nonzero weak solution to incompressible Euler in the class\footnote{We write $f \in C_{t,x}^\a$ if there exists $C \geq 0$ such that $|f(t + \De t, x + \De x) - f(t,x)| \leq C (|\De t| + |\De x|)^\a$ uniformly in $t,x, \De t, \De x$.}
\ALI{
v \in C_{t,x}^\a(\R \times (\R/\Z)^3), p \in C_{t,x}^{2\a}(\R \times (\R/\Z)^3)
}
such that $v$ is identically $0$ outside a finite time interal.  In particular, the solution $v$ above fails to conserve energy.
\end{thm}

The strategy of proof for Theorem~\ref{thm:main} will be to construct an iteration scheme that establishes the key estimates of the Ideal Case Scenario conjectured in \cite[Sections 10, 13]{isettThesis}.  As in the previous works on Onsager's conjecture described above, a large part of this iteration scheme will be based on the method of convex integration.  We will rely in particular on the framework of estimates developed in \cite{isettThesis}, which had been designed originally to potentially achieve the Ideal Case Scenario.  

One of the main difficulties in convex integration is how to control the interference terms that arise when different high frequency waves in the construction interact with each other through the nonlinear term in the equation.  Following a suggestion of P. Constantin, the idea in \cite{deLSzeCts,deLSzeHoldCts} that turned out to be key for addressing this difficulty was to find a way to design high frequency waves in the construction using ``Beltrami flows'' -- a certain family of stationary solutions to 3D incompressible Euler.  
A version of these Beltrami flows (modified to be well-adapted to the ambient velocity field of the construction) also played a key role in the proof in \cite{isettThesis}, but in that treatment they suffered a deficiency that controlling the high frequency interference terms between Beltrami flows required very sharp cutoffs in time that ultimately limited the regularity of the construction to $1/5$ rather than $1/3$.  

A key idea in the convex integration part of this work, which comes from a recent paper of Daneri and Sz\'{e}kelyhidi \cite{danSze}, is to use Mikado flows as an alternative to Beltrami flows to build the waves in the construction.  Mikado flows (see Section~\ref{sec:mikado} below), are stationary solutions to Euler built by adding together ``straight-pipe'' flows supported in disjoint cylinders that point in multiple directions.  The key difference between Mikado flows and Beltrami flows is that a Mikado flow on its own does not generate unacceptable error terms over a sufficiently long time scale, even when it is made to be well-adapted to the ambient velocity field as was done in \cite{danSze}.  The main difficulty in using Mikado flows to improve the regularity of solutions is that there seems to be no way to control the interference terms that arise when {\it distinct} Mikado flow-based waves interact with each other over the time scale one requires to improve the regularity.  For the h-principle application in \cite{danSze}, it was sufficient to use only a single Mikado flow-based wave, and so no interaction terms were present; however, to produce solutions using an iterative convex integration scheme, one requires an unbounded number of waves, since the time scale during which each high-frequency wave remains coherent shrinks to zero as the frequencies become large.  To improve on the regularity $1/5$, one must be able to control the interference between these waves over a sufficiently long time scale.


Our new method to address this difficulty of distinct wave interference is the following.  Applying convex integration directly would mean generating a sequence of Euler-Reynolds flows $(v, p, R)_{(k)}$ (see Definition~\ref{defn:euReynFlow} below), where the $k$'th error in solving the Euler equation, called $R_{(k)}$, tends to $0$ uniformly and has compact support in time contained in an interval (say, $[0,1]$).  Given $(v,p,R)_{(k)}$, we first find a new Euler-Reynolds flow $(\tilde{v}, \tilde{p}, \tilde{R})_{(k)}$ that is an acceptably small perturbation of the original $(v,p,R)_{(k)}$ obeying essentially the same estimates, such that the new error $\tilde{R}$ decomposes as a sum $\tilde{R} = \sum_I R_I$ such that the $R_I$ are supported in short time intervals that are well-separated from each other.  This technique seems related to the construction in \cite{shnNonUnq}.  After this procedure (which we call a ``gluing approximation technique''), we can apply convex integration to $(\tilde{v}, \tilde{p}, \tilde{R})_{(k)}$ by using a single Mikado-flow based wave to eliminate each $R_I$ up to a small error that is consistent with the Ideal Case Scenario.  
The distinct Mikado flow-based waves will not interact at all in the convex integration due to their supports being well-separated in time.  

The challenge of this technique is to construct the $(\tilde{v}, \tilde{p}, \tilde{R})_{(k)}$ such that all of the desired estimates will hold over the desired time scale (which is relatively long).  Our method for proving the existence and necessary estimates for the new $(\tilde{v}, \tilde{p}, \tilde{R})$ exploits a special structure in the linearization of the Euler and Euler-Reynolds equations to achieve this goal. \hltRed{This important structure is highlighted in more detail towards the end of Section~\ref{sec:goodAntiDiv} below.}

With the confirmation of Onsager's conjecture in the standard formulation of Conjecture~\ref{conj:onsag} now completed by Theorem~\ref{thm:main}, we note that there are several natural generalizations of Conjecture~\ref{conj:onsag} that have been considered in previous work and remain interesting open questions.  
Most immediately, Conjecture~\ref{conj:onsag} should extend as well to dimensions $d \geq 2$ and to general, nonperiodic domains including the whole space.  Our proof extends readily to dimensions $d \geq 3$, but leaves open the case\footnote{The best result recorded for the two-dimensional case is the existence of $(1/10-\ep)$-H\"{o}lder solutions given in \cite{choffDeLSzeCts2d}.  However, the main observations in \cite{choffDeLSzeCts2d} can be used to extend all of the results and arguments based on Beltrami flows on $C_{t,x}^{1/5-\ep}$ and $L_t^1C_x^{1/3-\ep}$ solutions in dimension $3$ (e.g. \cite{isettThesis,isettOh,buckDeLSzeOnsCrit}) to the two-dimensional setting.} $d = 2$ due to the lack of a suitable replacement for Mikado flows.  Our proof also does not produce finite energy solutions\footnote{See \cite{isettOh} for a construction of $C_{t,x}^{1/5-\ep}$ solutions on $\R^3$ with compact support and exposition of the additional issues arising in constructing solutions in the nonperiodic setting.} in $\R^3$ due mainly to analysis related to the gluing approximation technique.  
Further open questions include the extension of Onsager's conjecture to more general fluid equations including active scalar equations and the Boussinesq equation considered in \cite{corFarGanPor,shvConvInt,isettVicol,taoZhangBousCts,taoZhangBous}  and a version of Onsager's conjecture for the steady state Euler equations considered in \cite{luoShv2Dhomog,choffSzeStationary,shvydkoy2017homogeneous}.    

$ $

\noindent {\bf Acknowledgments.}  We thank S.-J. Oh for his discussions during the final preparations of this work.  We also thank the anonymous referees for their recommendations for the final version of this paper.  The work of the author is supported by the National Science Foundation under Award No. DMS-1402370 and DMS-1700312.

\numberwithin{equation}{section} 

\section{Organization of Paper}
The Main Lemma of the paper is stated as Lemma~\ref{lem:mainLem} below after some preliminary general notation introduced in Section~\ref{sec:notation}.  In Section~\ref{sec:subLemmas}, we introduce the three Main Sublemmas of the paper (The Regularization Lemma, the Gluing Approximation Lemma, and the Convex Integration Lemma) and show that they imply the Main Lemma.  Section~\ref{sec:regStep} contains the proof of the Regularization Lemma.  The proof of the Gluing Approximation Lemma occupies Sections~\ref{sec:outlineGlue}-\ref{sec:proofGlueLem}.  The proof of the Convex Integration Lemma occupies Sections~\ref{sec:mikado}-\ref{sec:concludingProof}.  The proof of Theorem~\ref{thm:main} using the Main Lemma is then given in Section~\ref{proofMainThm}.  The Appendix provides proofs or statements of analytical facts that were used in the proofs of the Main Sublemmas of the paper.



\section{Notation and Preliminaries} \label{sec:notation}
If $x \in \R$, we will write $(x)_+ = \max \{ x, 0 \}$.  We will make use of the following ``counting inequality'', which is stated as Lemma 17.1 in \cite{isett} and can be shown by induction on $m$
\ali{
\sum_{i=1}^m(x_i - y)_+ &\leq (\sum_{i=1}^m x_i - y)_+, \quad \tx{ for all }  x_1, x_2, \ldots, x_m, y \geq 0 \label{ineq:counting}
}
We will use the Einstein summation convention to sum over indices that are repeated; for example $\nb_j v^j = \sum_{j=1}^3 \nb_j v^j$ is the divergence of a vector field $v$.  Indices are raised or lowered to distinguish covariant and contravariant indices as in the conventions of invariant index notation.  The summation convention will be used only to pair a raised index and a lowered index.  
We will write $\SS$ to denote the subspace of $\R^3 \otimes \R^3$ consisting of symmetric $(2,0)$ tensors.

For partial derivatives, we will distinguish between multi-indices and first order indices by writing a multi-index in vector form.  For instance, if $\va = (a_1, a_2, a_3)$ is a multi-index of order $|\va| = 3$, then $\nb_{\va} = \nb_{a_1} \nb_{a_2} \nb_{a_3}$ is the corresponding third-order partial derivative.  In contrast, $\nb_a$ without a vector symbol denotes the first order, a'th partial derivative.  The full derivative of a tensor will be denoted using a superscript; for example $\nb^k f$ refers to the full, $k$th derivative of a function $f$, which takes values in the $k$-fold tensor product of $(\R^3)^*$.

In what follows we will refer to functions $f : \R \times \T^3 \to \R$ or $f : \T^3 \to \R$, but the discussion in this section generalizes immediately to vector fields and tensor fields taking values in $\R^n$.

For functions $f : \R \times \T^3 \to \R$, we will use the following notation to describe their time support
\ali{
\suppt f &:= \{ t \in \R ~:~ \supp f \cap \{ t \} \times \T^3 \neq \emptyset \} \notag
}
For simplicity, we refer to a function $f : \R \times \T^3 \to \R$ of space and time as {\bf smooth} if all of its spatial derivatives are continuous on $\R \times \T^3$ (which implies $f \in \bigcap_{k \geq 0} C_t C_x^k(\R \times \T^3)$).  We will write $C^\infty$ or $C^\infty(\R \times \T^3) = \bigcap_{k \geq 0} C^k(\R \times \T^3)$ to refer to the usual class of infinitely differentiable functions.  
The distinction between the two can be safely neglected in reading the argument since all the functions involved that are required in the course of the proof to be ``smooth'' will in fact be $C^\infty$; however, the higher differentiability in time will not be as important.

If $f : \T^3 \to \R$, we will write $u = \De^{-1} f$ to mean the unique function $u : \T^3 \to \R$ solving
\ALI{
\De u = (1 - \Pi_0) f, \qquad \int_{\T^3} u(x) dx = 0
}
where $\Pi_0 f = |\T^3|^{-1} \int_{\T^3} f(x) dx$ is the average value of $f$.

Given a subset $S \subseteq \R$ and $\tau \geq 0$, we will denote its $\tau$-neighborhood in $\R$ by
\ali{
N(S; \tau) &:= \{ t + t' ~:~  t \in S, |t'| \leq \tau \} \notag
}If $f : \T^3 \to \R$ is continuous and $0 < \a < 1$, we denote its homogeneous H\"{o}lder seminorm by
\ali{
[f]_\a = \cda{f} &:= \sup_{x \in \T^3} \sup_{h \in \R^3 \setminus \{0 \}} \fr{|f(x+h) - f(x)|}{|h|^\a} \label{defn:HoldSmnrm}
}
For integers $k \geq 0$, a function belongs to $f \in C^{k,\a}$ if $f \in C^k(\T^3)$ and $\cda{\nab^k f}$ is finite.  
The mean value theorem leads to the following interpolation inequality
\begin{prop} If $f : \T^3 \to \R$ is $C^1$ and $0 < \a < 1$, then
\ali{
\cda{f} &\lsm_\a \co{\nab f}^\a \co{f}^{(1-\a)} \label{ineq:holdInterp}
}
\end{prop}

\part{The Main Lemma and Sublemmas}
To state the Main Lemma, we first recall the concept of an Euler-Reynolds flow from\footnote{In \cite{deLSzeCts} the definition is given in an equivalent form where $R^{j\ell}$ is required to have $0$ trace $\de_{j\ell} R^{j\ell} = 0$ pointwise.} \cite{deLSzeCts} and define the notion of frequency-energy levels that will be used in our paper.  
\begin{defn}\label{defn:euReynFlow}  A vector field $v^\ell : \R \times \T^3 \to \R^3$, function $p : \R \times \T^3 \to \R$ and symmetric tensor field $R^{j\ell} : \R \times \T^3 \to \SS$ satisfy the {\bf Euler Reynolds Equations} if
\ALI{
\pr_t v^\ell + \nb_j(v^j v^\ell) + \nb^\ell p &=\nb_j R^{j\ell} \\
\nb_j v^j &= 0
}
on $\R \times \T^3$.  Any solution to the Euler-Reynolds equations $(v,p,R)$ is called an {\bf Euler-Reynolds Flow}.  The symmetric tensor field $R^{j\ell}$ is called the {\bf stress} tensor.
\end{defn}

Our notion of frequency energy levels will be based on the one introduced in \cite{isett}, but simpler in that we do not assume control over the pressure gradient or over the advective derivative of $R$.  This simplification ultimately arises due to a special feature of the gluing approximation technique summarized in Lemma~\ref{lem:glueLem} below, which is that the stress tensor $\tilde{R}$ arising from the gluing approximation technique turns out to exhibit a suitable estimate on its advective derivative even if the starting Euler-Reynolds flow does not satisfy such a bound.  This feature of the argument will also allow us to circumvent the use of the mollification on the flow technique introduced in \cite[Section 18]{isett}, which would otherwise have been needed within the convex integration part of the proof.

\begin{defn}\label{defn:frenlvls}  Let $(v, p, R)$ be a solution of the Euler-Reynolds equation, $\Xi \geq 3$ and $e_v \geq e_R \geq 0$ be non-negative numbers.  We say that $(v, p, R)$ have {\bf frequency-energy levels} bounded by $(\Xi, e_v, e_R)$ to order $L$ in $C^0$ if \hltRed{their spatial derivatives $\nb^k v$ and $\nb^k R$ of order $k$ are continuous for all $k \leq L$ }
and the following estimates hold
\ali{
\co{\nab^{k} v} &\leq \Xi^{k} e_v^{1/2}, \quad \tx{ for all } 1 \leq k \leq L \label{eq:frEnVelocbds} \\
\co{\nab^{k}R} &\leq \Xi^{k} e_R, \quad \tx{ for all } 0 \leq k \leq L
}
\end{defn}
The Main Lemma of our paper states the following
\begin{lem}[The Main Lemma] \label{lem:mainLem}  Let $L = 3$ and $\eta > 0$.  There exists a constant $C$ depending only on $\eta$ such that the following holds.   Let $(v, p, R)$ be any solution of the Euler-Reynolds equation with frequency-energy levels bounded by $(\Xi, e_v, e_R)$ to order $L$ in $C^0$ and let $J$ be an open subinterval of $\R$ such that (recalling the notation of Section~\ref{sec:notation})
\ALI{
\suppt v \cup \suppt R &\subseteq J.
}
Define the parameter $\hxi = \Xi (e_v/e_R)^{1/2}$.  Let $N$ be any positive number obeying the conditions
\ali{
N \geq \max \{ \, \Xi^\eta, (e_v/e_R)^{1/2} \, \}  \label{eq:Nrestrict}
}
Then there exists a solution $(v_1, p_1, R_1)$ of Euler-Reynolds with frequency-energy levels bounded by
\ali{
(\Xi', e_v', e_R') &= \left(C N \Xi, \lhxi e_R, \plhxi^{5/2} \fr{e_v^{1/2} e_R^{1/2}}{N} \right) 
}
such that 
\ali{
\suppt v_1 \cup \suppt R_1 &\subseteq N(J; \Xi^{-1} e_v^{-1/2}) \label{ct:suppCdn}
}
and such that the correction $V = v_1 - v$ obeys the bounds
\ali{
\co{V} &\leq C \plhxi^{1/2} e_R^{1/2} \label{ineq:coBdV} \\
\co{\nab V} &\leq C N \Xi \plhxi^{1/2} e_R^{1/2}, \label{ineq:coNbvCorrect}
}
\end{lem}

Lemma~\ref{lem:mainLem} will follow from a combination of three Main Sublemmas that we will describe in the following Section~\ref{sec:subLemmas}.  The proof of Lemma~\ref{lem:mainLem} assuming these sublemmas will be included at the end of Section~\ref{sec:subLemmas}.

\section{The Main Sublemmas} \label{sec:subLemmas}
Here we state the three sublemmas that together will imply our Main Lemma, Lemma~\ref{lem:mainLem}.

\begin{lem}[The Regularization Lemma] \label{lem:regSublem}  There is an absolute constant $C_0$ such that the following holds.  Let $(v_0, p_0,R_0)$ be an Euler-Reynolds flow with frequency-energy levels bounded by $(\Xi, e_v, e_R)$ to order $3$ in $C^0$ such that $\suppt v_0 \cup \suppt R_0 \subseteq J$.  Define $\nhat := (e_v/e_R)^{1/2}$.  Then there exists an Euler-Reynolds flow $(v, p, R)$ such that $\suppt v \cup \suppt R \subseteq J$ that obeys the estimates
\ali{
\co{ \nab^{k} v } &\leq C_0 \nhat^{(k-3)_+} \Xi^{k} e_v^{1/2}, \quad k = 1, \ldots, 5 \label{eq:v1Est} \\
\co{ \nab^{k} R } &\leq C_0 \nhat^{(k-2)_+} \Xi^{k} e_R, \quad k = 0, \ldots, 5 \label{eq:R1Est} \\
\co{ v - v_0 } &\leq C_0 e_R^{1/2} \label{eq:incBdReg}
}
Furthermore, one can arrange that $v, R \in \bigcap_{k \geq 0} C_t C_x^k$ are smooth. 
\end{lem}

\begin{lem}[The Gluing Approximation] \label{lem:glueLem}  For any $C_0 \geq 1$ there exist positive constants $C_1 \geq 1$ and $\de_0 \in(0, 1/25)$ such that the following holds.  Let $(v, p, R)$ be a smooth Euler-Reynolds flow that satisfies the estimates  \eqref{eq:v1Est}-\eqref{eq:R1Est} for $C_0$ and $\suppt v \cup \suppt R \subseteq J$.  Define $\hxi := \nhat \Xi = \Xi(e_v/e_R)^{1/2}$.  Then for any $0 < \de \leq \de_0$ there exist: a constant $C_\de \geq 1$, a constant $\th > 0$, a sequence of times $\{t(I)\}_{I \in \Z} \subseteq \R$ and an Euler-Reynolds flow $(\tilde{v}, \tilde{p}, \tilde{R})$, $\widetilde{R} = \sum_{I \in \Z} R_I$, that satisfy the following support restrictions
\ali{
\suppt \tilde{v} \cup \suppt \widetilde{R} &\subseteq N(J; \fr{1}{3}\Xi^{-1} e_v^{-1/2}) \label{ct:growth} \\
2^{-1} \de \plhxi^{-2} \Xi^{-1} e_v^{-1/2} \leq \th &\leq  \de \plhxi^{-2} \Xi^{-1} e_v^{-1/2} \label{ineq:thBound} \\
\suppt R_I &\subseteq \left[t(I) - \fr{\th}{2}, t(I) + \fr{\th}{2}\right] \label{eq:supptRI} \\
\bigcup_I \bigcup_{I' \neq I} [t(I) - \th, t(I) + \th] &\cap [t(I') - \th, t(I') + \th] = \emptyset \label{ct:disjointness}
}
and the following estimates
\ali{
\co{ \tilde{v} - v } &\leq C_1 e_R^{1/2}  \\
\co{ \nab^{k} \tilde{v} } &\leq C_1 \Xi^{k} e_v^{1/2}, \quad k = 1, 2, 3 \label{eq:newVelocBdGlue} \\
\sup_I \co{ \nab^{k} R_I  } &\leq C_\de \nhat^{(k-2)_+} \Xi^{k} \lhxi  e_R, \quad k = 0, 1,2, 3 \label{eq:newRIbdGlue} \\
\sup_I \co{ \nab^{k} (\pr_t + \tilde{v} \cdot \nab) R_I } &\leq C_\de  \plhxi^3 \Xi e_v^{1/2} \Xi^{k} e_R, \quad k = 0, 1, 2. \label{eq:newDtRIbdGlue}
}
\end{lem}

\begin{lem}[The Convex Integration Lemma] \label{lem:convexInt} There exists an absolute constant $b_0$ such that for any $C_1, C_\de \geq 1$ and $\de, \eta > 0$ there is a constant $\tilde{C} = \tilde{C}_{\eta, \de, C_1, C_\de}$ for which the following holds.  Suppose $J$ is a subinterval of $\R$ and $(v, p, R)$ is an Euler-Reynolds flow, $R = \sum_I R_I$, that satisfy the conclusions \eqref{ct:growth}-\eqref{ct:disjointness} and \eqref{eq:newVelocBdGlue}-\eqref{eq:newRIbdGlue} of Lemma~\ref{lem:glueLem} (with $(\tilde{v}, \widetilde{R})$ replaced by $(v, R)$ ) for some $(\Xi, e_v, e_R)$, some $\th  > 0$ and some sequence of times $\{ t(I)\}_{I \in \Z} \subseteq \R$.  Suppose also that
\ali{
|\th| \co{ \nab v } &\leq b_0. \label{eq:b0bd}
}
Let $N \geq \max \{ \Xi^\eta, (e_v/e_R)^{1/2} \}$.  Then there is an Euler-Reynolds flow $(v_1, p_1, R_1)$ with frequency-energy levels in the sense of Definition~\ref{defn:frenlvls} bounded by
\ali{
(\Xi', e_v', e_R') &= \left(\tilde{C} N \Xi, \lhxi e_R, \plhxi^{5/2} \fr{e_v^{1/2} e_R^{1/2}}{N} \right) \label{eq:newFrEnLvls}
}
such that
\ali{
\suppt v_1 \cup \suppt R_1 &\subseteq N(J; \Xi^{-1} e_v^{-1/2}) \label{ct:suppCvxInt} \\
\co{ v_1 - v } &\leq \tilde{C} \plhxi^{1/2} e_R^{1/2} \label{eq:cobdcorrect}
}
\end{lem}
\begin{proof}[Proof of Lemma~\ref{lem:mainLem}]
It is now straightforward to show that Lemmas~\ref{lem:regSublem}, \ref{lem:glueLem} and \ref{lem:convexInt} together imply Lemma~\ref{lem:mainLem}.  Indeed, suppose that $(v_0, p_0, R_0)$ and $J$ are the Euler-Reynolds flow and time interval in the hypotheses of Lemma~\ref{lem:mainLem} and let $N$ and $\eta$ be the parameters given in the Lemma.  Apply Lemma~\ref{lem:regSublem} to this $(v_0, p_0, R_0)$ to obtain $(v_{01}, p_{01}, R_{01})$ obeying the conclusions of that Lemma for the constant $C_0$.  Let $C_1, \de_0$ be the constants in Lemma~\ref{lem:glueLem} associated to $C_0$.  Choose $\de > 0$ such that $\de \leq \de_0$ and $C_1 \de \leq b_0$, where $b_0$ is the absolute constant in Lemma~\ref{lem:convexInt}.  Apply Lemma~\ref{lem:glueLem} with this value of $\de$ to the $(v_{01}, p_{01}, R_{01})$ above to obtain an Euler-Reynolds flow $(\tilde{v}, \tilde{p}, \tilde{R})$ together with parameters $\th, C_\de$ and $\{ t(I) \}_{I \in \Z}$ that satisfy the conclusions Lemma~\ref{lem:glueLem}.  Observe also that \eqref{ineq:thBound} and \eqref{eq:newVelocBdGlue} imply
\ALI{
|\th| \co{ \nab \tilde{v} } \leq C_1 \de \leq b_0.
}
We may therefore apply the Convex Integration Lemma~\ref{lem:convexInt} to $(\tilde{v}, \tilde{p}, \tilde{R})$ with the parameter $N$ to obtain an Euler-Reynolds flow $(v_1, p_1, R_1)$ and a constant $\tilde{C}$ satisfying the conclusions of that Lemma.  Note that this constant $\tilde{C}$ depends only on $\eta$ since $C_0$ is an absolute constant and $C_1$ and therefore $\de$ and $C_\de$ depend only on $C_0$.  The correction $V = v_1 - v_0$ induced by this combination of Lemmas satisfies (using Definition~\ref{defn:frenlvls} in the last line)
\ali{
V = (v_1 - v_0) &= (v_1 - \tilde{v}) + (\tilde{v} - v_{01}) + (v_{01} - v_0) \notag \\
\co{ V } &\leq A_0 C_\de \plhxi^{1/2} e_R^{1/2} + C_1 e_R^{1/2} + C_0 e_R^{1/2} \notag \\
\co{ \nab V } &\leq \co{ \nab v_1} + \co{\nab v_0} \notag \\
&\leq \tilde{C} N \Xi \plhxi^{1/2} e_R^{1/2} + \Xi e_v^{1/2} \label{eq:Vc1BdLemLem}
}
The $C^0$ term is bounded by $\check{C}_0 \plhxi^{1/2} e_R^{1/2}$ with $\check{C}_0$ some constant, as stated in Lemma~\ref{lem:mainLem}.  The lower bound \eqref{eq:Nrestrict} on $N$ implies $e_v^{1/2} \leq N e_R^{1/2}$, so the right hand side of \eqref{eq:Vc1BdLemLem} is bounded by $C N \Xi  \plhxi^{1/2} e_R^{1/2}$ for some constant $C$ (depending on $\eta$) that we can take to be the one whose existence is asserted in Lemma~\ref{lem:mainLem}.  Since the containment \eqref{ct:suppCdn} follows from \eqref{ct:suppCvxInt}, we have proved all the conclusions of Lemma~\ref{lem:mainLem} for the above $(v_1, p_1, R_1)$.
\end{proof}

\part{The Gluing Sublemmas} 
In this part of the paper, we prove the two lemmas, Lemmas~\ref{lem:regSublem} and \ref{lem:glueLem}, related to the gluing approximation procedure.   
\section{The Regularization Step} \label{sec:regStep}
Here we prove the Regularization Lemma~\ref{lem:regSublem}.  This Lemma is important for the gluing approximation procedure, as the Gluing Approximation Lemma~\ref{lem:glueLem} loses spatial regularity.  We note that the loss of spatial regularity in the Gluing Approximation Lemma~\ref{lem:glueLem} is inherent to the gluing argument and is distinct from the separate loss of spatial regularity that occurs during the Convex Integration Lemma~\ref{lem:convexInt}, which is addressed by a different mollification technique.

We start by recalling the following quadratic commutator estimate.  Proofs of this statement can be found in \cite[Lemma 1]{deLSzeC1iso} or \cite[Proposition 5.3]{isett2}.
\begin{prop}\label{prop:mollifyCommute} Let $\eta \in C^\infty(\R^n)$ satisy $\int_{\R^n} \eta(h) dh = 1$, $(1 + |h|^2) \eta(h) \in L^1(\R^n)$, and define $\eta_\ep(h) := \ep^{-n} \eta(h/\ep)$.  If $f, g \in C^1(\R^n)$, then for all $0 \leq k < \infty$ we have
\ali{
\| \nab^{k} ( \eta_\ep \ast ( f g ) - (\eta_\ep \ast f) (\eta_\ep \ast g) ) \|_{C^0(\R^n)} &\lsm_{k} \ep^{2-k} \| \nab f \|_{C^0(\R^n)} \| \nab g \|_{C^0(\R^n)} \label{ineq:mollifyCommute} 
}
\end{prop}
With Proposition~\ref{prop:mollifyCommute} in hand, the proof of Lemma~\ref{lem:glueLem} goes as follows.  Given $(v_0,p_0,R_0)$ as in Lemma~\ref{lem:regSublem}, set $v = \eta_\ep \ast v_0$ and $p = \eta_\ep \ast p_0$, then $\ep = \hxi^{-1} = \nhat^{-1} \Xi^{-1} = (e_v/e_R)^{-1/2} \Xi^{-1}$.  By writing the Euler-Reynolds equations in divergence form
\ALI{
\pr_t v_0^\ell + \nb_j(v_0^j v_0^\ell) + \nb^\ell p_0 &= \nb_j R_0^{j\ell}, \quad \nb_j v_0^j = 0
}
we see that $(v,p, R)$ satisfies the Euler-Reynolds equations with a new stress given by
\ali{
R^{j\ell} &= [ \eta_\ep \ast v_0^j \eta_\ep \ast v_0^\ell - \eta_\ep \ast (v_0^j v_0^\ell) ] + \eta_\ep \ast R_0^{jl} \label{eq:newMollifyStress}
}
The term $\eta_\ep \ast R^{jl}$ is easily seen to obey the desired estimate \eqref{eq:R1Est}.  For example, 
\ALI{
\co{ \nab^{k + 2}\eta_\ep \ast R_0^{jl} } = \co{ \nab^{k}\eta_\ep \ast \nab^{2}R_0^{jl} } &\lsm_{k} \hxi^{k} \Xi^2 e_R = \nhat^{k} \Xi^{k + 2} e_R, \quad k \geq 0,
}
while for $k \leq 2$ we have $\co{ \nab^{k} \eta_\ep \ast R_0^{jl} } = \co{ \eta_\ep \ast \nab^{k} R_0^{jl} } \leq \co{\nab^{k} R_0^{jl}} $.

Let $Q_\ep^{j\ell}$ denote the commutator term in \eqref{eq:newMollifyStress}.  By \eqref{ineq:mollifyCommute}, we have 
\ALI{ \| Q_\ep \|_{C^0} \lsm \ep^2 \| \nab v_0 \|_{C^0}^2 \leq \hxi^{-2} \Xi^2 e_v = e_R.}
Also, if $\nb_{b_1}\nb_{b_2}$ is any second partial derivative operator, we have
\ALI{
\nb_{b_1}\nb_{b_2} Q_\ep^{jl} &= \eta_\ep \ast (\nb_{b_1}\nb_{b_2} v_0^j ) \eta_\ep \ast v_0^l - \eta_\ep \ast (\nb_{b_1}\nb_{b_2} v_0^j v_0^l)  \\
&+ \eta_\ep \ast ( \nb_{b_1} v_0^j) \eta_\ep \ast( \nb_{b_2} v_0^l ) - \eta_\ep \ast (\nb_{b_1} v_0^j \nb_{b_2} v_0^l ) + \tx{ similar terms}
}
From this expression and \eqref{ineq:mollifyCommute}, our control on $\co{\nab^3 v_0}$ gives us the desired bound \eqref{eq:R1Est}
\ALI{
\co{\nab^{k + 2} Q_\ep^{jl} } &\lsm_{k} \ep^{2 - k} ( \co{ \nab^3 v_0 } \co{ \nab v_0 } + \co{\nab^2 v_0 }^2 ) \\
&\lsm_{k} \nhat^{k} \Xi^{k + 2} e_R
}
The desired estimate for $\co{ \nab Q_\ep }$ follows by interpolating $\co{ \nab Q_\ep} \lsm \co{ Q_\ep }^{1/2} \co{ \nab^{2} Q_\ep }^{1/2}$.   

The proof of \eqref{eq:v1Est} follows similarly to the bounds for $\eta_\ep \ast R_0^{j\ell}$ above.  Finally, \eqref{eq:incBdReg} follows from $\co{ v_0 - \eta_\ep \ast v_0 } \lsm \ep \co{ \nab v_0 } \leq e_R^{1/2}$. 

\section{Proof of Gluing Approximation: Outline} \label{sec:outlineGlue}
Sections~\ref{sec:glueConstruct}-\ref{sec:gluingProof} below are devoted to the proof of the Gluing Approximation Lemma~\ref{lem:glueLem}.  The construction of the new Euler-Reynolds flow $(\tilde{v}, \tilde{p}, \tilde{R})$ is described in Sections~\ref{sec:glueConstruct}-\ref{sec:exist}.  Section~\ref{sec:gluingPrelims} contains preliminary facts that will be used in the estimates in the proof of Lemma~\ref{lem:glueLem}.  The main estimates of the gluing construction are carried out over Sections~\ref{sec:pressEsts}-\ref{sec:rhoIests}.  These estimates are then applied in Section~\ref{sec:proofGlueLem} to establish that the $(\tilde{v}, \tilde{p}, \tilde{R})$ defined in Sections~\ref{sec:glueConstruct}-\ref{sec:exist} does indeed satisfy the estimates claimed in Lemma~\ref{lem:glueLem}.

\section{The Gluing Construction} \label{sec:glueConstruct}
We now begin the proof of the Gluing Approximation Lemma~\ref{lem:glueLem}.  In this construction, we will use the notation $\lsm$ for constants that are allowed to depend on the $C_0$ of Lemma~\ref{lem:regSublem}, while we will write $\lsm_\de$ if the constant depends on the $\de$ given in the Lemma~\ref{lem:glueLem}.  

  Let $(v, p, R)$ be the given Euler-Reynolds flow.  We desire a new Euler-Reynolds flow $(\tilde{v}, \tilde{p}, \tilde{R})$ such that the new stress $\tilde{R}$ can be decomposed as a sum $\tilde{R} = \sum_I R_I$ in which each piece is localized in time to a time interval that is smaller than the natural time scale of the construction  $| \suppt R_I | \leq |\th| \lsm \Xi^{-1} e_v^{-1/2}$.  It is also important that the gap between these supports is comparable to the natural time scale of the construction: $\tx{dist} ( \suppt R_I, \suppt R_{I'} ) \geq \Xi^{-1}e_v^{-1/2}$ (although we will ultimately miss this goal by a small margin.)
	
	Since we need $\tilde{R} = 0$ to vanish in the gaps between the interals supporting each $R_I$, our new solution $\tilde{v}$ has to solve the Euler equations within those gaps.  Thus, we construct $\tilde{v}$ by solving the Euler equations on many time intervals of length $8 \th$, and gluing the solutions together with the following partition of unity construction.  
	
	We introduce the velocity increment $y^\ell$ and the pressure increment $\bar{p}$, so that the new solution will be given by $\tilde{v}^\ell = v^\ell + y^\ell$, $\tilde{p} = p + \bar{p}$.  The pair $(y^\ell, \bar{p})$ must satisfy the equation
\ali{
\begin{split} \label{eq:eulerIncEqn}
\pr_t y^\ell + v^j \nb_j y^\ell + y^j \nb_j v^\ell + \nb_j(y^j y^\ell) + \nb^\ell \bar{p} &= \nb_j \tilde{R}^{j\ell} - \nb_j R^{j\ell} \\
\nb_j y^j &= 0
\end{split}
}
For each $I \in \Z$ set $t_0(I) = 8 \th \cdot I $.  Define $(u_I^\ell, p_I)$ to be the unique classical solution to the incompressible Euler that satisfies the initial condition $u_I^\ell(t_0(I), x) = v^l(t_0(I), x)$, where $v^l$ is our given Euler-Reynolds flow.  It is well-known (see Theorem~\ref{thm:locExistEuler} below) that, for any $\a > 0$, if $\th$ times the $C^{1,\a}$ norm of the initial data is sufficiently small, then $u_I^\ell$ is well-defined and remains smooth on the interval $[t_0(I) - 8 \th, t_0(I) + 8 \th]$.  (Our $\th$ will be larger than the reciprocal of $C^{1,\a}$ norm of the initial data, but we will nonetheless be able to prove existence.)

Define $y_I^\ell$ and $\bar{p}_I$ so that $u_I^\ell = v^\ell + y_I^\ell$ and $p_I = p + \bar{p}_I$.  Then $(y_I^\ell, \bar{p}_I)$ solve
\ali{
\begin{split}
\pr_t y_I^\ell + v^j \nb_j y_I^\ell + y_I^j \nb_j v^\ell + \nb_j ( y_I^j y_I^\ell) + \nb^\ell \bar{p}_I &= - \nb_j R^{j\ell} \\
\nb_j y_I^j &= 0 \\
y_I^\ell(t(I), x) &= 0 
\end{split} \label{eq:eulerIncIeqn}
}
Now choose a partition of unity in time $(\eta_I(t))_{I \in \Z}$ with the following properties:
\ali{
\begin{split} \label{eq:partitionProperties}
\suppt \eta_I(t) &\subseteq \left[ t_0(I) - \fr{9 \th}{2}, t_0(I) + \fr{9 \th}{2}\right] \\
\eta_I(t) &= 1 \quad \tx{ if } t \in \left[t_0(I) - \fr{7 \th}{2}, t_0(I) + \fr{7 \th}{2} \right] \\
\suppt \eta_I \cap \suppt \eta_{I'} &= \emptyset \quad \tx{if } |I - I'| > 1 \\
\sum_I \eta_I(t) &= 1 \quad \tx{ for all } t \in \R 
\end{split}
\\
\sup_I \sup_t \left| \fr{d^k}{dt^k} \eta_I(t) \right| &\lsm |\th|^{-k}, \quad k = 0, 1, 2 \label{eq:partitionBound}
}
An easy way to construct such a partition of unity is to start with the rough partition of unity given by characteristic functions $\chi_I(t) = 1_{(t_0(I) - 4 \th, t_0(I) + 4 \th]}(t)$, and then regularize it by mollification to obtain $\eta_I = \psi_\th \ast \chi_I$, with $\psi_\th(t)$ a smooth mollifier, $\int_\R \psi_\th(t) dt = 1$, with support in $|t| < \fr{\th}{2}$.  Now set 
\ali{
y^\ell = \sum_I \eta_I y_I^\ell, \quad \bar{p} = \sum_I \eta_I \bar{p}_I \notag
}
From \eqref{eq:eulerIncIeqn} and \eqref{eq:partitionProperties}, one sees that $y^\ell$ is divergence free and satisfies
\ali{
\begin{split} \label{eq:gluedEulerEqn}
\pr_t y^\ell + v^j \nb_j y^\ell + y^j \nb_j v^\ell + \nb_j(y^j y^\ell) + \nab^\ell \bar{p} &= \sum_I \eta_I'(t) y_I^\ell + \sum_I \eta_I \eta_{I+1} \nb_j(y_I^j y_{I+1}^\ell + y_{I+1}^j y_I^\ell ) \\
&+ \sum_I ( \eta_I^2 - \eta_I) \nb_j ( y_I^j y_I^\ell ) - \nb_j R^{j\ell}
\end{split}
}
For each $I$, we will choose a {\bf symmetric anti-divergence} for $y_I^\ell$, by which we mean a symmetric tensor $r_I^{j\ell}$ verifying
\ali{
\nb_j r_I^{j\ell} &= y_I^\ell \label{eq:symDivEqnI}
}
Now set $t(I) = t_0(I) + 4 \th = (8 \th) \cdot I + 4 \th$, 
and define
\ali{
\begin{split}
R_I^{j\ell} &= 1_{[t(I) - \th, t(I)+\th]}\, \eta_I'(t) (r_I^{jl} - r_{I+1}^{j\ell}) + \eta_I \eta_{I+1} ( y_I^j y_{I+1}^\ell + y_{I+1}^j y_I^\ell ) \\
&- \eta_I \eta_{I+1} ( y_I^j y_I^\ell + y_{I+1}^j y_{I+1}^\ell )  \label{eq:RIformula}
\end{split}
}
Comparing with \eqref{eq:gluedEulerEqn} and using the identities $\eta_I'(t) = 1_{[t(I) + \th, t(I)-\th]} \eta_I'(t) - 1_{[t(I-1) + \th, t(I-1)-\th]} \eta_{I - 1}'(t)$ and $\eta_I(1-\eta_I) = \eta_I \eta_{I+1} + \eta_{I-1} \eta_I$, we see that our goal equation \eqref{eq:eulerIncEqn} holds for $y^\ell$ and $\tilde{R}^{j\ell} = \sum_I R_I^{j\ell}$ provided that \eqref{eq:symDivEqnI} holds for all $I$.


As for $\th$, we choose $\th = \de \plhxi^{-2} \Xi^{-1} e_v^{-1/2}$ so that equality holds for the upper bound in \eqref{ineq:thBound}.  However, the lower bound, which strengthens condition \eqref{ct:disjointness}, will be important in the proof of Lemma~\ref{lem:convexInt}).

It is now apparent that the containment \eqref{ct:growth} holds using the condition $\de < \de_0 < 1/25$, as $y_I = 0$ for all $I$ such that $\suppt \eta_I$ intersects the complement of $N(J; \Xi^{-1} e_v^{-1/2} )$.  This fact follows from uniqueness of classical solutions to incompressible Euler (see Theorem~\ref{thm:locExistEuler} below) and the assumption that $\suppt R \subseteq J$, which implies that $v^\ell$ already solves the Euler equations outside of $J \times \T^3$.   Also, the conditions \eqref{ineq:thBound}-\eqref{ct:disjointness} are apparent from the construction above once we show that the $y_I$ and $r_I$ are well-defined on the supports of the $\eta_I$.

Our construction is now complete except that we have not proven that $y_I^\ell$ exists on the required interval for the above formulas to be well-defined, and we have not defined or proven the existence of the tensors $r_I^{j\ell}$ verifying \eqref{eq:symDivEqnI}.  
We discuss the construction of $r_I^{j\ell}$ in the following Section~\ref{sec:goodAntiDiv} below.  We will then address the existence of $y_I^\ell$ in Section~\ref{sec:exist} below.

\section{Constructing a Good Anti-Divergence } \label{sec:goodAntiDiv}
In this Section, we discuss our construction of a symmetric anti-divergence for the $y_I^\ell$ of the gluing construction in Section~\ref{sec:glueConstruct}.  The construction will involve an operator $\RR^{j\ell}$ that we define as follows.  Recall that any smooth vector field on $\T^n$ has a ``Helmholtz decomposition'' such that
\ali{
U^\ell &= \HH U^\ell + \Pi_0 U^\ell + \nb^\ell \De^{-1}[ \nb_b U^b] \label{eq:helmholtz} \\ 
\Pi_0 U^\ell &:= \fr{1}{|\T^n|} \int_{\T^n} U^\ell dx \notag \\
\nb_\ell \HH U^\ell = 0, &\qquad \int_{\T^n} \HH U^\ell dx = 0 \notag
}
Given a smooth vector field $U$ on $\T^n$,  $\RR^{j\ell}[U]$ is the smooth, symmetric $(2,0)$ tensor on $\T^n$ defined by
\ali{
\RR^{j\ell} [ U ] &= \De^{-1} (\nb^\ell \HH U^j + \nb^j \HH U^\ell ) + \de^{j\ell} \De^{-1}[\nb_b U^b] \label{eq:RRformula}
}
Then $\RR^{j\ell}$ is an operator (of order $-1$) that satisfies, for all $U \in C^\infty(\T^n)$,
\ali{
\nb_j \RR^{j\ell}[U] &= (1 - \Pi_0) U^\ell  \label{eq:invertModIntegral} \\
\int_{\T^n} \RR^{j\ell}[U](x) dx &= 0 \notag
}
In particular, $\RR^{j\ell}$ inverts the divergence operator when restricted to integral $0$ vector fields.  (Other inverses for the divergence can be chosen that would be equivalent for our purposes.  This choice has the advantage of having a simple formula.)

Using the operator $\RR$ above, one can find a solution to \eqref{eq:symDivEqnI} by taking $r_I^{j\ell}$ equal to $\RR^{j\ell}[y_I]$.  Indeed, by the conservation of momentum for Euler and Euler-Reynolds, $y_I$ has integral $0$, so we have $\nb_j \RR^{j\ell}[y_I] = y_I$.  However, the best estimate one has in general for this solution is $\co{ \RR^{j\ell}[y_I] } \lsm \co{ y_I }$, which leads to an estimate for the $R_I$ in \eqref{eq:RIformula} of the form 
\ali{
\co{R_I} = \co{ \eta'_I(t) r_I^{jl} + \ldots } &\lsm |\th|^{-1} \co{y_I} + \ldots \label{eq:easyRIbound}
}
Let us assume now that we are able to prove a bound of the order $\co{y_I} \lsm e_R^{1/2}$, which is essentially the best estimate we can hope for from \eqref{eq:eulerIncIeqn} if we neglect the pressure term.  Making this assumption, the right hand side of \eqref{eq:easyRIbound} is still even larger than
%
\ali{
\co{R_I} &\lsm \Xi e_v^{1/2} e_R^{1/2} + \ldots \notag
}
However, our goal estimate for $\co{R_I}$ is that $\co{R_I} \lsm e_R$ should be the same size as the original stress $\co{R^{j\ell}}$, so \eqref{eq:easyRIbound} is missing the mark by an enormous factor.  

One can do better if one uses the evolution equation for $y_I^\ell$.  Indeed, writing \eqref{eq:eulerIncIeqn} in divergence form, one has
\ali{
\pr_t y_I^\ell &= - \nb_j[ v^j y_I^\ell + y_I^j v^\ell + R^{j\ell} + \bar{p}_I \de^{j\ell} ] \notag \\
y_I^\ell(t,x) &= \nb_j \int_{t(I)}^t [y_I^j(\tau) v^\ell(\tau) + y_I^\ell(\tau) v^j(\tau) + R^{j\ell}(\tau) + \bar{p}_I(\tau) \de^{j\ell} ] d\tau \label{eq:intAntidiv}
}
If we take the integral in \eqref{eq:intAntidiv} as a definition for $r_I^{j\ell}$ and substitute into \eqref{eq:easyRIbound}, this construction of an anti-divergence leads to a much better estimate of the form
\ali{
\co{ R_I } &\lsm \co{y_I} \co{v} + \ldots, \label{eq:badBoundRI}
}
as the integration over time cancels with the large time cutoff factor in \eqref{eq:easyRIbound}.  

Unfortunately, this bound is still way off the mark $\co{R_I} \lsm e_R$ that we desire, as the bounds we have are of the order $\co{ y_I } \lsm e_R^{1/2}$ and $\co{v} \lsm 1$.  Another problem with the construction of \eqref{eq:intAntidiv} is that it does not lead to a good estimate on the advective derivative $\co{ (\pr_t + \tilde{v} \cdot \nab) R_I }$.

A natural attempt to address the latter problem is to obtain $r_I^{j\ell}$ by solving a transport equation, for instance by solving
\ali{
\begin{split}
(\pr_t + v^i \nb_i) \nb_j r_I^{j\ell} &= (\pr_t + v^i \nab_i) y_I^\ell \\
r_I^{j\ell}(t(I), x) &= 0  \label{eq:evolEqnDiv}
\end{split}
}
In this way one effectively integrates in time over trajectories to obtain $r_I^{j\ell}$, as opposed to formula \eqref{eq:intAntidiv}, which integrates in time at a fixed point $x$.  Of course, \eqref{eq:evolEqnDiv} is not a well-defined evolution equation for $r_I^{j\ell}$.  However, as was observed in \cite{isett}, in certain cases one can find solutions to \eqref{eq:evolEqnDiv} by solving another evolution equation related to \eqref{eq:evolEqnDiv} by commuting the divergence and using $\RR$ above to invert.  Here we will follow a similar approach while taking advantage of the evolution equation for $y_I^\ell$.

Commuting $v \cdot \nab$ in \eqref{eq:evolEqnDiv} with $\nb_j$ and using \eqref{eq:eulerIncIeqn}, the equation we want to solve \eqref{eq:evolEqnDiv} becomes 
\ali{
\nb_j[ (\pr_t + v^i \nb_i) r_I^{j\ell} ] &= \nb_j v^i \nb_i r_I^{j\ell} - y_I^j \nb_j v^\ell - \nb_j(y_I^j y_I^\ell) - \nb_j(\bar{p}_I \de^{j\ell}) - \nb_j R^{j\ell} \notag
}
We therefore construct $r_I^{j\ell}$ as a sum of two parts that solve the following system of equations
\ali{
r_I^{j\ell} &= \rho_I^{j\ell} + z_I^{j\ell} \notag \\
(\pr_t + v^i \nb_i) z_I^{j\ell} &= - y_I^j y_I^\ell - \bar{p}_I \de^{j\ell} - R^{j\ell} \label{eq:zItrans} \\
(\pr_t + v^i\nb_i) \rho_I^{j\ell} &= \RR^{j\ell}[\nb_a v^i \nb_i( \rho_I^{ab} + z_I^{ab} )- y_I^i \nb_i v^b ] \label{eq:rhoItrans} \\
\rho_I^{jl}(t(I), x) = z_I^{j\ell}(t(I), x) &= 0 \label{eq:initCondsRhoz}
}
Existence of a solution to \eqref{eq:rhoItrans} will follow as in the arguments of \cite{isett} provided $y_I$ remains regular.

At this point, it is quite unclear that our method has any hope of obtaining a solution $r_I^{j\ell}$ that approaches our desired estimate.  Indeed, one of the terms in the estimate for $\rho_I$ from \eqref{eq:rhoItrans} is 
\ali{
\co{ \rho_I } &\leq |\th| \co{ (\pr_t + v \cdot \nab) \rho_I^{j\ell} } \leq |\th| \co{ \RR^{j\ell}[y_I^i \nb_i v^b] } + \ldots \label{eq:rhoIbound}
}
The trivial estimate for \eqref{eq:rhoIbound} is $\co{ \RR^{j\ell}[y_I^i \nb_i v^b] } \lsm \co{ y_I } \co{ \nab v } \lsm \Xi e_v^{1/2} e_R^{1/2}$, while a better estimate comes by writing $\co{ \RR^{j\ell} \nb_i[ y_I^i v^b ] } \lsm \co{y_I} \co{v} + \ldots \lsm e_R^{1/2} + \ldots$ (pretending for the moment that the $0$th order operator $\RR^{j\ell} \nb_i$ is bounded on $C^0$).  Substituting this bound into \eqref{eq:easyRIbound}, one obtains the same estimate $\co{R_I} \lsm e_R^{1/2} + \cdots$ from \eqref{eq:badBoundRI} that we concluded was insufficient when we considered the construction in \eqref{eq:intAntidiv}.

The key idea for handling this difficulty is to exploit a special structure in the term $\RR^{j\ell}[y_I^i \nb_i v^b]$ and the incompressible Euler and Euler-Reynolds equations that allows one to prove a much better estimate.  Specifically, we observe that the estimates available for $y_I$ and $v$, even taking into account the incompressibility, cannot by themselves yield any estimate stronger than the bound $\lsm e_R^{1/2} + \cdots$ noted above.  The key to obtaining an improved estimate will be to bound this term through a dynamical approach that takes advantage of the structure of the evolution equation for $y_I$ derived from the incompressible Euler and Euler-Reynolds equations.  
This structure leads ultimately to an essentially ideal bound on the resulting $R_I$, and holds also for the similar linearized Euler term $\De^{-1}\nb_\ell[y_I^i \nb_i v^\ell]$ that appears as part of the pressure increment $\bar{p}_I$ in \eqref{eq:zItrans}.  The calculations used to estimate these terms are presented in the proof of Proposition~\ref{prop:pressEstimates} in Section~\ref{sec:pressEsts} (see the Proof of \eqref{eq:co03press}, $k = 0$) and in Section~\ref{sec:rhoIests} (see the Proof of ~\eqref{eq:coTransBdrI}).


\section{Existence Considerations} \label{sec:exist}
We now address the existence and well-definedness of the solutions to equations \eqref{eq:eulerIncIeqn} and \eqref{eq:zItrans}.  

The classical local well-posedness theory for incompressible Euler yields the following result:
\begin{thm} \label{thm:locExistEuler} Let $t_0 \in \R$ and $u_0 : \T^n \to \R^n$ be a smooth divergence free vector field.  Then there exists a unique open interval $\wtld{J} \subseteq \R$ containing $t_0$ such that there exists a solution $u : \wtld{J} \times \T^n \to \R^n$ to the incompressible Euler equations that is smooth $u \in \bigcap_{k \geq 0} C_t C_x^k$ and for all $T^* \in \pr \wtld{J}$ endpoints of $\wtld{J}$, we have
\ali{
\limsup_{t \to T^*,  t \in \wtld{J}} \co{ \nab u(t) } &= \infty \label{eq:continueNabu}
}
Furthermore, the $u^\ell$ above is unique among solutions to incompressible Euler with the regularity $\nab u \in C^0(\wtld{J} \times \T^n)$.
\end{thm}
\noindent See for example \cite{bardFrisch}, which gives a lower bound for the time of existence depending on the $C^{1,\a}$ norm of the initial data.

Theorem~\ref{thm:locExistEuler} therefore provides, for each $I$, a unique interval of existence $\wtld{J}_I$ and a unique smooth solution $(u_I^\ell,p_I)$ on $\wtld{J}_I \times \T^3$ to the incompressible Euler equations that realizes the (smooth, divergence free) initial data $u_I^\ell(t_0(I), x) = v^\ell(t_0(I), x)$.  

It follows that $y_I^\ell := u_I^\ell - v^\ell$ and $\bar{p}_I = p_I - p$ are well defined and smooth on $\wtld{J}_I \times \T^3$ (since $v^\ell$ and $p$ are also smooth on that domain).  From this, one has from the theory of regular transport equations (Proposition~\ref{prop:maxTransport} in the Appendix) that there is a unique solution $z_I^{j\ell}$ to \eqref{eq:zItrans}, \eqref{eq:initCondsRhoz} that is also smooth and well-defined on $\wtld{J}_I \times \T^3$ (since the velocity $v$ and all the terms on the right hand side of \eqref{eq:zItrans} are smooth and well-defined on that domain).  The initial value problem  \eqref{eq:rhoItrans}-\eqref{eq:initCondsRhoz} for $\rho_I$ has of the form
\ali{
\begin{split} \label{eq:transportElliptic}
(\pr_t + v \cdot \nab) \rho_I^{j\ell} &= \RR^{j\ell}[\nb_a v^i \nb_i( \rho_I^{ab} ) + Z^b ] \\
\rho_I^{j\ell}(t(I), x) &= \rho_{I,0}^{j\ell}(x) 
\end{split}
}
where $\rho_{I,0}^{j\ell} = 0$ and $Z^b$ are smooth on $\wtld{J}_I \times \T^3$.  The well-posedness of equation \eqref{eq:transportElliptic}, called a {\it transport-elliptic equation} due to the presense of $\RR$, has been considered in \cite{isett}.  It follows from the analysis there (see Theorem~\ref{prop:transElliptWPthry} in the Appendix below) that there exists a solution $\rho_I$ to \eqref{eq:rhoItrans}-\eqref{eq:initCondsRhoz} on $\wtld{J}_I \times \T^3$ that is smooth. 

The well-definedness of our construction now follows from Propositions~\ref{prop:gotAnAntidiv} and \ref{prop:gluingProp} below.
\begin{prop} \label{prop:gotAnAntidiv} If $z_I$ and $\rho_I$ are smooth solutions to \eqref{eq:zItrans}, \eqref{eq:rhoItrans}, \eqref{eq:initCondsRhoz} on $\wtld{J}_I \times \T^3$, then the tensor $r_I^{j\ell} = \rho_I^{j\ell} + z_I^{j\ell}$ is a symmetric anti-divergence for $y_I^\ell$.  That is, equation \eqref{eq:symDivEqnI} is satisfied on $\wtld{J}_I \times \T^3$.
\end{prop}

\begin{prop}[The Gluing Proposition] \label{prop:gluingProp} Let $\a = 1/17$ and $C_0$ be as in the assumptions of Lemma~\ref{lem:glueLem}.  Then there exists $\de_0 \in (0, 1/25)$ depending on $C_0$ and there exist implicit constants depending on $C_0$ such that for all $I \in \Z$ and for $\th_0 = \de_0 \plhxi^{-2} \Xi^{-1} e_v^{-1/2}$ we have
\ali{
[t_0(I) - 8\th_0, t_0(I) + 8 \th_0] &\subseteq \wtld{J}_I \notag \\
\co{ \nb^k y_I } &\lsm \nhat^{(k-2)_+} \Xi^k e_R^{1/2} \label{eq:nbkyIbd}  \\
\cda{ \nb^k y_I } &\lsm \hxi^\a \nhat^{(k-2)_+} \Xi^k e_R^{1/2} \notag\\
\co{ \nb^k \rho_I } + \co{ \nb^k z_I } &\lsm \nhat^{(k-2)_+} \Xi^k \hat{\varep} \notag \\
\cda{ \nb^k \rho_I } + \cda{ \nb^k z_I } &\lsm \hxi^\a \nhat^{(k-2)_+} \Xi^k \hat{\varep} \notag\\
\hat{\varep} &:= \fr{1}{\lhxi} \fr{e_R}{\Xi e_v^{1/2}} \notag
}
where the above estimates hold uniformly in $I$ for $t \in [t_0(I) - 8\th_0, t_0(I) + 8 \th_0]$ and $k = 0, 1, 2, 3$.
\end{prop}

Proposition~\ref{prop:gotAnAntidiv} on the equality of $\nb_j r_I^{j\ell} = y_I^\ell$ is not immediate, but rather relies crucially on the fact that both vector fields $y_I^\ell$ and $v^\ell$ are divergence free.  
We give the following proof:
\begin{proof}[Proof of Proposition~\ref{prop:gotAnAntidiv}] 
Taking the divergence of \eqref{eq:zItrans}, \eqref{eq:rhoItrans}, relabeling indices and summing gives
\ALI{
\nb_j[ (\pr_t + v^i \nb_i ) z_I^{j\ell} ] &= - \nb_j ( y_I^j y_I^\ell + \bar{p}_I \de^{j\ell} + R^{j\ell} )\\
\nb_j[ (\pr_t + v^i \nb_i ) \rho_I^{j\ell} ] &= \nb_j \RR^{j\ell}[\nb_a v^i \nb_i r_I^{jb}- y_I^i \nb_i v^b ] \\
&=  \nb_j \RR^{j\ell}\nb_i[\nb_a v^i r_I^{jb} - y_I^i v^b ] \\
&\stackrel{\eqref{eq:invertModIntegral}}{=} \nb_i[ \nb_a v^i r_I^{a\ell} - y_I^i v^\ell ] = (1 - \Pi_0)\nb_i[ \nb_a v^i r_I^{a\ell} - y_I^i v^\ell ]  \\
&= \nb_j v^i \nb_i r_I^{j\ell} - y_I^j \nb_j v^\ell \\
\nb_j[ (\pr_t + v^i \nb_i ) r_I^{j\ell} ] &= \nb_j v^i \nb_i r_I^{j\ell} - y_I^j \nb_j v^\ell  - \nb_j ( y_I^j y_I^\ell + \bar{p}_I \de^{j\ell} + R^{j\ell} ) \\
(\pr_t + v^i \nb_i ) \nb_j r_I^{j\ell} &= (\pr_t + v^j \nb_j) y_I^\ell
}
The vector field $f^\ell = \nb_j r_I^{j\ell} - y_I^\ell$ is thus a smooth solution to $(\pr_t + v \cdot \nab) f^\ell = 0$ with initial data $f^\ell(t_0(I), x) = 0$.  By uniqueness, $f^\ell$ is identically $0$ and Proposition~\ref{prop:gotAnAntidiv} follows.
\end{proof} 

We now turn to the proof of Proposition~\ref{prop:gluingProp}.  

\section{Preliminaries for the Gluing Proposition} \label{sec:gluingPrelims}
In this Section we prepare for the proof of The Gluing Proposition~\ref{prop:gluingProp} by introducing some notation and preparatory lemmas.  If $T$ is an operator acting on functions (or tensor fields) on $\T^3$ (or $\R^3$), then
\ALI{
\| T \| &:= \inf \{ A \, : \, \forall~ f \in C^\infty(\T^3), \, \, \| Tf \|_{C^0(\T^3)} \leq A \| f \|_{C^0(\T^3)} \} 
}
will denote the operator norm of $T$, regarded as a bounded operator on $C^0(\T^3)$.  If $T f = K \ast f$ is a convolution operator\footnote{The standard definition for convolution involves $f(x-h)$ rather than $f(x+h)$, but this minor difference will not be important.} with kernel $K$
\ALI{
K \ast f(x) &= \int_{\R^3} f(x + h) K(h) dh, \quad x \in \T^3
}
then we have a bound
\ali{
\| K \ast \| &\leq \| K \|_{L^1(\R^3)}  \label{eq:convL1bd}
}
We will use this inequality only in cases where $K$ is a Schwartz function on $\R^3$.

We will employ the Littlewood-Paley decomposition of continuous functions (or tensor fields) with the following notation.  Choose a radially symmetric, Schwartz function $\chi : \R^3 \to \R$ whose Fourier transform $\hat{\chi}(\xi)$ has compact support in the ball of radius $2$ in frequency space $\widehat{\R}^3$, and such that $\hat{\chi}(\xi) = 1 $ for all $|\xi| \leq 1$.  Then $\int_{\R^3} \chi(h) dh = 1$.  For $q \in \Z$, set $\chi_{\leq q}(h) = 2^{3q} \chi(2^q h)$, so that $\hat{\chi}_{\leq q} \in C_c^\infty(B_{2^q}(\widehat{\R}^3)$ and $\int_{\R^3} \chi_{\leq q}(h) dh = 1$.  

\begin{defn} \label{def:LPNotation}  For any continuous $f : \T^3 \to \R^n$, we define
\ALI{
P_{\leq q} f(x) &= \int_{\R^3} f( x + h ) \chi_{\leq q}(h) dh \\
P_q f(x) &= P_{\leq q} f(x) - P_{\leq q-1}f(x) = \chi_q \ast f(x) \\
\chi_q &= \chi_{\leq q} - \chi_{\leq q - 1}
}
We will also use the notation $P_{[q_1, q_2]} f := P_{\leq q_2} f - P_{\leq q_1} f$ and $P_{\approx q} f = P_{[q-2, q+2]} f$.  With this definition, we have $P_q f = P_{\approx q} P_q f$.
\end{defn}
Observe that $\chi_q$ defined above is Schwartz and
\ALI{
\supp \hat{\chi}_q &\subseteq \{ 2^{q - 1} \leq |\xi| \leq 2^{q + 1} \} \\
\chi_q(x) &= 2^{3q} \chi_0(2^q x)
}
Every continuous tensor field $f$ on $\T^3$ has a decomposition that converges weakly in $\DD'$ 
\ali{
f(x) &= \Pi_0 f + \sum_{q = 0}^\infty P_q f(x) \label{eq:LPdecomp}
}
The sum here extends only over $q \geq 0$ because the $\T^3$-periodicity causes $P_q f$ to vanish for $q < 0$.

We will employ the following (standard) Littlewood-Paley characterization of the $\dot{C}^\a$ seminorm.  (See the Appendix for a proof.)  This estimate implies that \eqref{eq:LPdecomp} converges absolutely in $C^0$ for $f \in C^\a$.
\begin{prop} \label{prop:LPhold}  For all $0 < \a < 1$, there exist implicit constants depending on $\a$ such that for all $f \in C^0(\T^3)$
\ali{
\| f \|_{\dot{C}^\a(\R^3)} &\lsm_\a \sup_{q \in \Z} 2^{\a q} \co{P_q f} \lsm_\a \| f \|_{\dot{C}^\a(\R^3)}  \label{eq:LPhold}
}
Moreover, the inequality $\sup_{q \in \Z} 2^{q} \co{P_q f} \lsm \co{ \nb f}$ holds in the case $\a = 1$.
\end{prop}
The Proposition applies as well to tensor fields on $\T^3$.  
The above Littlewoood-Paley characterization of the $\dot{C}^\a$ seminorm will help us to control $\dot{C}^\a$ seminorms of the solutions to our transport equations when combined with the following well-known commutator estimate.  (See the Appendix for a proof.)
\begin{prop}[Littlewood-Paley Commutator Estimate] \label{prop:LPcommutunab}  If $0 < \a \leq 1$, then there is an implicit constant depending on $\a$ (and the dimension $n$) such that for any smooth vector field $u \in L^\infty(\R^n)$ and for any smooth function $f \in L^\infty(\R^n)$, we have
\ali{
\| u \cdot \nab P_q f - P_q ( u \cdot \nab f ) \|_{C^0(\R^n)} &\lsm_\a 2^{-\a q} \| \nab u \|_{C^0(\R^n)} \| f \|_{\dot{C}^\a(\R^n)}. \label{eq:commutEst}
}
\end{prop}

In what follows, all implicit constants in the notation $\lsm$ are allowed to depend on the parameter $C_0$ in the bounds assumed in Lemma~\ref{lem:glueLem}, but not on $\de$.  The above Propositions~\ref{prop:LPhold}-\ref{prop:LPcommutunab} will be employed only for the particular choice of $\a = 1/17$ of Proposition~\ref{prop:gluingProp} or $\a = 1$.  (Any other choice of $\a \in (0,1)$ would also suffice.)  If an implicit constant depends on the parameter $\de$, we will write $\lsm_\de$.  

With these preliminaries in hand, we are now ready to begin the proof of Proposition~\ref{prop:gluingProp}.

\section{Proof of the Gluing Proposition} \label{sec:gluingProof}
We now prove the Proposition~\ref{prop:gluingProp}.  The analysis draws some inspiration from the methods in \cite{isett2}.

We start by defining a dimensionless norm for $(y_I, \rho_I, z_I)$ that controls all the quantities we wish to control in Proposition~\ref{prop:gluingProp}.  We denote this norm by $\hh(t)$.  The norm also depends on $I$, but we will suppress this dependence for simplicity of notation. 
\begin{defn} \label{defn:hNorm} For $\a = 1/17$, we define the dimensionless norm $\hh(t) : \wtld{J}_I \to \R$ by 
\ALI{
\begin{split}
\hh(t) &:= \fr{1}{e_R^{1/2}} \left(\sum_{k=0}^3 \fr{ \co{\nb^k y_I(t)}}{\nhat^{(k-2)_+} \Xi^k} \right) \\
&+ \fr{1}{\hat{\varep}} \left(\sum_{k=0}^3 \fr{ \co{\nb^k \rho_I(t)} + \co{\nb^k z_I(t)}}{\nhat^{(k-2)_+} \Xi^k} \right)  \\
&+ \fr{1}{\hxi^\a \nhat \Xi^3} \left( \fr{\cda{\nb^3 y_I(t)} }{e_R^{1/2}} + \fr{\cda{\nb^3 \rho_I(t)} + \cda{\nb^3 z_I(t)}}{\hat{\varep}} \right) 
\end{split}
}
\end{defn}
The definition of $\hh(t)$ and the interpolation inequality \eqref{ineq:holdInterp} imply the following bounds for $t \in \wtld{J}_I$
\ali{
\begin{split}
\co{\nab^k y_I(t)} + \hxi^{-\a} \cda{\nab^k y_I(t)} &\lsm \nhat^{(k-2)_+}\Xi^k e_R^{1/2} \hh(t), \quad k = 0, 1, 2, 3 \\
\co{\nab^k \rho_I} + \co{\nab^k z_I} + \hxi^{-\a}(\cda{\nab^k \rho_I} + \cda{\nab^k z_I} ) &\lsm \nhat^{(k-2)_+} \Xi^k \hat{\varep} \hh(t), \quad k = 0, 1, 2, 3
\end{split} \label{eq:theEstimateshh}
}

The following Proposition is our main estimate on the growth rate of $\hh(t)$.  
\begin{prop}\label{prop:hGrowth}  There exists a constant $B_1$ depending on $C_0$ such that for all $t \in \wtld{J}_I$,
\ali{
\hh(t) &\leq B_1 \plhxi^2 \Xi e_v^{1/2} \left|\int_{t_0(I)}^t ( 1 + \hh(\tau))^2 d\tau \right| \label{eq:propHGrowthbd}
}
\end{prop}
Our main proposition for the Gluing Lemma, Proposition~\ref{prop:gluingProp}, follows quickly from Proposition~\ref{prop:hGrowth} by the following argument.
\begin{proof}[Proof of Proposition~\ref{prop:gluingProp}]  For simplicity consider the case $I = 0$ so that $t_0(I) = 0$, and restrict to $t \geq 0$.  Set $H(\hat{t}) = 1 + \hh([B_1 \plhxi^2 \Xi e_v^{1/2}]^{-1} \hat{t} )$.  Then $H(\hat{t})$ is continuous on the appropriate rescaled interval $\hat{J}$, and, after changing variable, \eqref{eq:propHGrowthbd} gives the estimate
\ALI{
H(\hat{t}) &\leq 1 + \int_0^{\hat{t}} H(\hat{\tau})^2 d\hat{\tau}
}
It follows by Bihari's inequality (i.e. the nonlinear Gronwall inequality) that
\ALI{
H(\hat{t}) &\leq (1 - \hat{t})^{-1} \quad \tx{ for all } \hat{t} \in \hat{J}
}
Returning to $\hh(t)$, we obtain $1 + \hh(t) \leq 2$ for all $t \in \wtld{J}_I$ such that $0 \leq t \leq \fr{1}{2B_1}  \plhxi^{-2} \Xi^{-1} e_v^{-1/2} $.  One analogously obtains the same estimate for $-t$ in the same range.  Set $\de_0$ in Proposition~\ref{prop:gluingProp} to be $\de_0 = \fr{1}{16 B_1}$ and let $\th_0 := \de_0\plhxi^{-2} \Xi^{-1} e_v^{-1/2}$ be as in that Proposition.  Then all of the estimates in Proposition~\ref{prop:gluingProp} follow directly from \eqref{eq:theEstimateshh} for $t \in \wtld{J}_I \cap [t_0(I) - 8 \th_0, t_0(I) + 8 \th_0]$.  In particular, $\sup \{ \co{\nab y_I(t)} \, : \, t \in \wtld{J}_I \cap [t_0(I) - 8 \th_0, t_0(I) + 8 \th_0] \}$ is finite, so the maximal open interval of existence $\wtld{J}_I$ must therefore contain $[t_0(I) - 8 \th_0, t_0(I) + 8 \th_0]$.  (Otherwise one contradicts \eqref{eq:continueNabu} of Proposition~\ref{thm:locExistEuler} using $\co{ \nb u_I(t) } \leq \co{ \nb y_I(t)} + \co{ \nb v(t) }$.)  The case of general $I \neq 0$ is equivalent to the $I = 0$ case by a translation in the time variable, so we have finished proving Proposition~\ref{prop:gluingProp}.
\end{proof}
We now continue to the proof of Proposition~\ref{prop:hGrowth}.  

\subsection{Estimates for the Pressure Increment} \label{sec:pressEsts}
We start by estimating the pressure increment $\bar{p}_I$, which appears in both of the equations \eqref{eq:eulerIncIeqn} and \eqref{eq:zItrans}.  The bounds we obtain are as follows, with implicit constants as usual depending on $C_0$.
\begin{prop}[Pressure Estimates] \label{prop:pressEstimates}  For $t \in \wtld{J}_I$,
\ali{
\co{\nb^k \bp_I(t) } &\lsm \nhat^{(k-2)_+}\Xi^{|k|}\lhxi e_R ( 1 + \hh(t) )^2, \quad k = 0, 1, 2, 3 \label{eq:co03press} \\
\hxi^{-\a}\cda{\nb^k \bp_I(t) } &\lsm  \nhat^{(k-2)_+}\Xi^{|k|}\lhxi e_R ( 1 + \hh(t) )^2, \quad k = 0, 1, 2, 3 \label{eq:canbkpress} \\
\co{ \nb^{k+1} \bp_I(t) } &\lsm \nhat^{(k-2)_+} (\Xi e_v^{1/2}) \lhxi  \Xi^{k} e_R^{1/2} ( 1 + \hh(t))^2, \quad k = 2, 3 \label{eq:34press} \\
\hxi^{-\a} \cda{\nb^{k+1} \bp_I(t) } &\lsm  \nhat^{(k-2)_+} (\Xi e_v^{1/2}) \lhxi  \Xi^{k} e_R^{1/2} ( 1 + \hh(t))^2, \quad k = 2, 3 \label{eq:34capress}
}
\end{prop}
We start with the $k = 0$ case of \eqref{eq:co03press}, namely $\co{ \bp_I(t) } \lsm \lhxi e_R ( 1+ \hh(t))^2$
\begin{proof}[Proof of \eqref{eq:co03press}, k = 0]
Taking the divergence of \eqref{eq:eulerIncIeqn}, and using the fact that both $v^\ell$ and $y_I^\ell$ are divergence free, we have
\ALI{
\nb_\ell[v^j \nb_j y_I^\ell] &+ \nb_\ell[ y_I^j \nb_j v^\ell ] + \nb_\ell \nb_j(y_I^j y_I^\ell) + \De \bp_I = - \nb_\ell \nb_j R^{j\ell} \\
\bp_I &= -2 \De^{-1} \nb_\ell[ y_I^j \nb_j v^\ell ] - \De^{-1} \nb_\ell \nb_j[y_I^j y_I^\ell] - \De^{-1} \nb_\ell \nb_j R^{j\ell}
}
We can then decompose $\bp_I = - 2 \bp_{I,1} - \bp_{I,2} - \bp_{I,3}$ as
\ali{
\bp_{I,1} &= \De^{-1} \nb_\ell[ y_I^j \nb_j v^\ell ]  \label{eq:trickyPressTerm} \\
\bp_{I,2} + \bp_{I,3} &= \De^{-1} \nb_\ell \nb_j[y_I^jy_I^\ell] + \De^{-1} \nb_\ell \nb_j R^{j\ell} \label{eq:OkPressTerms}
}
We start by estimating \eqref{eq:OkPressTerms} as follows.  Choose $\hq \in \Z$ such that $2^{\hq - 1} < \hxi \leq 2^{\hq}$.  Then by \eqref{eq:LPdecomp} we have
\ali{
\bp_{I,2} + \bp_{I,3} &= \bp_{I,2L} + \bp_{I,3L} + \bp_{I,2H} + \bp_{I,3H} \notag\\
\bp_{I,2L} + \bp_{I,3L} &= \De^{-1} \nb_\ell \nb_j P_{\leq \hq}[ y_I^j y_I^\ell + R^{j\ell} ] \label{eq:p23L} \\
\bp_{I,2H} + \bp_{I,3H} &= \sum_{q > \hq} \De^{-1} \nb_\ell \nb_j P_q[ y_I^j y_I^\ell + R^{j\ell} ] \label{eq:p23H}
}
The operator acting on low frequencies is bounded on $C^0$ by
\ali{
\| \De^{-1} \nb_\ell \nb_j P_{\leq \hq} \| &\leq \sum_{q = 0}^{\hq} \| \De^{-1} \nb_\ell \nb_j P_q \| \notag \\
&\lsm \sum_{q = 0}^{\hq} 1 = (1 + \hq) \lsm \lhxi , \label{eq:logBdOp}
}
where we used \eqref{eq:convL1bd} and the fact that $ \De^{-1} \nb_\ell \nb_j P_q = \tilde{K}_q \ast $ are convolution operators whose kernels $\tilde{K}_q(h) = 2^{3q} K_0(2^q h)$ are rescaled Schwartz functions satisfying the same $L^1$ bound.  This scaling is due to the multiplier for $\De^{-1} \nb_\ell \nb_j$ being homogeneous of degree $0$.  From \eqref{eq:logBdOp}, we now obtain
\ali{
\co{ \bp_{I,2L} + \bp_{I,3L} } &\lsm \lhxi ( e_R \hh(t)^2 + e_R) \leq \lhxi e_R (1 + \hh(t))^2 \notag
}
For the high-frequency term, we apply $P_q = P_{[q-2, q+2]} P_q = P_{\approx q} P_q$ and use $\dot{C}^\a$ control to obtain
\ALI{
\co{\bp_{I,2H} + \bp_{I,3H} } &\leq \sum_{q  > \hq} \| \De^{-1} \nb_\ell \nb_j P_{\approx q} \| \co{ P_q[ y_I^j y_I^\ell + R^{j\ell} ] } \\
&\stackrel{\eqref{eq:LPhold}}{\lsm} \sum_{q > \hq} 2^{-\a q} \cda{y_I^j y_I^\ell + R^{j\ell}} \\
&\lsm \sum_{q > \hq} 2^{-\a q} \hxi^\a (e_R \hh(t)^2 + e_R) \\
&\lsm e_R (1 + \hh(t))^2
}
The term $\bp_{I,1}$ in \eqref{eq:trickyPressTerm} is similar to the problematic term that we encountered in the equation \eqref{eq:rhoItrans}.  For this term, we start with a similar decomposition
\ali{
\bp_{I,1} &= \bp_{I,1L} + \bp_{I,1H} \label{eq:bpI1decomp}\\
\bp_{I,1L} &= \De^{-1} \nb_\ell P_{\leq \hq} [ y_I^j \nb_j v^\ell ] \notag \\
\bp_{I,1H} &= \sum_{q > \hq} \De^{-1} \nb_\ell P_q[ y_I^j \nb_j v^\ell ] \notag
}
For the high frequency term, we can estimate
\ALI{
\bp_{I,1H} &\leq \sum_{q > \hq} \| \De^{-1} \nb_\ell P_{\approx q} \| \, \co{ P_q[ y_I^j \nb_j v^\ell ] } \\
&\lsm \sum_{q > \hq} 2^{-q} \co{ P_q[ y_I^j \nb_j v^\ell ] } \\
&\lsm \sum_{q > \hq} 2^{-q} 2^{-q} \co{ \nab[ y_I^j \nb_j v^\ell ] } \\
&\lsm \hxi^{-2} ( \co{ \nb y_I } \co{\nb v} + \co{y_I} \co{\nb^2 v} ) \\
&\lsm \hxi^{-2} \Xi^2 e_v^{1/2} e_R^{1/2} \hh(t) \leq e_R \hh(t)
}
In the second line, we used an argument similar to the proof \eqref{eq:logBdOp} to bound the $L^1$ norm of the kernel for the convolution operator, but in this case the scaling gains an inverse of the frequency $2^q$ since the multiplier for the operator is $-1$-homogeneous rather than $0$-homogeneous.

The low frequency term requires more delicate analysis, as we know already that the bound
\ALI{
\co{ \bp_{I,1L} } &\leq \| \De^{-1} \nb_\ell P_{\leq \hq} \nb_j \| \, \co{ y_I^j v^\ell}
}
is not sufficient. 

First we decompose $v$ into high and low frequencies
\ali{
\bp_{I,1L} &= \bp_{I,1LL} + \bp_{I,1LH} \notag \\
\bp_{I,1LL} &= \De^{-1} \nb_\ell P_{\leq \hq} [ y_I^j \nb_j P_{\leq \hq} v^\ell ] \notag \\
\bp_{I,1LH} &= \sum_{q > \hq} \De^{-1} \nb_\ell P_{\leq \hq}[y_I^j \nb_j P_q v^\ell ] \label{eq:bpI1lHdef}
}
Using that $y_I$ is divergence free, we estimate the LH term by
\ali{
\co{\bp_{I,1LH}} &\leq \sum_{q > \hq} \| \De^{-1} \nb_\ell P_{\leq \hq} \nb_j \| ~\co{ y_I P_q v^\ell } \notag \\
&\stackrel{\eqref{eq:logBdOp}}{\lsm} \sum_{q > \hq} (1 + \hq) ( e_R^{1/2} \hh(t) ) (2^{-q} \co{\nb v} ) \notag \\
&\lsm  (1 + \hq) \hxi^{-1} e_R^{1/2} \hh(t) \Xi e_v^{1/2} \notag \\
&\lsm \lhxi e_R \hh(t) \notag
}
For the LL term, we use the following technique closely related to the estimates for pressure increments in \cite{isett2}.  Namely, we decompose this term into {\it frequency increments} of the form 
\ali{
\bp_{I, 1LL} &= \sum_{q = -1}^{\hq - 1} \De^{-1} \nb_\ell P_{q+1} [ y_I^j \nb_j P_{\leq q+1} v^\ell ] + \De^{-1} \nb_\ell P_{\leq q} [ y_I^j \nb_j P_{q+1} v^\ell ] \label{eq:LLfreqinc}\\
&= \sum_{q = -1}^{\hq - 1} \bp_{I,1LqA} + \bp_{I,1LqB} \notag
}
For the first type of term, we use that $P_{\leq q+1} v^\ell$ is restricted to frequencies at most $2^{q + 2}$ to write
\ali{
 \De^{-1} \nb_\ell P_{q+1} [ y_I^j \nb_j P_{\leq q+1} v^\ell ] &= \De^{-1} \nb_\ell P_{q+1} [ P_{\leq q +4} y_I^j \nb_j P_{\leq q+1} v^\ell ] \notag
}
That is, the higher frequencies of $y_I$ cannot contribute to the $P_{q+1}$ projections of the product.

By Proposition~\ref{prop:gotAnAntidiv}, we have that $y_I^j = \nb_i r_I^{ij}$.  By substituting, we then have
\ali{
\co{ \bp_{I,1LqA} } &\leq \| \De^{-1} \nb_\ell P_{q+1} \| \co{P_{\leq q +4} \nb_i r_I^{ij} } \co{\nb_j P_{\leq q+1} v^\ell} \notag \\
&\lsm 2^{-q} \| P_{\leq q+4} \nb_i \| ~ \co{r_I} \co{\nb v } \notag \\
&\lsm 2^{-q} \cdot 2^q \cdot ( \hat{\varep} \hh(t))  \Xi e_v^{1/2} \notag \\
\co{ \bp_{I,1LqA} } &\lsm \plhxi^{-1} e_R \hh(t)  \label{ineq:bpIqlqA}
}
For the second term in \eqref{eq:LLfreqinc}, we use that $\nb_j y_I^j = 0$ and the same frequency restrictions to write
\ALI{
\bp_{I,1LqB} &=  \De^{-1} \nb_\ell P_{\leq q} [ y_I^j \nb_jP_{q+1} v^\ell ] =  \De^{-1} \nb_\ell P_{\leq q} \nb_j[ y_I^j P_{q+1} v^\ell ] \\
&= \De^{-1}  \nb_\ell P_{\leq q} \nb_j [ (P_{\leq q+4} \nb_i r_I^{ij} ) P_{q+1} v^\ell],
}
From this identity, we conclude
\ali{
\co{ \bp_{I,1LqB}} &\leq \| \De^{-1}  \nb_\ell P_{\leq q} \nb_j \| ~\| P_{\leq q+4} \nb_i \| ~\co{ r_I} \co{ P_{q+1} v^\ell } \notag\\
&\lsm (2+q) 2^q \hat{\varep} \hh(t) 2^{-q} \co{ \nb v } \notag\\
&\lsm (2+q) \hat{\varep} \hh(t) \Xi e_v^{1/2} = (2+q) \plhxi^{-1} e_R \hh(t) \label{ineq:bvI1lqb}
}
Substituting these estimates into \eqref{eq:LLfreqinc} gives
\ali{
\co{ \bp_{I, 1LL} } &\lsm \sum_{q = -1}^{\hq-1} ( 2 + q ) \plhxi^{-1} e_R \hh(t) \notag \\
&\lsm (1 + \hq)^2 \plhxi^{-1} e_R \hh(t) \lsm \lhxi e_R \hh(t) \label{eq:bpI1LL}
}
\hltRed{We remark that the above estimates can also be derived without the use of frequency localization of Fourier transforms of products by starting with the terms in a different form such as 
\ALI{
\De^{-1} \nb_\ell P_{q+1}[ y_I^j \nb_j P_{\leq q+1} v^\ell] &= \De^{-1} \nb_\ell P_{q+1}\nb_i [ r_I^{ij} \nb_j P_{\leq q+1} v^\ell] - \De^{-1} \nb_\ell P_{q+1} [ r_I^{ij} \nb_i \nb_j P_{\leq q+1} v^\ell].
}
}
With inequality \eqref{eq:bpI1LL} we have concluded the proof of \eqref{eq:co03press} for $k = 0$.
\end{proof}
We now move on to derivatives of order $1 \leq k \leq 3$.
\begin{proof}[Proof of \eqref{eq:co03press}]  Note that \eqref{eq:co03press} for $k = 1$ follows from the $k = 0, 2$ cases by the interpolation inequality $\co{ \nb f } \lsm \co{f}^{1/2} \co{\nb^2 f}^{1/2}$.  Also, the $k = 3$ case of \eqref{eq:co03press} is the same as the $k = 2$ case of \eqref{eq:34press}, which will be proven shortly.  Thus we consider only the case of $k = 2$ derivatives at this point.

Let $\nb_{a_1,a_2}^2 = \nb_{a_1} \nb_{a_2}$ be a partial derivative operator of order $2$.  We use the same decompositions and the same calculations as in the $k = 0$ case.  For example
\ali{
\nb_{a_1,a_2}^2 \bp_I &= - 2 \nb_{a_1,a_2}^{2} \bp_{I,1} - \nb_{a_1,a_2}^2 \bp_{I,2} - \nb_{a_1,a_2}^2 \bp_{I,3} \notag
}
Taking two derivatives of \eqref{eq:p23L}, we have
\ALI{
\co{ \nb_{a_1,a_2}^2 \bp_{I,2L} } &\lsm (1 + \hq) (\co{\nb^2(y_I^j y_I^\ell) } + \co{\nb^2 R } \\
&\lsm \lhxi ( \Xi^2 e_R \hh(t)^2 + \Xi^2 e_R ) \\
&\lsm \lhxi \Xi^2 e_R (1 + \hh(t))^2
}
Similarly for \eqref{eq:p23H} one has
\ALI{
\co{ \nb_{a_1,a_2}^2 \bp_{I,2H} } &\lsm \sum_{q > \hq} 2^{-\a q} \cda{\nb^2(y_I^j y_I^\ell) + \nb^2 R^{j\ell}} \\
\lsm \hxi^{-\a} (&\cda{\nb^2 y_I} \co{y_I} + \cda{\nb y_I} \co{\nb y_I} + \cda{y_I} \co{\nb^2 y_I} + \cda{\nb^2 R^{j\ell} }  ) \\
\lsm \hxi^{-\a} (&\hxi^\a \Xi^2 e_R \hh(t)^2 + \hxi^\a \Xi^2 e_R ) \\
&\lsm \Xi^2 e_R (1 + \hh(t))^2
}
The point in the above inequality is that the estimates for $y_I$ and $R$ do not start to see factors of $\hxi$ or $\nhat \Xi$ until going beyond the second derivative.  Above we interpolate between the $|a| = 2, 3$ cases of \eqref{eq:R1Est} to bound $\cda{\nb^2 R^{j\ell} }$.

Applying $\nb_{a_1,a_2}^2 = \nb_{a_1} \nb_{a_2}$ to \eqref{eq:bpI1lHdef} and commuting the partial derivative through gives
\ali{
\co{\nb^2 \bp_{I,1LH}} &\leq \sum_{q > \hq} \| \De^{-1} \nb_\ell P_{\leq \hq} \nb_j \| ~\co{ \nb^2( y_I P_q v^\ell ) } \notag \\
&\lsm (1 + \hq) \sum_{q > \hq} \left( \sum_{a + b = 2}\co{\nb^a y_I} \co{ \nb^b P_q v } \right) \notag \\
&\lsm (1 + \hq) \sum_{q > \hq} 2^{-q}\left( \sum_{a + b = 2} \co{\nb^a y_I} \co{\nb^{b + 1} v} \right) \notag \\
&\lsm (1 + \hq) \left(\sum_{q > \hq} 2^{-q}\right)  \sum_{a + b = 2} \Xi^a e_R^{1/2} \hh(t) \Xi^{(b+1)} e_v^{1/2} \notag \\
&\lsm (1 + \hq) \hxi^{-1} \Xi^3 e_R^{1/2} \hh(t) e_v^{1/2} \notag \\
\co{\nb^2 \bp_{I,1LH}} &\lsm \lhxi \Xi^2 e_R \hh(t) \notag
}
Here we use that no powers of $\nhat$ appear in the estimates \eqref{eq:v1Est} for $v$ until after the third derivative.  Similarly, we obtain
\ALI{
\co{ \nb^2 \bp_{I,1LqA} } &\lsm \plhxi^{-1} \Xi^2 e_R \hh(t)  \\
\co{ \nb^2 \bp_{I,1LqB} } &\lsm (2+q) \plhxi^{-1} \Xi^2 e_R \hh(t)  \\
\co{\nb^2 \bp_{I,1LL} } &\lsm  \lhxi \Xi^2 e_R \hh(t)
}
by applying $\nb_{a_1}\nb_{a_2}$ to the formulas that led to the proof of the bounds \eqref{ineq:bpIqlqA}, \eqref{ineq:bvI1lqb}, \eqref{eq:bpI1LL}, and noting that each derivative costs a factor $|\nb| \lsm \Xi$, $|\nb|^2 \lsm \Xi^2$ in the estimates, since we do not consider derivatives beyond $\co{\nb^2 r_I}, \co{\nb^2 y_I}, \co{\nb^3 v}$.  (See also Section~\ref{sec:rhoIests} below, where an analogous term is treated in detail.)
\end{proof}
\begin{proof}[Proof of \eqref{eq:canbkpress}]  This inequality follows by interpolating \eqref{eq:co03press} and \eqref{eq:34press}.
\end{proof}
\begin{proof}[Proof of \eqref{eq:34press}]  Let $\nb_{\va}$ be a partial derivative operator of order $k + 1$; i.e. $\nb_{\va} = \nb_{a_1} \nb_{a_2} \nb_{a_3}$ if $k = 2$ or $\nb_{\va} = \nb_{a_1} \nb_{a_2} \nb_{a_3} \nb_{a_4}$ if $k = 3$.  Recall the decomposition $\bar{p}_I = - 2 \bp_{I,1} - \bp_{I,2} - \bp_{I,3}$ in \eqref{eq:trickyPressTerm}-\eqref{eq:OkPressTerms}.  Using the divergence free condition on $y_I$ and $v$, we have
\ALI{
\bp_{I,1} &= \De^{-1}( \nb_\ell y_I^j \nb_j v^\ell ) 
}
Write $\nb_{\va} = \nb_{a_1} \nb_{a_2} \nb_{\chka}$ where $\nb_{\chka}$ is a lower order partial derivative operator of order $|\chka| = k - 1 \geq 1$.

Using the decomposition of \eqref{eq:bpI1decomp}, the low frequency part can be bounded by
\ali{
\co{ \nb_{\va} \bp_{I,1L} } &\leq \| \nb_{a_1} \nb_{a_2} \De^{-1} P_{\leq \hq } \| ~ \co{ \nb_{\chka}[ \nb_\ell y_I^j \nb_j v^\ell ] } \notag\\
&\lsm ( 1 + \hq ) ~ \sum_{\b + \ga = k-1} \co{\nb^{1+ \b} y_I} \co{\nb^{1 + \ga} v} \\
&\lsm ( 1 + \hq ) \sum_{\b + \ga = k-1} (\nhat^{(\b - 1)_+}\Xi^{1+\b} e_R^{1/2} \hh(t) ) (\nhat^{(\ga - 2)_+} \Xi^{1+\ga} e_v^{1/2} ) \notag \\
&\lsm \lhxi \nhat^{(k-2)_+} \Xi e_v^{1/2} \Xi^k e_R^{1/2} \hh(t) \label{ineq:lowDerivOp}
}
(which is enough for \eqref{eq:34press}).  Meanwhile, the high frequency part can be bounded by
\ali{
\nb_{\va} \bp_{I,1H} &= \sum_{q > \bar{q} } \nb_{a_1} \nb_{a_2} \De^{-1} \nb_{\chka} P_q  ( \nb_\ell y_I^j \nb_j v^\ell ) \notag \\
&= \sum_{q > \hq } \nb_{a_1} \nb_{a_2} \De^{-1} P_{\approx q}  P_q \nb_{\chka} ( \nb_\ell y_I^j \nb_j v^\ell  ) \notag \\
\co{\nb_{\va} \bp_{I,1H}} &\leq \sum_{q > \hq} \| \nb_{a_1} \nb_{a_2} \De^{-1} P_{\approx q} \| ~ \co{ P_q \nb_{\chka} ( \nb_\ell y_I^j \nb_j v^\ell )} \notag \\
&\lsm \sum_{q > \hq} 2^{-\a q} \cda{ \nb_{\chka} [\nb_\ell y_I^j \nb_j v^\ell] } \notag \\
&\lsm \hxi^{-\a} \sum_{|\va_1| + |\va_2| = k-1} \cda{ \nb_{\va_1} \nb_\ell y_I^j \nb_{\va_2} \nb_j v^\ell } \notag \\
&\lsm \hxi^{-\a} \sum_{|\va_1| + |\va_2| = k-1 } ( \cda{\nb_{\va_1} \nb_\ell y_I^j }\co{\nb_{\va_2} \nb_j v^\ell } + \co{\nb_{\va_1} \nb_\ell y_I^j}\cda{\nb_{\va_2} \nb_j v^\ell} ) \notag \\
&\lsm \hxi^{-\a} \left( \sum_{|\va_1| + |\va_2| = k-1 } \hxi^\a [\nhat^{(|\va_1| - 1)_+} \Xi^{1 + |\va_1|} e_R^{1/2} \hh(t) ][\nhat^{(|\va_2|-2)_+} \Xi^{1+|\va_2|} e_v^{1/2}] \right) \notag\\
&\lsm \nhat^{(k-2)_+} \Xi e_v^{1/2} \Xi^k e_R^{1/2} \hh(t) \label{eq:cpI1H}
}
Performing the analogous computations for $\bp_{I,2} = \De^{-1}( \nb_\ell y_I^j \nb_j y_I^\ell ) = \bp_{I,2L} + \bp_{I,2H}$ gives
\ALI{
\co{\nb_{\va} \bp_{I,2L}} &\lsm (1 + \hq) \sum_{\b + \ga = k-1} \co{\nb^{1+ \b} y_I} \co{\nb^{1 + \ga} y_I} \\
&\lsm (1 + \hq) \sum_{\b + \ga = k-1} (\nhat^{(\b - 1)_+}\Xi^{1+\b} e_R^{1/2} \hh(t) ) (\nhat^{(\ga - 1)_+}\Xi^{1+\ga} e_R^{1/2} \hh(t) ) \\
&\stackrel{\eqref{ineq:counting}}{\lsm} \lhxi \nhat^{(k-2)_+} \Xi e_R^{1/2} \Xi^k e_R^{1/2} \hh(t)^2 \\
\co{\nb_{\va} \bp_{I,2H}} &\lsm \sum_{q > \hq} 2^{-\a q} \cda{ \nb_{\chka} [\nb_\ell y_I^j \nb_j y_I^\ell] } \\
&\lsm \hxi^{-\a} \left( \sum_{\ga + \b = k-1 } \hxi^\a [\nhat^{(\ga - 1)_+} \Xi^{1 + \ga} e_R^{1/2} \hh(t)] [\nhat^{(\b - 1)_+} \Xi^{1 + \b} e_R^{1/2} \hh(t)] \right) \\
&\stackrel{\eqref{ineq:counting}}{\lsm} \nhat^{(k-2)_+} \Xi e_R^{1/2} \Xi^k e_R^{1/2} \hh(t)^2
}
As for $\bp_{I,3} = \De^{-1} \nb_j \nb_\ell R^{j\ell} = \bp_{I,3L} + \bp_{I,3H}$, we have the bounds (recalling that $\nhat = (e_v/e_R)^{1/2}$)
\ALI{
\co{\nb_{\va} \bp_{I,3L}} &\leq \|  \nb_{a_1} \nb_{a_2}\De^{-1} P_{\leq \hq} \| \co{ \nb_{\chka}\nb_j \nb_\ell R^{j\ell} } \\
&\lsm (1 + \hq) \co{ \nb^{(k+1)} R } \\
&\lsm (1 + \hq) \nhat^{(k-1)_+} \Xi^{k+1} e_R \lsm \lhxi  \Xi e_v^{1/2} \nhat^{(k-2)_+} \Xi^k e_R^{1/2} \\
\co{\nb_{\va}\bp_{I,3H}} &\leq \sum_{q > \hq} \|  \nb_{a_1} \nb_{a_2}\De^{-1} P_{\approx q} \| \co{ P_q \nb_{\chka} \nb_j \nb_\ell R^{j\ell} } \\
&\lsm \sum_{q > \hq} 2^{-\a q} \cda{\nb_{\chka} \nb_j \nb_\ell R^{j\ell}} \\
&\lsm \hxi^{-\a} \hxi^\a \nhat^{(k-1)_+} \Xi^{k+1} e_R \leq \Xi e_v^{1/2} \nhat^{(k-2)_+} \Xi^k e_R^{1/2}
}
This bound completes the proof of \eqref{eq:34press}.
\end{proof}
\begin{proof}[Proof of \eqref{eq:34capress}]  The $k = 2$ case follows by interpolation from \eqref{eq:34press}, so we need only consider the case $k = 3$.   By the Littlewood-Paley characterization of $\dot{C}^\a$ (Proposition~\ref{prop:LPhold}), it suffices to show that for all partial derivative operators $\nb_{\va}$ of order $|\va| = 4$ and all $q \in \Z$, we have
\ali{
\co{ P_q \nb_{\va} \bp_I } &\lsm 2^{-\a q} \hxi^\a \nhat (\Xi e_v^{1/2}) \lhxi  \Xi^{3} e_R^{1/2} ( 1 + \hh(t))^2 \label{eq:goalBound}
}
We use the same decomposition~\eqref{eq:trickyPressTerm}-\eqref{eq:OkPressTerms} as in the proof of \eqref{eq:34press} above.  Again writing $\nb_{\va} = \nb_{a_1} \nb_{a_2} \nb_{\chka}$,
\ALI{
P_q \nb_{\va} \bp_{I,1} &= \De^{-1}P_q \nb_{\va}( \nb_\ell y_I^j \nb_j v^\ell ) \\
&= \De^{-1} \nb_{a_1} \nb_{a_2} P_{\approx q} P_q \nb_{\chka} [ \nb_\ell y_I^j \nb_j v^\ell ] \\
\co{P_q \nb_{\va} \bp_{I,1}} &\lsm \| \De^{-1} \nb_{a_1} \nb_{a_2} P_{\approx q} \| ~\co{P_q \nb_{\chka} [ \nb_\ell y_I^j \nb_j v^\ell ] } \\
&\lsm 2^{-\a q} \cda{\nb_{\chka} [ \nb_\ell y_I^j \nb_j v^\ell ] } \\
&\lsm 2^{-\a q}\sum_{|\vcb| + |\vcc| = 2} \cda{\nb_{\vcb} \nb_\ell y_I^j \nb_{\vcc} \nb_j v^\ell }
}
\ALI{
\co{P_q \nb_{\va} \bp_{I,1}}&\lsm 2^{-\a q} \sum_{|\vcb| + |\vcc| = 2} (\cda{ \nb_{\vcb} \nb_\ell y_I^j} \co{\nb_{\vcc} \nb_j v^\ell} + \co{\nb_{\vcb} \nb_\ell y_I^j} \cda{\nb_{\vcc} \nb_j v^\ell} ) \\
&\lsm 2^{-\a q} \sum_{|\vcb| + |\vcc| = 2} \hxi^\a [\nhat^{(|\vcb| - 1)_+} \Xi^{(1+|\vcb|)} e_R^{1/2} \hh(t) ] [\nhat^{(|\vcc| - 2)_+} \Xi^{(1+|\vcc|)} e_v^{1/2}] \\
&\lsm 2^{-\a q} \hxi^\a \nhat \Xi e_v^{1/2} \Xi^3 e_R^{1/2} \hh(t) 
}
Performing the same computation for $P_q \nb_{\va} \bp_{I,2} = \De^{-1}P_q \nb_{\va}( \nb_\ell y_I^j \nb_j y_I^\ell )$ gives
\ALI{
\co{P_q \nb_{\va} \bp_{I,2}} &\lsm 2^{-\a q} \sum_{|\vcb| + |\vcc| = 2} (\cda{ \nb_{\vcb} \nb_\ell y_I^j} \co{\nb_{\vcc} \nb_j y_I^\ell} + \co{\nb_{\vcb} \nb_\ell y_I^j} \cda{\nb_{\vcc} \nb_j y_I^\ell} \\
&\lsm 2^{-\a q} \sum_{\ga + \b = 2} \hxi^\a [\nhat^{(\ga - 1)_+} \Xi^{(1+\ga)} e_R^{1/2} \hh(t) ][\nhat^{(\b - 1)_+} \Xi^{(1+\b)} e_R^{1/2} \hh(t) ] \\
&\lsm 2^{-\a q} \hxi^\a \nhat \Xi e_R^{1/2} \Xi^3 e_R^{1/2} \hh(t)^2
}
Similarly, we have 
\ALI{
P_q \nb_{\va} \bp_{I,3} &= \De^{-1} \nb_{a_1} \nb_{a_2} P_{\approx q} P_q \nb_{\chka} \pr_j \pr_\ell R^{j\ell} \\
\co{P_q \nb_{\va} \bp_{I,3} } &\lsm \co{ P_q \nb_{\chka} \pr_j \pr_\ell R^{j\ell} } \\
&\lsm 2^{-\a q} \cda{ \nb^4 R } \\
&\lsm 2^{-\a q} \hxi^\a ( \nhat^2 \Xi^4 e_R ) = 2^{-\a q} \hxi^\a \nhat \Xi e_v^{1/2} \Xi^3 e_R^{1/2}  
}
Combining these three estimates gives \eqref{eq:goalBound} and hence \eqref{eq:34capress}.
\end{proof}
We have now proven Proposition~\ref{prop:pressEstimates}.  With estimates for the pressure in hand, we turn to the evolution equations \eqref{eq:eulerIncIeqn},\eqref{eq:zItrans},\eqref{eq:rhoItrans} for $y_I$, $z_I$ and $\rho_I$.

\subsection{Estimates for the Velocity Increment}
The main result of this Section is the following array of estimates for the velocity increment $y_I^\ell$.
\begin{prop}[Velocity Estimates] \label{prop:velocEstimates}  For all $t \in \wtld{J}_I$ and for all multi-indices $\va$ with $0 \leq |\va| \leq 3$ and for all $q \in \Z$, the following estimates hold
\ali{
\co{ \nb_{\va} y_I(t) } &\lsm \nhat^{(|\va| - 2)_+}\Xi^{|\va|} e_R^{1/2} \lhxi \Xi e_v^{1/2} \left| \int_{t_0(I)}^t ( 1 + \hh(\tau))^2 d\tau \right| \label{eq:cointbdyI} \\
\hxi^{-\a} \cda{ \nb^3 y_I(t) } &\lsm \nhat \Xi^3 e_R^{1/2} \lhxi \Xi e_v^{1/2} \left| \int_{t_0(I)}^t ( 1 + \hh(\tau))^2 d\tau \right| \label{eq:cdaIntbdyI} \\
\co{ (\pr_t + v^j \nb_j + y_I^j \nb_j) \nb_{\va} y_I(t) } &\lsm \nhat^{(|\va| - 2)_+} \Xi^{|\va|} e_R^{1/2} \lhxi \Xi e_v^{1/2} ( 1 + \hh(t))^2  \label{eq:coTransBdyI}\\
\co{ (\pr_t + v^j \nb_j + y_I^j \nb_j) P_q \nb_{\va} y_I(t) } &\lsm 2^{-\a q} \hxi^\a \nhat \Xi^{3} e_R^{1/2} \lhxi \Xi e_v^{1/2} ( 1 + \hh(t))^2, \quad \tx{ if } |\va| = 3 \label{eq:LPtransBdyI}
}
\end{prop}
Observe that \eqref{eq:coTransBdyI} implies \eqref{eq:cointbdyI} by the standard estimates for transport equations (Appendix Proposition~\ref{prop:maxTransport}).  Similarly, \eqref{eq:cdaIntbdyI} follows from \eqref{eq:LPtransBdyI} and Proposition~\ref{prop:maxTransport} using Proposition~\ref{prop:LPhold}.

We write the evolution equation \eqref{eq:eulerIncIeqn} (using $\nb_j y_I^j = 0$) for $y_I^\ell$ in the form
\ali{
(\pr_t + u_I^j \nb_j) y_I^\ell &= - y_I^j \nb_j v^\ell - \nb^\ell \bar{p}_I - \nb_j R^{j\ell} \label{eq:yIevol}
}
The vector field $u_I^\ell = v^\ell + y_I^\ell$ appearing above satisfies the estimates
\ali{
\co{ \nb_{\va} u_I } + \hxi^{-\a} \cda{\nb_{\va} u_I} &\lsm \nhat^{(|\va| - 3)_+} \Xi^{|\va|} e_v^{1/2} (1+\hh(t)), \quad 1 \leq |\va| \leq 3 \label{eq:uIbds}
}
We now prove the estimate \eqref{eq:coTransBdyI}.
\begin{proof}[Proof of \eqref{eq:coTransBdyI}]  Applying $\nb_{\va}$ to \eqref{eq:yIevol} and commuting gives an equation
\ali{
(\pr_t + u_I^j \nb_j) \nb_{\va} y_I^\ell &= \sum_{|\vcb| + |\vcc| = |\va|} c_{\va,\vcb,\vcc} \nb_{\vcb} u_I^j \nb_{\vcc} \nb_j y_I^\ell \cdot 1_{|\vcb| \geq 1} \label{eq:commutYIterm} \\
&- \nb_{\va} [ y_I^j \nb_j v^\ell ] - \nb_{\va} \nb^\ell \bp_I - \nb_{\va} \nb_j R^{j\ell}  \label{eq:forcingTermsYIeqn}
}
The commutator term in line \eqref{eq:commutYIterm} is bounded by (using $|\vcc| = |\va| - |\vcb| \leq |\va| - 1$)
\ali{
\co{\eqref{eq:commutYIterm}} &\stackrel{\eqref{eq:uIbds}}{\lsm} \sum_{|\vcb| + |\vcc| = |\va|} [\nhat^{(|\vcb| - 3)_+} \Xi^{|\vcb|} e_v^{1/2}(1+\hh(t)) ] [\nhat^{(|\vcc| - 1)_+}\Xi^{(|\vcc| + 1)} e_R^{1/2} \hh(t) ] 1_{|\vcb| \geq 1} \notag \\
&\lsm \sum_{|\vcb| + |\vcc| = |\va|} \left[\nhat^{(|\vcb| - 1 - 2)_+} \nhat^{(|\vcc| + 1 - 2)_+} 1_{|\vcb| \geq 1} \right] \Xi e_v^{1/2} \Xi^{|\va|} e_R^{1/2} (1 + \hh(t))^2 \notag \\
&\stackrel{\eqref{ineq:counting}}{\lsm} \nhat^{(|\va| - 2)_+} \Xi e_v^{1/2} \Xi^{|\va|} e_R^{1/2} (1 + \hh(t))^2 \label{eq:commutYIbdfirst}
}
The first term in line \eqref{eq:forcingTermsYIeqn} is bounded by
\ali{
\co{\nb_{\va} [ y_I^j \nb_j v^\ell ]} &\lsm \sum_{|\vcb| + |\vcc| = |\va|} \co{\nb_{\vcb} y_I^j } \co{ \nb_{\vcc} \nb_j v^\ell } \notag\\
&\lsm \sum_{|\vcb| + |\vcc| = |\va|} \left[\nhat^{(|\vcb| - 2)_+} \Xi^{|\vcb|} e_R^{1/2} \hh(t) \right] \left[ \nhat^{(|\vcc| + 1 - 3)_+} \Xi^{|\vcc| + 1} e_v^{1/2} \right] \notag \\
&\lsm \nhat^{(|\va| - 2)_+} \Xi e_v^{1/2} \Xi^{|\va|} e_R^{1/2}\hh(t) \label{eq:carryingyIestimate}
}
The other terms in \eqref{eq:forcingTermsYIeqn} are bounded by 
\ali{
\co{\nb_{\va} \nb^\ell \bp_I} &\stackrel{\eqref{eq:co03press},\eqref{eq:34press}}{\lsm} \lhxi \Xi e_v^{1/2} \nhat^{(|\va| - 2)_+} \Xi^{|\va|} e_R^{1/2} (1 + \hh(t))^2 \notag \\
\co{\nb_{\va} \nb_j R^{j\ell} } &\lsm \nhat^{(|\va| + 1 - 2)_+} \Xi^{|\va| + 1} e_R \leq \nhat^{(|\va| - 2)_+} \Xi e_v^{1/2} \Xi^{|\va|} e_R^{1/2} \notag
}
\end{proof}
\begin{proof}[Proof of \eqref{eq:LPtransBdyI}]  For $|\va| = 3$, we apply $P_q$ to \eqref{eq:commutYIterm}-\eqref{eq:forcingTermsYIeqn} to obtain
\ali{
(\pr_t + u_I^j \nb_j) P_q \nb_{\va} y_I^\ell &= u_I^j \nb_j P_q \nb_{\va} y_I^\ell - P_q[u_I^j \nb_j \nb_{\va} y_I^\ell ] \label{eq:LPCommutTermyI} \\
&+ \sum_{|\vcb| + |\vcc| = |\va|} c_{\va,\vcb,\vcc} P_q [\nb_{\vcb} u_I^j \nb_{\vcc} \nb_j y_I^\ell] \cdot 1_{|\vcb| \geq 1} \label{eq:PqcommutyI} \\
&- P_q\nb_{\va} [ y_I^j \nb_j v^\ell ] - P_q \nb_{\va} \nb^\ell \bp_I - P_q\nb_{\va} \nb_j R^{j\ell} \label{eq:PqCommutForceyI}
}
By Proposition~\ref{prop:LPcommutunab}, the commutator term in line \eqref{eq:LPCommutTermyI} obeys the bound 
\ALI{
\co{ \eqref{eq:LPCommutTermyI} } &\lsm 2^{-\a q} \co{\nb u_I} \cda{ \nb_{\va} y_I^\ell } \\
&\stackrel{\eqref{eq:uIbds}}{\lsm} 2^{-\a q} [\Xi e_v^{1/2} ( 1 + \hh(t) ) ] [\hxi^\a \nhat \Xi^{|\va|} e_R^{1/2} \hh(t) ] \\
&\lsm 2^{-\a q} \hxi^\a \nhat \Xi e_v^{1/2} \Xi^{3} e_R^{1/2} (1+\hh(t))^2 
}
Following the proof of \eqref{eq:commutYIbdfirst}, the term \eqref{eq:PqcommutyI} obeys (for $|\va| = 3$)
\ALI{
\co{\eqref{eq:PqcommutyI}} &\lsm \sum_{|\vcb| + |\vcc| = |\va|} 2^{-\a q} \cda{\nb_{\vcb} u_I^j \nb_{\vcc} \nb_j y_I^\ell } 1_{|\vcb| \geq 1}  \\
&\lsm 2^{-\a q} \sum_{|\vcb| + |\vcc| = |\va|} (\cda{\nb_{\vcb} u_I^j } \co{ \nb_{\vcc} \nb_j y_I^\ell } + \co{\nb_{\vcb} u_I^j} \cda{\nb_{\vcc} \nb_j y_I^\ell }) 1_{|\vcb| \geq 1} \\
&\lsm 2^{-\a q} \sum_{|\vcb| + |\vcc| = |\va|} \hxi^\a [\nhat^{(|\vcb| - 3)_+} \Xi^{|\vcb|} e_v^{1/2}(1+\hh(t)) ] [\nhat^{(|\vcc| - 1)_+}\Xi^{(|\vcc| + 1)} e_R^{1/2} \hh(t) ] 1_{|\vcb| \geq 1} \\
&\lsm 2^{-\a q} \hxi^\a \nhat \Xi e_v^{1/2} \Xi^{3} e_R^{1/2} (1 + \hh(t))^2.
}
Modifying the proof of \eqref{eq:carryingyIestimate} as in the previous computation gives the bound 
\ALI{
\co{P_q\nb_{\va} [ y_I^j \nb_j v^\ell ]} &\lsm 2^{-\a q} \hxi^\a (\nhat^{(3 - 2)_+} \Xi e_v^{1/2} \Xi^{|\va|} e_R^{1/2}\hh(t)).
}
The other two terms in \eqref{eq:PqCommutForceyI} are bounded by:
\ALI{
\co{ P_q \nb_{\va} \nb^\ell \bp_I } &\stackrel{\eqref{eq:34capress}}{\lsm} 2^{-\a q} \hxi^\a \lhxi \nhat \Xi e_v^{1/2} \Xi^{3} e_R^{1/2} ( 1 + \hh(t))^2  \\
\co{P_q\nb_{\va} \nb_j R^{j\ell}} &\lsm 2^{-\a q} \cda{\nb^4 R } \lsm 2^{-\a q} \hxi^\a \nhat^{(4 - 2)_+} \Xi^4 e_R \\
&\lsm 2^{-\a q} \hxi^\a \nhat \Xi e_v^{1/2} \Xi^3 e_R.
}
This estimate concludes the proof of \eqref{eq:LPtransBdyI}.
\end{proof}
We now proceed to estimate the components of the anti-divergence $r_I^{j\ell} = \rho_I^{j\ell} + z_I^{j\ell}$.


\subsection{Estimates for the Anti-Divergence I: \texorpdfstring{$z_I$}{}}
The main result of this Section and the following one is the following array of estimates for the components $z_I^{j\ell}$ and $\rho_I^{j\ell}$ of the anti-divergence $r_I^{j\ell} = \rho_I^{j\ell} + z_I^{j\ell}$.  We use the notation $\hq$ to denote the integer $\hq \in \Z$ such that $2^{\hq - 1} < \hxi \leq 2^{\hq}$.
\begin{prop}[Anti-Divergence Estimates] \label{prop:antiDivEsts}   For all $t \in \wtld{J}_I$, and for all multi-indices $\va$ with $0 \leq |\va| \leq 3$ and for all $q \in \Z$ with $q > \hq$, the following estimates hold
\ali{
\co{ \nb_{\va} z_I(t) } + \co{ \nb_{\va} \rho_I(t) } &\lsm \nhat^{(|\va| - 2)_+}\Xi^{|\va|} \hat{\varep} \plhxi^2 \Xi e_v^{1/2} \left| \int_{t_0(I)}^t ( 1 + \hh(\tau))^2 d\tau \right| \label{eq:cointbdrI} \\
\hxi^{-\a} ( \cda{ \nb^3 z_I(t) } + \cda{ \nb^3 \rho_I(t) } ) &\lsm \nhat \Xi^3 \hat{\varep} \plhxi^2 \Xi e_v^{1/2} \left| \int_{t_0(I)}^t ( 1 + \hh(\tau))^2 d\tau \right| \label{eq:cdaIntbdrI} \\
\co{ D_t \nb_{\va} z_I(t) } + \co{ D_t \nb_{\va} \rho_I(t) } &\lsm \nhat^{(|\va| - 2)_+} \Xi^{|\va|} \hat{\varep} \plhxi^2 \Xi e_v^{1/2} ( 1 + \hh(t))^2  \label{eq:coTransBdrI}\\
\co{ D_t P_q \nb_{\va} z_I(t) } + \co{ D_t P_q \nb_{\va} \rho_I(t) } &\lsm 2^{-\a q} \hxi^\a \nhat \Xi^{3} \hat{\varep} \plhxi^2 \Xi e_v^{1/2} ( 1 + \hh(t))^2, \quad \tx{ if } |\va| = 3 \label{eq:LPtransBdrI}
}
where $D_t$ is the operator $D_t := (\pr_t + v \cdot \nab)$ and $\hat{\varep} = e_R / (\lhxi \, \Xi e_v^{1/2})$.
\end{prop}
As in the discussion following Proposition~\ref{prop:velocEstimates}, note that \eqref{eq:coTransBdrI} implies \eqref{eq:cointbdrI}.  The bounds \eqref{eq:cointbdrI} and \eqref{eq:LPtransBdrI} imply \eqref{eq:cdaIntbdrI} as follows.  From \eqref{eq:LPtransBdrI} and Appendix Proposition~\ref{prop:maxTransport}, we have
\ALI{
\hxi^{-\a} \sup_{q > \hq} 2^{\a q} (\co{ P_q \nb_{\va} z_I(t) } + \co{ P_q \nb_{\va} \rho_I(t) }) &\lsm \nhat \Xi^3 \hat{\varep} \plhxi^2 \Xi e_v^{1/2} \left| \int_{t_0(I)}^t ( 1 + \hh(\tau))^2 d\tau \right|, \quad |\va| = 3
}
On the other hand, in the range $q \leq \hq$, we have $2^q \leq 2^{\hq} \leq 2 \hxi$, and
\ali{
\hxi^{-\a} \sup_{q \leq \hq} 2^{\a q} (\co{P_q \nb_{\va} z_I(t) } + \co{ P_q \nb_{\va} \rho_I(t) }) &\lsm \co{\nb_{\va} z_I(t) } + \co{\nb_{\va} \rho_I(t) }, \quad |\va| = 3 \notag
}
Thus \eqref{eq:LPtransBdrI} together with \eqref{eq:cointbdrI} imply \eqref{eq:cdaIntbdrI} by the characterization of $\dot{C}^\a$ in Proposition~\ref{prop:LPhold}.

In this Section, we address the estimates \eqref{eq:coTransBdrI} and \eqref{eq:LPtransBdrI} for $z_I$.  The proof of \eqref{eq:coTransBdrI} and \eqref{eq:LPtransBdrI} for $\rho_I$ will be given in Section~\ref{sec:rhoIests} below.
\begin{proof}[Proof of \eqref{eq:coTransBdrI} for $z_I$] Applying $\nb_{\va}$ to \eqref{eq:zItrans} gives
\ali{
(\pr_t + v^i \nb_i) \nb_{\va} z_I^{j\ell} &= \sum_{|\vcb| + |\vcc| = |\va|} c_{\va,\vcb,\vcc} \nb_{\vcb} v^i \nb_{\vcc} \nb_i z_I^{j\ell} \cdot 1_{|\vcb| \geq 1} \label{eq:commutermZI}\\
&- \nb_{\va} [ y_I^j y_I^\ell ] - \nb_{\va} \bar{p}_I \de^{j\ell} - \nb_{\va} R^{j\ell} \label{eq:forceTermszI}
}
The commutator term of line \eqref{eq:commutermZI} is bounded by (using $|\vcc| = |\va| - |\vcb| \leq |\va| - 1$)
\ali{
\co{\eqref{eq:commutermZI}} &\lsm \sum_{|\vcb| + |\vcc| = |\va|} \co{\nb_{\vcb} v^i} \co{\nb_{\vcc} \nb_i z_I^{j\ell} } \cdot 1_{|\vcb| \geq 1} \notag \\
&\lsm \sum_{|\vcb| + |\vcc| = |\va|} \left[ \nhat^{(|\vcb| - 3)_+} \Xi^{|\vcb|} e_v^{1/2} \right] \left[ \nhat^{(|\vcc|-1)_+} \Xi^{|\vcc|+1} \hat{\varep} \hh(t) \right] \cdot 1_{|\vcb| \geq 1} \notag\\
&\lsm  \sum_{|\vcb| + |\vcc| = |\va|}\nhat^{(|\vcb| - 1 - 2)_+} \nhat^{(|\vcc|+ 1-2)_+}1_{|\vcb| \geq 1} [ \Xi e_v^{1/2} \Xi^{|\va|} \hat{\varep} \hh(t) ]  \notag \\
&\stackrel{\eqref{ineq:counting}}{\lsm} \nhat^{(|\va| - 2)_+} \Xi e_v^{1/2} \Xi^{|\va|} \hat{\varep} \hh(t) \label{eq:commutermzIest}
}
Using $e_R = \lhxi \Xi e_v^{1/2} \hat{\varep}$, the first term in \eqref{eq:forceTermszI} is bounded by
\ALI{
\co{\nb_{\va} [ y_I^j y_I^\ell ]} &\lsm \sum_{|\vcb| + |\vcc| = |\va|} \co{\nb_{\vcb} y_I} \co{ \nb_{\vcc} y_I } \\
&\lsm \sum_{|\vcb| + |\vcc| = |\va|} \left[ \nhat^{(|\vcb|-2)_+} \Xi^{|\vcb|} e_R^{1/2} \hh(t) \right] \left[ \nhat^{(|\vcc|-2)_+} \Xi^{|\vcc|} e_R^{1/2} \hh(t) \right] \\
&\lsm \nhat^{(|\va|-2)_+} \Xi^{|\va|} e_R \hh(t)^2 \\
&\lsm \nhat^{(|\va|-2)_+} \Xi^{|\va|} \lhxi \Xi e_v^{1/2} \hat{\varep} \hh(t)^2
}
The other two terms are bounded by
\ALI{
\co{\nb_{\va} \bp_I \de^{j\ell} } &\stackrel{\eqref{eq:co03press}}{\lsm} \nhat^{(|\va|-2)_+}\Xi^{|\va|}\lhxi e_R ( 1 + \hh(t) )^2 \\
&\lsm \nhat^{(|\va|-2)_+}\Xi^{|\va|} \plhxi^2 \Xi e_v^{1/2} \hat{\varep} ( 1 + \hh(t))^2 \\
\co{\nb_{\va} R^{j\ell}} &\lsm \nhat^{(|\va| - 2)_+} \Xi^{|\va|} e_R \\
&\lsm \nhat^{(|\va|-2)_+}\Xi^{|\va|} \lhxi \Xi e_v^{1/2} \hat{\varep}
}
Combining these estimates gives \eqref{eq:coTransBdrI} for $z_I$.
 \end{proof}
\begin{proof}[Proof of \eqref{eq:LPtransBdrI} for $z_I$]  Take $|\va| = 3$ and apply $P_q$ to equation \eqref{eq:commutermZI}-\eqref{eq:forceTermszI} to obtain
\ali{
(\pr_t + v^i \nb_i) P_q \nb_{\va} z_I^{j\ell} &= v^i \nb_i P_q \nb_{\va} z_I - P_q[ v^i \nb_i \nb_{\va} z_I ] \label{eq:LPcommutermzI}\\
&+ \sum_{|\vcb| + |\vcc| = |\va|} c_{\va,\vcb,\vcc} P_q[\nb_{\vcb} v^i \nb_{\vcc} \nb_i z_I^{j\ell}] \cdot 1_{|\vcb| \geq 1} \label{eq:commutermZIPq}\\
&- P_q\nb_{\va} [ y_I^j y_I^\ell ] - P_q\nb_{\va} \bar{p}_I \de^{j\ell} - P_q\nb_{\va} R^{j\ell} \label{eq:forceTermszIPq}
}
We estimate \eqref{eq:LPcommutermzI} using Proposition~\ref{prop:LPcommutunab} by 
\ali{
\co{\eqref{eq:LPcommutermzI}} &\lsm 2^{-\a q}\co{ \nb v } \cda{\nb_{\va} z_I } \notag \\
\co{\eqref{eq:LPcommutermzI}} &\lsm 2^{-\a q} \Xi e_v^{1/2} \hxi^\a \nhat^{(|\va| - 2)_+} \Xi^{|\va|} \hat{\varep} \hh(t) \label{eq:LPcommutermzIbd}
}
We estimate \eqref{eq:commutermZIPq} by 
\ali{
\co{\eqref{eq:commutermZIPq}} &\lsm \sum_{|\vcb| + |\vcc| = |\va|} 2^{-\a q} \cda{\nb_{\vcb} v^i \nb_{\vcc} \nb_i z_I^{j\ell} }\cdot 1_{|\vcb| \geq 1} \notag\\
&\lsm 2^{-\a q}\sum_{|\vcb| + |\vcc| = |\va|} ( \cda{\nb_{\vcb} v^i} \co{\nb_{\vcc} \nb_i z_I^{j\ell}} + \co{\nb_{\vcb} v^i} \cda{\nb_{\vcc} \nb_i z_I^{j\ell}} )\cdot 1_{|\vcb| \geq 1} \notag \\
&\lsm 2^{-\a q} \sum_{|\vcb| + |\vcc| = |\va|} \hxi^\a \left[ \nhat^{(|\vcb| - 1 - 2)_+} \Xi^{|\vcb|} e_v^{1/2} \right] \left[ \nhat^{(|\vcc| +1 - 2)_+} \Xi^{|\vcc| + 1} \hat{\varep} \hh(t) \right]\cdot 1_{|\vcb| \geq 1} \notag \\
\co{\eqref{eq:commutermZIPq}} &\lsm 2^{-\a q} \hxi^\a \nhat^{(|\va|-2)_+} \Xi e_v^{1/2} \Xi^{|\va|} \hat{\varep} \hh(t) \label{eq:commutermZIPqbd}
}
Similarly, the first term of \eqref{eq:forceTermszIPq} is bounded by
\ALI{
\co{P_q\nb_{\va} [ y_I^j y_I^\ell ]} &\lsm 2^{-\a q} \sum_{|\vcb| + |\vcc| = |\va|} \cda{\nb_{\vcb} y_I^j \nb_{\vcc} y_I^{j\ell} } \\
&\lsm 2^{-\a q} \sum_{|\vcb| + |\vcc| = |\va|} \hxi^\a \left[ \nhat^{(|\vcb| - 2)_+} \Xi^{|\vcb|} e_R^{1/2} \hh(t) \right] \left[ \nhat^{(|\vcc| - 2)_+} \Xi^{|\vcc|} e_R^{1/2} \hh(t) \right] \\
&\lsm 2^{-\a q} \hxi^\a \nhat^{(|\va|-2)_+} \Xi^{|\va|} e_R \hh(t)^2 \\
&\lsm 2^{-\a q} \hxi^\a \nhat^{(|\va|-2)_+} \lhxi \Xi e_v^{1/2} \Xi^{|\va|} \hat{\varep} \hh(t)^2 
}
Finally, the latter two terms of \eqref{eq:forceTermszIPq} are bounded by
\ALI{
\co{ P_q\nb_{\va} \bar{p}_I \de^{j\ell}} &\lsm 2^{-\a q} \cda{\nb_{\va} \bar{p}_I} \\
&\stackrel{\eqref{eq:canbkpress}}{\lsm} 2^{-\a q} \hxi^\a \lhxi \nhat^{(|\va| - 2)_+} \Xi^{|\va|} e_R (1 + \hh(t))^2 \\
&\lsm 2^{-\a q} \hxi^\a \plhxi^2 \Xi e_v^{1/2} \nhat^{(|\va| - 2)_+} \Xi^{|\va|} \hat{\varep} (1 + \hh(t))^2 \\
\co{P_q\nb_{\va} R^{j\ell}} &\lsm 2^{-\a q} \cda{ \nb_{\va} R^{j\ell} } \\
&\lsm 2^{-\a q} \hxi^\a \nhat^{(|\va| - 2)_+} \Xi^{|\va|} e_R \\
&\lsm 2^{-\a q} \hxi^\a  \lhxi \Xi e_v^{1/2} \nhat^{(|\va| - 2)_+} \Xi^{|\va|} \hat{\varep}
}
This estimate concludes the proof of \eqref{eq:LPtransBdrI} for $z_I$.
\end{proof}

\subsection{Estimates for the Anti-Divergence II: \texorpdfstring{$\rho_I$}{}} \label{sec:rhoIests}
In this Section, we establish the estimates \eqref{eq:coTransBdrI}-\eqref{eq:LPtransBdrI} for $\rho_I^{j\ell}$.
\begin{proof}[Proof of \eqref{eq:coTransBdrI} for $\rho_I$]  Let $\nb_{\va}$ be a partial derivative operator of order $0 \leq |\va| \leq 3$.  Applying $\nb_{\va}$ to \eqref{eq:rhoItrans}, we obtain 
\ali{
(\pr_t + v^i \nb_i) \nb_{\va} \rho_I^{j\ell} &=  \sum_{|\vcb| + |\vcc| = |\va|} c_{\va,\vcb,\vcc} \nb_{\vcb} v^i \nb_{\vcc} \nb_i \rho_I^{j\ell} \cdot 1_{|\vcb| \geq 1} \label{eq:commutermrhoI} \\
&+ \nb_{\va}\RR^{j\ell}[\nb_a v^i \nb_i r_I^{ab}] + \nb_{\va} \RR^{j\ell}[y_I^i \nb_i v^b] \label{eq:forceTermrhoI}
}
Repeating the proof of the estimate \eqref{eq:commutermzIest}, the commutator term in \eqref{eq:commutermrhoI} is bounded by
\ali{
\co{\eqref{eq:commutermrhoI}} &\lsm \nhat^{(|\va| - 2)_+} \Xi e_v^{1/2} \Xi^{|\va|} \hat{\varep} \hh(t) \notag
}
Let $f^{j\ell} := \nb_{\va}\RR^{j\ell}[\nb_a v^i \nb_i r_I^{ab}]$ denote the first term in \eqref{eq:forceTermrhoI}.  Using $\nb_i v^i = 0$, we have
\ali{
f^{j\ell} &= \RR^{j\ell}\nb_i[\nb_{\va}[\nb_a v^i r_I^{ab}] ] \notag
}
Letting $\hq \in \Z$ be as in the pressure estimates of Section~\ref{sec:pressEsts} ($2^{\hq - 1} < \hxi \leq 2^{\hq}$), we decompose $f^{j\ell} = f_L^{j\ell} + f_H^{j\ell}$, where
\ali{
f_L^{j\ell} &= \RR^{j\ell}\nb_i P_{\leq \hq}[\nb_{\va}[\nb_a v^i r_I^{ab}]] \notag \\
f_H^{j\ell} &= \sum_{q > \hq} \RR^{j\ell} \nb_i P_q[\nb_{\va}[\nb_a v^i r_I^{ab}] ] \label{eq:fHjl}
}
In frequency space, the $0$th order operator $\RR^{j\ell} \nb_i$ is represented by a Fourier multiplier $m_{ib}^{j\ell}$
\ali{
 \widehat{\RR^{j\ell} \nb_i}[U](\xi) &= m_{ib}^{j\ell}(\xi) \widehat{U}^b(\xi), \quad \xi \in \widehat{\R}^3 \notag
}
such that $m_{ib}^{j\ell}(\xi)$ is homogeneous of degree $0$ and smooth away from the origin.  This fact follows from the formulas \eqref{eq:helmholtz} and \eqref{eq:RRformula}, from which we observe that the corresponding Fourier-multiplier for $\RR^{j\ell}$ is smooth away from the origin and homogeneous of degree $-1$.  Repeating the proof of \eqref{eq:logBdOp}, this observation allows us to bound
\ALI{
\| \RR^{j\ell}\nb_i P_{\leq \hq} \| &\lsm (1 + \hq) \lsm \lhxi \\
\co{f_L^{j\ell}} &\lsm \lhxi \sum_{|\vcb| + |\vcc| = |\va|} \co{\nb_{\vcb} \nb_a v^i } \co{ \nb_{\vcc} r_I } \notag \\
&\lsm \lhxi \sum_{|\vcb| + |\vcc| = |\va|} [ \nhat^{(|\vcb| - 2)_+} \Xi^{|\vcb| + 1} e_v^{1/2} ] [ \nhat^{(|\vcc| - 2)_+} \Xi^{|\vcc|} \hat{\varep} \hh(t) ] \notag \\
\co{f_L^{j\ell}} &\lsm \lhxi \Xi e_v^{1/2} \nhat^{(|\va| - 2)_+} \Xi^{|\va|} \hat{\varep} \hh(t)
}
Meanwhile, the high-frequency term in \eqref{eq:fHjl} can be estimated by
\ALI{
f_H^{j\ell} &= \sum_{q > \hq} \RR^{j\ell} \nb_i P_{\approx q} P_q[\nb_{\va}[\nb_a v^i r_I^{ab}] \\
\co{f_H^{j\ell}} &\lsm \sum_{q > \hq} \| \RR^{j\ell} \nb_i P_{\approx q} \|~ \co{P_q[\nb_{\va}[\nb_a v^i r_I^{ab}]} \notag \\
&\lsm \sum_{q > \hq} 2^{-\a q} \cda{\nb_{\va}[\nb_a v^i r_I^{ab}]} \notag \\
&\lsm \hxi^{-\a} \sum_{|\vcb| + |\vcc| = |\va|} (\cda{ \nb_{\vcb} \nb_a v^i } \co{\nb_{\vcc} r_I } + \co{ \nb_{\vcb} \nb_a v^i } \cda{\nb_{\vcc} r_I } ) \notag \\
&\lsm \hxi^{-\a} \sum_{|\vcb| + |\vcc| = |\va|} \hxi^\a [\nhat^{(|\vcb| - 2 )_+} \Xi^{|\vcb| + 1} e_v^{1/2} ][\nhat^{(|\vcc| - 2)_+} \Xi^{|\vcc|} \hat{\varep} \hh(t) ] \notag \\
\co{f_H^{j\ell}}&\lsm \Xi e_v^{1/2} \nhat^{(|\va|-2)_+} \Xi^{|\va|} \hat{\varep} \hh(t)
}
We now turn to the second term in \eqref{eq:forceTermrhoI}, which we call $g^{j\ell} :=  \nb_{\va}\RR^{j\ell}[y_I^i \nb_i v^b]$, which is the dangerous term discussed in Section~\ref{sec:goodAntiDiv}.  In proving the pressure estimate \eqref{eq:co03press}, we have already encountered a term $\bp_{I,1}$ with a similar structure.  The term $g^{j\ell}$ will be estimated by similar techniques.  We start by decomposing $g^{j\ell} = g_L^{j\ell} + g_H^{j\ell}$, where
\ali{
g_L^{j\ell} &=  \nb_{\va}\RR^{j\ell} P_{\leq \hq}[y_I^i \nb_i v^b] \label{eq:rhoLowFreqForce}\\
g_H^{j\ell} &= \sum_{q > \hq} \nb_{\va}\RR^{j\ell} P_q[y_I^i \nb_i v^b]\notag
}
We treat $g_H^{j\ell}$ as follows using $P_q = P_{\approx q} P_q$
\ALI{
g_H^{j\ell} &= \sum_{q > \hq} \RR^{j\ell} P_{\approx q} [\nb_{\va} P_q[y_I^i \nb_i v^b ] \\
\co{ g_H } &\lsm \sum_{q > \hq} \| \RR^{j\ell} P_{\approx q} \| \co{ P_q \nb_{\va}[y_I^i \nb_i v^b] } \\
&\lsm \sum_{q > \hq} 2^{-q} 2^{-\a q} \cda{ \nb_{\va}[y_I^i \nb_i v^b] } \notag \\
&\lsm \hxi^{-1} \hxi^{-\a} \sum_{|\vcb| + |\vcc| = |\va|} ( \cda{\nb_{\vcb} y_I } \co{ \nb_{\vcc} \nb_i v^b } + \co{\nb_{\vcb} y_I } \cda{ \nb_{\vcc} \nb_i v^b }) \notag \\
&\lsm \hxi^{-1} \hxi^{-\a} \sum_{|\vcb| + |\vcc| = |\va|} \hxi^{\a} [\nhat^{(|\vcb| - 2)_+} \Xi^{|\vcb|} e_R^{1/2} \hh(t) ] [\nhat^{(|\vcc| - 2)_+} \Xi^{|\vcc| + 1} e_v^{1/2} ] \notag \\
&\lsm \hxi^{-1} \nhat^{(|\va| - 2)_+} \Xi^{|\va| + 1} e_R^{1/2} e_v^{1/2} \hh(t) \notag \\
&\lsm \nhat^{(|\va| - 2)_+} \Xi^{|\va|} e_R \hh(t) \notag \\
\co{ g_H } &\lsm \lhxi \Xi e_v^{1/2} \nhat^{(|\va| - 2)_+} \Xi^{|\va|} \hat{\varep} \hh(t)
}
For the low frequency part in \eqref{eq:rhoLowFreqForce}, we further decompose into $g_L^{j\ell} = g_{LL}^{j\ell} + g_{LH}^{j\ell}$, where
\ali{
g_{LL}^{j\ell} &= \nb_{\va}\RR^{j\ell} P_{\leq \hq}[ y_I^i \nb_i P_{\leq \hq} v^b ] \label{eq:gLLforcerhoI} \\
g_{LH}^{j\ell} &= \sum_{q > \hq} \nb_{\va}\RR^{j\ell} P_{\leq \hq}[y_I^i \nb_i P_q v^b]  
}
We estimate the latter term $g_{LH}^{j\ell}$ using $\nb_i y_I^i = 0$ to obtain
\ALI{
g_H^{j\ell} &= \RR^{j\ell} P_{\leq \hq} \nb_i [ \nb_{\va}[y_I^i P_q v^b ] ] \\
\co{ g_{LH}^{j\ell} } &\lsm \sum_{q > \hq} \| \RR^{j\ell} P_{\leq \hq} \nb_i \|~ \co{ \nb_{\va}[y_I^i P_q v^b ]  } \\
&\lsm \lhxi \sum_{q > \hq} \sum_{|\vcb| + |\vcc| = |\va|} \co{\nb_{\vcb} y_I} \co{P_q \nb_{\vcc} v^b} \notag \\
&\lsm \lhxi \sum_{q > \hq} \sum_{|\vcb| + |\vcc| = |\va|} \co{\nb_{\vcb} y_I}[ 2^{-q} \co{ \nab \nb_{\vcc} v^b } ] \notag \\
&\lsm \lhxi \hxi^{-1} \sum_{|\vcb| + |\vcc| = |\va|} \co{\nb_{\vcb} y_I} \co{\nb \nb_{\vcc} v^b } \notag \\
&\lsm \lhxi \hxi^{-1} \sum_{|\vcb| + |\vcc| = |\va|} [\nhat^{(|\vcb| - 2)_+} \Xi^{|\vcb|} e_R^{1/2} \hh(t) ] [\nhat^{(|\vcc| + 1 - 3)_+} \Xi^{|\vcc| + 1} e_v^{1/2}  ] \notag \\
&\lsm \lhxi \hxi^{-1} \nhat^{(|\va| - 2)_+} \Xi^{|\va| + 1} e_R^{1/2} e_v^{1/2} \hh(t) \notag \\
&\lsm \lhxi \nhat^{(|\va| - 2)_+} \Xi^{|\va|} e_R \hh(t) \notag \\
\co{g_H^{j\ell}} &\lsm \plhxi^2 \Xi e_v^{1/2} \nhat^{(|\va| - 2)_+} \Xi^{|\va|} \hat{\varep} \hh(t)
}
Our technique to estimate the term $g_{LL}^{j\ell}$ in \eqref{eq:gLLforcerhoI} is to first decompose into {\it frequency increments}:
\ali{
g_{LL}^{j\ell} &= \sum_{q = -1}^{\hq - 1} g_{LLqA}^{j\ell} + g_{LLqB}^{j\ell} \label{eq:freqincDecompRhoI} \\
g_{LLqA}^{j\ell} &= \nb_{\va}\RR^{j\ell} P_{q+1}[ y_I^i \nb_i P_{\leq q} v^b ] \notag \\
g_{LLqB}^{j\ell} &= \nb_{\va}\RR^{j\ell} P_{\leq q}[ y_I^i \nb_i P_{q+1} v^b ]\notag
}
As in the estimates for \eqref{eq:LLfreqinc}, observe that $P_{\leq q} v^b$ and $P_{q+1} v$ are restricted to frequencies $|\xi| \leq 2^{q + 2}$.  The Littlewood Paley components of $y_I$ supported in $|\xi| > 2^{q + 5}$ thus cannot contribute to the $P_{q+1}$ projection of the products above, and we have
\ali{
g_{LLqA}^{j\ell} &= \nb_{\va}\RR^{j\ell} P_{q+1}[ P_{\leq q+3} y_I^i \nb_i P_{\leq q} v^b ]  \\
g_{LLqB}^{j\ell} &= \nb_{\va}\RR^{j\ell} P_{\leq q}[ P_{\leq q + 3} y_I^i \nb_i P_{q+1} v^b ] \label{eq:LLqBrhoI}
}
From Proposition~\ref{prop:gotAnAntidiv}, we have that $y_I^i = \nb_a r_I^{ai}$.  Substituting, we obtain the bound
\ali{
\co{ g_{LLqA}^{j\ell} } &\lsm \| \RR^{j\ell} P_{q+1} \|~ \co{ \nb_{\va} [ P_{\leq q+3} y_I^i \nb_i P_{\leq q} v^b ]  } \notag \\
&\lsm 2^{-q} \sum_{|\vcb| + |\vcc| = |\va|} \co{\nb_{\vcb} P_{\leq q + 3} y_I^i } \co{\nb_{\vcc} \nb_i v } \notag \\
&\lsm 2^{-q}\sum_{|\vcb| + |\vcc| = |\va|} \co{ P_{\leq q+3} \nb_a \nb_{\vcb} r_I^{ai} } \co{\nb_{\vcc} \nb_i v } \notag \\
&\lsm 2^{-q} \sum_{|\vcb| + |\vcc| = |\va|} \| P_{\leq q+3} \nb_a \|~ \co{\nb_{\vcb} r_I^{ai}} \co{\nb_{\vcc} \nb_i v } \notag \\
&\lsm 2^{-q} \sum_{|\vcb| + |\vcc| = |\va|} 2^q [\nhat^{(|\vcb| - 2)_+} \Xi^{|\vcb|} \hat{\varep} \hh(t) ] [ \nhat^{(|\vcc| + 1 - 3)_+} \Xi^{|\vcc| + 1} e_v^{1/2} ] \notag \\
\co{ g_{LLqA}^{j\ell} } &\lsm \Xi e_v^{1/2} \nhat^{(|\va| -2)_+} \Xi^{|\va|} \hat{\varep} \hh(t) \label{eq:gLLqAbd}
}
Using that $\nb_i y_I^i = 0$ and again applying Proposition~\ref{prop:gotAnAntidiv}, we estimate \eqref{eq:LLqBrhoI} by
\ali{
g_{LLqB}^{j\ell} &= \RR^{j\ell} P_{\leq q} \nb_i \nb_{\va}[  P_{\leq q + 3} \nb_a r_I^{ai} P_{q+1} v^b ] \notag \\
\co{g_{LLqB}^{j\ell}} &\lsm \| \RR^{j\ell} P_{\leq q} \nb_i \|~ \sum_{|\vcb| + |\vcc| = |\va|} \| P_{\leq q + 3} \nb_a \|~ \co{\nb_{\vcb} r_I} \co{P_{q+1} \nb_{\vcc} v} \notag \\
&\lsm (2 + q) \sum_{|\vcb| + |\vcc| = |\va|} 2^q \co{\nb_{\vcb} r_I} [2^{-q} \co{\nb \nb_{\vcc} v} ] \notag \\
&\lsm (2 + q) \sum_{|\vcb| + |\vcc| = |\va|} [\nhat^{(|\vcb| - 2)_+} \Xi^{|\vcb|} \hat{\varep} \hh(t) ][ \nhat^{(|\vcc| +1 -3)_+} \Xi^{|\vcc| +1} e_v^{1/2}] \notag \\
\co{g_{LLqB}^{j\ell}} &\lsm (2 + q) \Xi e_v^{1/2} \nhat^{(|\va| -2)_+} \Xi^{|\va|} \hat{\varep} \hh(t) \label{eq:gLLqBbd}
}
Summing the bounds \eqref{eq:gLLqAbd} and \eqref{eq:gLLqBbd} in \eqref{eq:freqincDecompRhoI} gives
\ALI{
\co{g_{LL}^{j\ell} } &\lsm \sum_{q = -1}^{\hq - 1} (2 + q) \Xi e_v^{1/2} \nhat^{(|\va| - 2)_+} \Xi^{|\va|} \hat{\varep} \hh(t) \\
&\lsm (1 + \hq)^2 \Xi e_v^{1/2} \nhat^{(|\va| - 2)_+} \Xi^{|\va|} \hat{\varep} \hh(t) \\
&\lsm \plhxi^2 \Xi e_v^{1/2} \nhat^{(|\va| - 2)_+} \Xi^{|\va|} \hat{\varep} \hh(t)
}
This estimate concludes the proof of \eqref{eq:coTransBdrI} for $\rho_I^{j\ell}$.
\end{proof}
We now prove the estimate \eqref{eq:LPtransBdrI} for $\rho_I$
\begin{proof}[Proof of \eqref{eq:LPtransBdrI} for $\rho_I$]  Let $\nb_{\va}$ be a partial derivative operator of order $|\va| = 3$, and let $P_q$, $q \in \Z$ be a Littlewood Paley projection with $q > \hq$.  Applying $P_q$ to \eqref{eq:commutermrhoI}-\eqref{eq:forceTermrhoI}, we have
\ali{
(\pr_t + v^i \nb_i) P_q\nb_{\va} \rho_I^{j\ell} &= v^i \nb_i P_q\nb_{\va} \rho_I^{j\ell} - P_q[ v^i \nb_i \nb_{\va} \rho_I^{j\ell} ] \label{eq:PqcommutermrhoI} \\
&+  \sum_{|\vcb| + |\vcc| = |\va|} c_{\va,\vcb,\vcc} P_q[ \nb_{\vcb} v^i \nb_{\vcc} \nb_i \rho_I^{j\ell}] \cdot 1_{|\vcb| \geq 1} \label{eq:commutermrhoIPq} \\
&+ \nb_{\va}\RR^{j\ell}P_q [\nb_a v^i \nb_i r_I^{ab}] + \nb_{\va} \RR^{j\ell}P_q[y_I^i \nb_i v^b] \label{eq:forceTermrhoIPq}
}
Repeating the arguments leading to \eqref{eq:LPcommutermzIbd} and \eqref{eq:commutermZIPqbd} but with $\rho_I$ in place of $z_I$, we obtain
\ALI{
\co{\eqref{eq:PqcommutermrhoI}} &\lsm 2^{-\a q} \hxi^\a \Xi e_v^{1/2}  \nhat^{(|\va|-2)_+} \Xi^{|\va|} \hat{\varep} \hh(t) \\
\co{\eqref{eq:commutermrhoIPq}} &\lsm 2^{-\a q} \hxi^\a  \Xi e_v^{1/2} \nhat^{(|\va|-2)_+} \Xi^{|\va|} \hat{\varep} \hh(t)
}
To bound the first term in \eqref{eq:forceTermrhoIPq}, we use that $\nb_i v^i = 0$ and $P_q = P_{\approx q} P_q$ to write
\ALI{
\nb_{\va}\RR^{j\ell}P_q[\nb_a v^i \nb_i r_I^{ab}] &= \RR^{j\ell} \nb_i P_{\approx q} P_q \nb_{\va}[\nb_a v^i r_I^{ab} ] \\
\co{\nb_{\va}\RR^{j\ell}P_q[\nb_a v^i \nb_i r_I^{ab}]} &\lsm \| \RR^{j\ell} \nb_i P_{\approx q} \| ~\co{P_q \nb_{\va}[\nb_a v^i r_I^{ab} ]} \notag \\
&\lsm \sum_{|\vcb| + |\vcc| = |\va|} \co{P_q \nb_{\vcb} r_I^{ab} \nb_{\vcc}\nb_a v^i ]} \notag\\
&\lsm 2^{-\a q} \sum_{|\vcb| + |\vcc| = |\va|} \cda{\nb_{\vcb} r_I^{ab} \nb_{\vcc}\nb_a v^i } \notag }
\ALI{
&\lsm 2^{-\a q} \sum_{|\vcb| + |\vcc| = |\va|} (\cda{\nb_{\vcb} r_I^{ab}} \co{ \nb_{\vcc}\nb_a v^i} + \co{\nb_{\vcb} r_I^{ab}} \cda{ \nb_{\vcc}\nb_a v^i}) \notag \\
&\lsm 2^{-\a q} \sum_{|\vcb| + |\vcc| = |\va|} \hxi^\a [\nhat^{(|\vcb| - 2)_+} \Xi^{|\vcb|} \hat{\varep} \hh(t) ] [\nhat^{(|\vcc| - 2)_+} \Xi^{|\vcc| + 1} e_v^{1/2} ] \notag  \\
\co{\nb_{\va}\RR^{j\ell}P_q[\nb_a v^i \nb_i r_I^{ab}]} &\lsm 2^{-\a q} \hxi^\a \Xi e_v^{1/2} \nhat^{(|\va| -2)_+} \Xi^{|\va|} \hat{\varep} \hh(t)
}
To bound the second term in \eqref{eq:forceTermrhoIPq}, we use that $\hxi \leq 2^{\hq} < 2^q$ to obtain 
\ALI{
\nb_{\va} \RR^{j\ell}P_q[y_I^i \nb_i v^b] &= \RR^{j\ell}P_{\approx q} P_q\nb_{\va} [y_I^i \nb_i v^b] \\
\co{\nb_{\va} \RR^{j\ell}P_q[y_I^i \nb_i v^b] } &\leq  \| \RR^{j\ell}P_{\approx q} \| \co{P_q\nb_{\va} [y_I^i \nb_i v^b] } \\
&\lsm 2^{-q} 2^{-\a q} \cda{\nb_{\va} [y_I^i \nb_i v^b]} \notag \\
&\lsm \hxi^{-1} 2^{-\a q} \sum_{|\vcb| + |\vcc| = |\va|} (\cda{ \nb_{\vcb} y_I^i } \co{ \nb_{\vcc} \nb_i v^b} + \co{ \nb_{\vcb} y_I^i } \cda{ \nb_{\vcc} \nb_i v^b} ) \notag \\
&\lsm \hxi^{-1} 2^{-\a q} \sum_{|\vcb| + |\vcc| = |\va|} \hxi^\a [ \nhat^{(|\vcb| -2)_+} \Xi^{|\vcb|} e_R^{1/2} \hh(t) ] [ \nhat^{(|\vcc| + 1 - 3)_+} \Xi^{|\vcc| + 1} e_v^{1/2} ] \notag \\
&\lsm \hxi^{-1} 2^{-\a q} \hxi^\a \nhat^{(|\va| - 2)_+} \Xi e_v^{1/2} \Xi^{|\va|} e_R^{1/2} \hh(t) \notag \\
&\lsm 2^{-\a q} \hxi^\a \nhat^{(|\va|-2)_+}\Xi^{|\va|} e_R \hh(t) \notag \\
\co{\nb_{\va} \RR^{j\ell}P_q[y_I^i \nb_i v^b] } &\lsm 2^{-\a q} \hxi^\a \lhxi \Xi e_v^{1/2} \nhat^{(|\va|-2)_+} \Xi^{|\va|} \hat{\varep} \hh(t) 
}
This bound concludes our proof of \eqref{eq:LPtransBdrI} for $\rho_I$.
\end{proof}
Having now proven Propositions~\ref{prop:velocEstimates} and \ref{prop:antiDivEsts}, we are ready to complete the proof of the Gluing Approximation Lemma~\ref{lem:glueLem}.

\subsection{Proof of the Gluing Lemma~\ref{lem:glueLem}} \label{sec:proofGlueLem}
In this Section, we prove the Gluing Approximation Lemma~\ref{lem:glueLem} using the results of Sections~\ref{sec:exist}-\ref{sec:rhoIests}.

Observe that Proposition~\ref{prop:hGrowth} follows immediately from Propositions~\ref{prop:velocEstimates} and \ref{prop:antiDivEsts} and the Definition~\ref{defn:hNorm} of $\hh(t)$.  Proposition~\ref{prop:gluingProp} has also been proven, as it follows from Proposition~\ref{prop:hGrowth} by the argument following the statement of Proposition~\ref{prop:hGrowth}.  Thus, the time interval $[t_0(I) - 8 \th_0, t_0(I) + 8 \th_0]$ is contained in the time interval of existence $\wtld{J}_I$, and the gluing construction is well-defined.  Furthermore, recall from the proof of Proposition~\ref{prop:gluingProp} from Proposition~\ref{prop:hGrowth} that the dimensionless norm $\hh(t)$ associated with each index $I$ satisfies $\hh(t) \leq 1$ for $t \in [t_0(I) - 8 \th_0, t_0(I) + 8 \th_0]$.  

In what follows, set $J_I = [t_0(I) - 8 \th_0, t_0(I) + 8 \th_0]$.  We claim the following estimates.
\begin{prop} \label{prop:ytildevBds}Let $y^\ell$, $\tilde{v}^\ell = v^\ell + y^\ell$ and $r_I^{j\ell} = \rho_I^{j\ell} + z_I^{j\ell}$ be as defined in Sections~\ref{sec:glueConstruct}-\ref{sec:exist}.
\ali{
\sup_{t \in \R} \co{\nb^k y} &\lsm  \nhat^{(k - 2)_+} \Xi^k e_R^{1/2}, \quad k = 0,1,2,3 \label{eq:nbkyBd}\\
\sup_I \sup_{t \in J_I} \co{\nb^k (\pr_t + \tilde{v} \cdot \nab) y_I} &\lsm \lhxi \Xi e_v^{1/2} \nhat^{(k-2)_+} \Xi^k e_R^{1/2}, \quad k = 0, 1, 2 \label{eq:transyIbd}\\
\sup_I \sup_{t \in J_I} \co{\nb^k (\pr_t + \tilde{v} \cdot \nab) r_I} &\lsm \plhxi^2 \Xi e_v^{1/2} \nhat^{(k-2)_+} \Xi^k \hat{\varep}, \quad k=0, 1, 2 \label{eq:transrIbd}
}
\end{prop}
\begin{proof}
The bound \eqref{eq:nbkyBd} follows from the definition $y^\ell = \sum_I \eta_I y_I^\ell$ of $y^\ell$ and the estimate \eqref{eq:nbkyIbd}.  To obtain \eqref{eq:transyIbd}, let $\nb_{\va}$ be a partial derivative operator of order $|\va| = k \leq 2$.  Then
\ali{
(\pr_t + \tilde{v}^i \nb_i) y_I^\ell &= (\pr_t + v^i + y_I^i) \nb_i y_I^\ell + y^i \nb_i y_I^\ell - y_I^i \nb_i y_I^\ell \notag \\
\nb_{\va} [(\pr_t + \tilde{v}^i \nb_i) y_I^\ell] &= \nb_{\va} [(\pr_t + v^i \nb_i + y_I^i\nb_i) y_I^\ell] + \nb_{\va} [ (y^i - y_I^i) \nb_i y_I^\ell ] \label{eq:nbvaDtyI}
}
The first term on the right hand side of \eqref{eq:nbvaDtyI} is equal to the sum of the terms in line \eqref{eq:forcingTermsYIeqn}.  For these terms, we established a bound
\ALI{
\co{\nb_{\va} [(\pr_t + v^i\nb_i + y_I^i\nb_i) y_I^\ell]} &\lsm \lhxi \Xi e_v^{1/2} \nhat^{(|\va| - 2)_+} \Xi^{|\va|} e_R^{1/2} (1 + \hh(t))^2 \\
&\lsm \lhxi \Xi e_v^{1/2} \nhat^{(|\va| - 2)_+} \Xi^{|\va|} e_R^{1/2}
}
The other term in \eqref{eq:nbvaDtyI} is bounded by 
\ALI{
\co{\nb_{\va} [ (y^i - y_I^i) \nb_i y_I^\ell ]} &\lsm \sum_{|\vcb| + |\vcc| = |\va|} ( \co{\nb_{\vcb} y^i} + \co{\nb_{\vcb} y_I^i} ) \co{\nb_{\vcc} \nb_i y_I} \\
&\lsm \sum_{|\vcb| + |\vcc| = |\va|} [\nhat^{(|\vcb| - 2)_+} \Xi^{|\vcb|} e_R^{1/2} ] [\nhat^{(|\vcc| - 1)_+} \Xi^{|\vcc| + 1} e_R^{1/2}] \\
&\lsm \Xi e_R^{1/2} \nhat^{(|\va| - 1)_+} \Xi^{|\va|} e_R^{1/2} \leq \Xi e_v^{1/2} \nhat^{(|\va| - 2)_+} \Xi^{|\va|} e_R^{1/2}
}
To obtain \eqref{eq:transrIbd}, we compute
\ali{
(\pr_t + \tilde{v} \cdot \nab) r_I &= (\pr_t + v^i \nb_i)[\rho_I^{j\ell} + z_I^{j\ell}] + y^i \nb_i r_I^{j\ell} \notag \\
\nb_{\va} [(\pr_t + \tilde{v} \cdot \nab) r_I] &= \nb_{\va} [ (\pr_t + v^i \nb_i) r_I^{j\ell}] + \nb_{\va}[ y^i \nb_i r_I^{j\ell} ] \label{eq:newtransrI}
}
The first term on the right hand side of \eqref{eq:newtransrI} is equal to the sum of the terms in lines \eqref{eq:forceTermszI} and \eqref{eq:forceTermrhoI}.  For these terms, we proved a bound
\ALI{
\co{\nb_{\va} [ (\pr_t + v^i \nb_i) r_I^{j\ell}]} &\lsm \plhxi^2 \nhat^{(|\va| - 2)_+} \Xi e_v^{1/2} \Xi^{|\va|} \hat{\varep} ( 1+ \hh(t))^2 \\
&\lsm \plhxi^2 \Xi e_v^{1/2} \nhat^{(|\va| - 2)_+} \Xi^{|\va|} \hat{\varep} 
}
The other term in line \eqref{eq:newtransrI} can be estimated by 
\ALI{
\co{\nb_{\va}[ y^i \nb_i r_I^{j\ell} ]} &\lsm \sum_{|\vcb| + |\vcc| = |\va|} \co{\nb_{\vcb} y^i} \co{\nb_{\vcc} \nb_i r_I} \\
&\lsm \sum_{|\vcb| + |\vcc| = |\va|} [\nhat^{(|\vcb| - 2)_+} \Xi^{|\vcb|} e_R^{1/2}] [ \nhat^{(|\vcc| - 1)_+} \Xi^{|\vcc|} \hat{\varep} ] \\
&\lsm \Xi e_R^{1/2} \nhat^{(|\va| - 1)_+} \Xi^{|\va|} \hat{\varep} \leq \Xi e_v^{1/2} \nhat^{(|\va| - 2)_+} \Xi^{|\va|} \hat{\varep} 
}
The proof of Proposition~\ref{prop:ytildevBds} is now complete.
\end{proof}
Using the formulas
\ALI{
\tilde{v}^\ell &= v^\ell + y^\ell \\
R_I &= 1_{[t(I)-\th, t(I)+\th]}\eta_I'(t) (r_I^{jl} - r_{I+1}^{j\ell}) + \eta_I \eta_{I+1} ( y_I^j y_{I+1}^\ell + y_{I+1}^j y_I^\ell) \\
&-\eta_I\eta_{I+1}( y_I^j y_I^\ell + y_{I+1}^j y_{I+1}^\ell ) 
}
\ali{
\wtld{D}_t R_I &= 1_{[t(I)-\th, t(I)+\th]} [ \eta_I''(t) (r_I^{jl} - r_{I+1}^{j\ell})+ \eta_I'(t) \wtld{D}_t (r_I^{j\ell} - r_{I+1}^{j\ell})] \label{eq:rITermwDt} \\
&+ (\eta_I' \eta_{I+1} + \eta_I \eta_{I+1}') ( y_I^j y_{I+1}^\ell + y_{I+1}^j y_I^\ell ) \label{eq:ctf1wDtRI} \\
&+\eta_I \eta_{I+1} ( \wtld{D}_t y_I^j y_{I+1}^\ell + y_I^j \wtld{D}_t y_{I+1}^\ell  + \wtld{D}_ty_{I+1}^j y_I^\ell + y_{I+1}^j \wtld{D}_t y_I^\ell ) \label{eq:wtldyIyIterm1} \\
&- (\eta_I' \eta_{I+1} + \eta_I \eta_{I+1}') ( y_I^j y_I^\ell + y_{I+1}^j y_{I+1}^\ell ) \label{eq:ctf2wDtRI} \\
&- \eta_I\eta_{I+1} (\wtld{D}_t y_I^j y_I^\ell + y_I^j \wtld{D}_t y_I^\ell + \wtld{D}_t y_{I+1}^j y_{I+1}^\ell + y_{I+1}^j \wtld{D}_t y_{I+1}^\ell) \label{eq:wtldyIyIterm2} \\
\wtld{D}_t &:= (\pr_t + \tilde{v} \cdot \nab) \label{eq:wtldDtdef}
}
from \eqref{eq:RIformula} in Section~\ref{sec:glueConstruct} and applying the bounds of Proposition~\ref{prop:ytildevBds}, we obtain the following estimates
\begin{prop} \label{prop:lastBdsGlue} Uniformly in $t \in \R$, for $\wtld{D}_t$ as in \eqref{eq:wtldDtdef} we have
\ali{
\co{\nb^k \tilde{v} } 
&\lsm \Xi^k e_v^{1/2}, \quad k = 1, 2, 3 \label{eq:nbtildevkbd} \\
\sup_I \co{\nb^k R_I } &\lsm_\de \nhat^{(k - 2)_+} \Xi^k \lhxi  e_R, \quad k = 0, 1, 2, 3 \label{eq:nbkRI} \\ 
\sup_I \co{\nb^k \wtld{D}_t R_I} &\lsm_\de \plhxi^2 \Xi e_v^{1/2} \nhat^{(k - 2)_+} \Xi^k \lhxi  e_R, \quad k = 0, 1, 2 \label{eq:nbkDtRIbd}
}
\end{prop}
\begin{proof}
To obtain \eqref{eq:nbtildevkbd}, we use
\ALI{
\co{\nb^k v^\ell} + \co{ \nb^k y^\ell } &\lsm \Xi^k e_v^{1/2} + \nhat^{(k - 2)_+} \Xi^k e_R^{1/2} \lsm \Xi^k e_v^{1/2}, \quad k = 1, 2, 3
}
For \eqref{eq:nbkRI}, let $\nb_{\va}$ be a partial derivative operator of order $|\va| = k \leq 3$.  Then
\ALI{
\co{\nb_{\va} R_I} &\lsm \co{\eta_I'} \sup_I \co{ \nb_{\va} r_I } + \sum_{|\vcb| + |\vcc| = |\va|} \sup_I \co{\nb_{\vcb} y_I } \sup_I \co{\nb_{\vcc} y_I} \\
&\lsm_\de \plhxi^2 \Xi e_v^{1/2} \nhat^{(|\va| - 2)_+} \Xi^{|\va|} \hat{\varep} + \nhat^{(|\va| - 2)_+} \Xi^{|\va|} e_R \\
&\lsm_\de \lhxi \nhat^{(|\va| - 2)_+} \Xi^{|\va|} e_R
}
For \eqref{eq:nbkDtRIbd}, let $\nb_{\va}$ be a partial derivative operator of order $|\va| = k \leq 2$.  We estimate $\nb_{\va} \wtld{D}_t R_I$ by
\ALI{
\co{\nb_{\va} \eqref{eq:rITermwDt} } &\leq \sup_I [ \co{\eta_I''} \co{ \nb_{\va} r_I } + \co{ \eta_I' } \co{\nb_{\va} \wtld{D}_t r_I} ] \\
&\lsm_\de [\plhxi^2 \Xi e_v^{1/2}]^2 \nhat^{(|\va| - 2)_+} \Xi^{|\va|} \hat{\varep} \\
&\lsm_\de \plhxi^3 \Xi e_v^{1/2} \nhat^{(|\va| - 2)_+}\Xi^{|\va|} e_R \\
\co{\nb_{\va} \eqref{eq:ctf1wDtRI}} + \co{ \nb_{\va} \eqref{eq:ctf2wDtRI} } &\lsm_\de \plhxi^2 \Xi e_v^{1/2} \sum_{|\vcb| + |\vcc| = |\va|} \sup_I \co{\nb_{\vcb} y_I } \sup_I \co{\nb_{\vcc} y_I } \\
&\lsm_\de \plhxi^2 \Xi e_v^{1/2} \nhat^{(|\va| -2)_+} \Xi^{|\va|} e_R \\
\co{\nb_{\va} \eqref{eq:wtldyIyIterm1} } + \co{\nb_{\va} \eqref{eq:wtldyIyIterm2}} &\lsm  \sum_{|\vcb| + |\vcc| = |\va|} \sup_I \co{\nb_{\vcb} y_I } \sup_I \co{\nb_{\vcc} \wtld{D}_t y_I } \\
&\lsm \sum_{|\vcb| + |\vcc| = |\va|} [\nhat^{(|\vcb| -2)_+} \Xi^{|\vcb|} e_R^{1/2} ] [ \lhxi \Xi e_v^{1/2} \nhat^{(|\vcc|-2)_+} \Xi^{|\vcc|} e_R^{1/2} ] \\
&\lsm \lhxi \Xi e_v^{1/2} \nhat^{(|\va| - 2)_+} \Xi^{|\va|} e_R
}
Combining these bounds proves \eqref{eq:nbkDtRIbd} and finishes the proof of Proposition~\ref{prop:lastBdsGlue}.
\end{proof}
From Proposition~\ref{prop:lastBdsGlue}, we have established the estimates claimed in the Gluing Approximation Lemma~\ref{lem:glueLem} (noting that $\nhat^{(|\va|-2)_+} = 1$ in \eqref{eq:nbkDtRIbd} for $|\va| \leq 2$).  As we have already discussed the proof of the support properties in Sections~\ref{sec:glueConstruct}-\ref{sec:exist} (\hltRed{see in particular the discussion at the end of Section~\ref{sec:glueConstruct}}), 
we have finished the proof of Lemma~\ref{lem:glueLem}.

\part{The Convex Integration Sublemma}
In Sections~\ref{sec:mikado}-\ref{sec:concludingProof} below, we prove the Convex Integration Lemma~\ref{lem:convexInt}.  The proof of the Lemma is a combination of the framework of estimates in \cite{isett} with the construction of Mikado flows in \cite{danSze}.  We start by introducing some notation that will be used in this section.


For this proof, the notation $X \lesssim Y$ will mean that there exists a constant $C$ such that $X \leq C Y$, and this constant $C$ is allowed to depend on the constants $C_1$ and $C_\de$ and $\de$ provided in the assumptions of Lemma~\ref{lem:convexInt} (which are the conclusions of Lemma~\ref{lem:glueLem}).  If a constant depends on the parameter $\eta > 0$, we will write $\lsm_\eta$.  We will write $X \leq C_0 Y$ for an inequality involving a constant $C_0$ that is an absolute constant that is not allowed to depend on $C_1, C_\de$ or $\eta$.  The value of $C_0$ may change from line to line.

As opposed to the proof of Lemma~\ref{lem:glueLem}, in this proof the notation $\co{f}$ will refer to the $C^0 = C^0_{t,x}$ norm in both the time and space variables.

If $A = A_i^j \in \R^{3 \times 3}$ is a $3 \times 3$ matrix, we will let $|A|$ denote the Frobenius norm of $A$, namely $| A | = (\sum_{i,j = 1}^3 (A_i^j)^2)^{1/2}$.  In general, for any tensor field, it is implied that the Frobenius norm is taken pointwise whenever we write a norm such as the $C^0$ norm.


\section{Mikado Flows} \label{sec:mikado}
We recall the construction of Mikado flows from \cite{danSze}.  

Let $\F \subseteq \Z^3$ be a finite set of integer lattice vectors.  Then there exists a collection of points $(p_f)_{f \in \F}$ in $\T^3$ and a number $r_0 > 0$ such that if $\ell_f = \{ p_f + t f ~:~ t \in \R \} \subseteq \T^3$ denotes the periodization of the line passing through $p_f$ in the $f$ direction, and if $N_{\de}(\ell_f) = \{ X + h ~:~ X \in \T^3, h \in \R^3, |h| \leq \de \}$ denotes the closed $\de$-neighborhood of $\ell_f$, then
\ali{
N_{3r_0}(\ell_f) \cap N_{3r_0}(\ell_{\tilde{f}}) &= \emptyset, \quad \tx{ for all } f, \tilde{f} \in \F, f \neq \tilde{f}. \label{eq:disjointNbhds}
}
For each $f \in \F$, choose a function $\psi_f(X) : \T^3 \to \R$ of the form $\psi_f(X) = g_f(\tx{dist}(X, \ell_f))$ such that $g_f = g_f(d)$ is a smooth function with compact supported in $\supp g_f \subseteq \{ r_0 \leq d \leq 2r_0 \}$ and 
\ali{
\int_{\T^3} \psi_f(X) dX &= 0 \label{eq:int0psif} \\
\fr{1}{|\T^3|} \int_{\T^3} \psi_f^2(X) dx &= 1 \label{eq:int1psifsq}
}
Then $\psi_f(X)$ is a smooth function on $\T^3$ whose level surfaces are concentric periodic cylinders with central axis $\ell_f$.  From the orthogonality between $\nb \psi$ and $f$ and \eqref{eq:disjointNbhds}, one has that the vector fields $u_f^\ell = \psi_f(X) f^\ell$ are divergence free and have disjoint support
\ali{
\nb_\ell \psi_f(X) f^\ell &= 0 \label{eq:divFreePsif} \\
\supp \psi_f \cap \supp \psi_{\tilde{f}} &= \emptyset \quad \tx{ for all } f, \tilde{f} \in \F, f \neq \tilde{f} \label{eq:disjtSuppPsif}
}
Combining \eqref{eq:divFreePsif} and \eqref{eq:disjtSuppPsif}, one has that any linear combination $u^\ell = \sum_{f \in \F} \ga_f \psi_f(X) f^\ell$ is a stationary solution to incompressible Euler with $0$ pressure (i.e. $\nb_\ell u^\ell = 0$ and $\nb_j(u^j u^\ell) = 0$).  Solutions constructed as above are termed Mikado flows in \cite{danSze}.  

Letting $e_i$ denote the $i$th standard basis vector in $\Z^3$, we will fix our finite set $\F$ to be
\ali{
\F := \{ e_i \pm e_j ~:~ 1 \leq i < j \leq 3 \} \label{eq:choseFDirects}
}
and use $\psi_f$ to denote the above chosen $\psi_f$.  Note that the cardinality of $\F$ is $|\F| = 6$.


\section{The Coarse Scale Flow and Back-To-Labels Map} \label{sec:csfBTLmaps}
Let $(v, p, R)$ be the Euler-Reynolds flow given in the statement of Lemma~\ref{lem:convexInt}.  Let $v_\ep = \eta_\ep \ast v$ be a mollification of $v$ in the spatial variables, where $\eta_\ep : \R^3 \to \R$ is a standard smooth mollifier $\eta_\ep(h) = \ep^{-3} \eta(h/\ep)$ with compact supported in $\supp \eta_\ep(h) \subseteq \{ |h| \leq \ep \}$.  The positive number $\ep \leq 1$ will be chosen \hltRed{later in the proof in Section~\ref{sec:mollPrelim} below.}  Regardless of the choice of $\ep > 0$, one has that
\ali{
\hltRed{\co{ \nb v_\ep }} &\hltRed{\leq A_0 \co{ \nab v }} \notag
}
\hltRed{with $A_0 = \|\eta_1\|_{L^1(\R^3)}$ an absolute constant, which will later be set to $A_0 = 1$ by taking $\eta \geq 0$.}

Associated to $v_\ep$ we define the {\bf coarse scale flow} as the map $\Phi_s(t,x) : \R \times \R \times \T^3 \to \R \times \T^3$ by
\ali{
\begin{split} \label{eq:coarseScaleFlow}
\Phi_s(t,x) &= (t + s, \Phi_s^i(t,x)) \\
\fr{d}{ds} \Phi_s^i(t,x) &= v_\ep^i(\Phi_s(t,x)), \quad i = 1, 2, 3 \\
\Phi_0(t,x) &= (t,x)
\end{split}
}
Then $\Phi_s$ is the flow map of the four-vector field $\pr_t + v_\ep \cdot \nab$ on $\R \times \T^3$, and satisfies $\Phi_s \circ \Phi_{s'} = \Phi_{s + s'}$ for all $s, s' \in \R$.  In particular, for all $s \in \R$, $\Phi_s$ is a bijection on $\R \times \T^3$ with inverse map $\Phi_{-s}$.

Let $(t(I))_{I \in \Z}$ be as in the assumptions of Lemma~\ref{lem:convexInt}.  For each each $t(I) \in \R$, we define the {\bf back-to-labels map} $\Ga_I$ starting at $t(I)$ as the (unique, smooth) solution to
\ali{
(\pr_t + v_\ep^i \nb_i) \Ga_I(t,x) &= 0 \notag \\
\Ga_I(t(I), x) &= x \notag
}
One can regard $\Ga_I$ as a map $\Ga_I : \R \times \R^3 \to \R^3$ with the symmetry $\Ga_I(t,x + \ell) = \Ga_I(t,x) + \ell$ for all $\ell \in \Z^3$, which holds due to the integer periodicity of $v_\ep$ and uniqueness of solutions to the transport equation.  From this symmetry, we can also think of $\Ga_I : \R \times \T^3 \to \T^3$ as a map on the torus.

Recall that Lemma~\ref{lem:convexInt} assumes that 
\ali{
\th \co{\nb v} &\leq b_0 \notag
}
where $b_0$ is an absolute constant that remains to be chosen.  We will later choose some $b_0 \leq 1$, so that the estimates of the following Proposition hold:
\begin{prop}\label{prop:backToLabels} There exists an absolute constant $C_0$ such that if $\th \co{\nb v} \leq 1$ then for all $t \in [t(I) - \th, t(I) + \th]$ and $0 <\ep \leq 1$ 
\ali{
\co{ \nb \Ga_I } &\leq C_0 \label{eq:nbGaIbd}\\
\co{ (\pr_t + v_\ep \cdot \nb) \nb \Ga_I } &\leq C_0 \co{ \nb v } \label{eq:DtNbGaI}\\
\co{ \nb \Ga_I - \operatorname{Id} } &\leq C_0 \th \co{\nb v } \label{eq:closeToIdGaI}
}
Moreover, $\nb \Ga_I(t,x)$ is invertible at every point, and the inverse matrix $\nb \Ga_I^{-1}$ satisfies the same estimates \eqref{eq:nbGaIbd}-\eqref{eq:closeToIdGaI} when restricted to $t \in [t(I) - \th, t(I)+\th]$.
\end{prop}
\begin{proof}  The equation for $\nb \Ga_I$ is
\ali{
\label{eq:evolNabGaI}
\begin{split}
(\pr_t + v_\ep^i \nb_i) \nb_a \Ga_I^k &= - \nb_a v_\ep^i \nb_i \Ga_I^k \\
\nb_a \Ga_I^k(t(I), x) &= \mbox{Id}_a^k = \de_a^k
\end{split}
}
For $|s| \leq \th$, let $\Phi_s$ be as in \eqref{eq:coarseScaleFlow} and $\| \cdot \|^2$ denote the Frobenius norm.  Using \eqref{eq:evolNabGaI}, we have
\ali{
\fr{d}{ds} \| \nb \Ga_I \|^2(\Phi_s(t(I),x)) &= \sum_{a, k = 1}^3 (\nb \Ga_I)_a^k(\Phi_s)[(\pr_t + v_\ep \cdot \nab) (\nb \Ga_I)_a^k](\Phi_s) \notag \\
\left| \fr{d}{ds} \| \nb \Ga_I \|^2(\Phi_s(t(I),x)) \right| &\leq C_0 \co{\nb v_\ep} \| \nb \Ga_I \|^2(\Phi_s) \notag \\
\tx{ Gronwall }\quad  \Rightarrow \| \nb \Ga_I \|^2(\Phi_s(t(I),x)) &\leq \| \mbox{ Id } \| e^{C_0 \co{ \nb v_\ep }|s| } \leq \| \mbox{ Id } \| e^{C_0 \co{ \nb v } |s| }  \notag\\
|s| \co{\nb v} \leq 1 \quad \Rightarrow \| \nb \Ga_I \|^2(\Phi_s(t(I),x)) &\leq C_0^2 \notag
}
Since $\Phi_{s}(t(I), \cdot)$ maps $\{ t(I) \} \times \T^3$ onto $\{ t(I) + s \} \times \T^3$, \eqref{eq:nbGaIbd} follows.  Then \eqref{eq:DtNbGaI} follows from \eqref{eq:evolNabGaI} and \eqref{eq:nbGaIbd}.  To obtain \eqref{eq:closeToIdGaI}, write
\ALI{
\nb \Ga_I(\Phi_s(t(I), x) ) - \mbox{Id} &= \nb \Ga_I(\Phi_s(t(I), x) ) - \nb \Ga_I(\Phi_0(t(I), x) ) \\
&= \int_0^s \fr{d}{d\si} \nb \Ga_I(\Phi_\si(t(I), x) ) \\ 
&= \int_0^s [(\pr_t + v_\ep \cdot \nb) \nb \Ga_I] ( \Phi_\si(t(I), x)) ,
}
then apply \eqref{eq:nbGaIbd},\eqref{eq:DtNbGaI} and $|s| \leq \th$.  

To see the invertibiliy of $\nb \Ga_I$, let $Y_b^a$ be the unique matrix-valued solution to the equation
\ali{
\begin{split} \label{eq:inverseEqn}
(\pr_t + v_\ep^i \nb_i) Y_b^a &= \nb_i v_\ep^a Y_b^i \\
Y_b^a(t(I), x) &= \mbox{Id}^a_b
\end{split}
}
Then $Y_b^a = (\nb \Ga_I^{-1})_b^a$ is also the unique inverse to $\nb \Ga_I$, as we have
\ALI{
(\pr_t + v_\ep^i \nb_i) [ (\nb \Ga_I)_a^k Y_b^a ] &= - \nb_i \Ga_I^k \nb_a v_\ep^i Y_b^a + \nb_a \Ga_I^k \nb_i v_\ep^a Y_b^i = 0\\
(\nb \Ga_I)_a^k Y_b^a(t(I), x) &= \mbox{Id}_b^k
}
The proof of the estimates \eqref{eq:nbGaIbd}-\eqref{eq:closeToIdGaI} for $(\nb \Ga_I^{-1})_b^a$ proceed exactly as for $\nb \Ga_I$, but using \eqref{eq:inverseEqn}.
\end{proof}

\section{Ansatz for the Correction} \label{sec:Ansatz}
We now explain the Ansatz for the correction.   The Ansatz used here is equivalent to the one in \cite{danSze} with the only significant differences being the presence of time cutoffs and the use of multiple waves.


The new Euler-Reynolds flow $(v_1, p_1, R_1)$ in the conclusion of Lemma~\ref{lem:convexInt} will have a velocity field of the form $v_1 = v + V$.  The correction $V$ is a sum of divergence free vector fields $V_J$ indexed by $J = (I, f) \in \Z \times \F$.  The integer part $I \in \Z$ will specify the $t(I) \in \R$ around which $V_J$ is supported in time, while the index $f \in \F$ will correspond to the direction in which $V_J$ takes values. 
\ali{
V^\ell &= \sum_{J \in \Z \times \F} V_J^\ell, \qquad \nb_\ell V_J^\ell = 0 \quad \tx{ for all } J \in \Z \times \F \notag
}
The leading order term in each $V_J$, $J = (I, f) \in \Z \times \F$, has the structure of a Mikado flow (as described in Section~\ref{sec:mikado}), but rescaled to have a large frequency $\la \in \Z$ and made to move along the coarse scale flow by composition with the back-to-labels map 
\ali{
V_J^\ell &= \VR_J^\ell + \de V_J^\ell \notag \\
\VR_J^\ell(t,x) &= v_J^\ell(t,x) \psi_f(\la \Ga_I(t,x)), \quad J = (I, f) \label{eq:leadingTerm}
}
The leading order term $\VR_J$ is divergence free to leading order in the large parameter $\la$, as we will have
\ali{
v_J^\ell \nb_\ell[ \psi_f(\la \Ga_I(t,x)) ] &= 0 \label{eq:keyCancel}
}
The lower order term $\de V_J^\ell$ will be chosen such that $V_J^\ell$ is exactly divergence free.

More precisely, the amplitude $v_J^\ell$ in \eqref{eq:leadingTerm} will have the form
\ali{
v_J^\ell &= e_I^{1/2}(t) \ga_J(t,x) (\nb \Ga_I^{-1})_a^\ell f^a  \label{eq:vJAnsatz}
}
for functions $e_I^{1/2}(t)$ and $\ga_J(t,x)$ to be described shortly.  The identity \eqref{eq:keyCancel} follows from this Ansatz thanks to the presence of $(\nb \Ga_I^{-1})$ and \eqref{eq:divFreePsif}.  The function $e_I^{1/2}(t)$ is required to satisfy
\ali{
\suppt e^{1/2}_I(t) &\subseteq [t(I) - \th, t(I) + \th] \label{eq:suppteFnctn}
}

To correct the Ansatz~\eqref{eq:leadingTerm} and ensure the divergence free condition, for each $f \in \F$ choose a smooth $(2,0)$ tensor field $\Om_f^{\a \b} : \T^3 \to \R^3 \otimes \R^3$ such that $\Om_f^{\a \b}$ is anti-symmetric in $\a,\b$, and
\ali{
\nb_\a \Om_f^{\a \b}(X) &= \psi_f(X) f^\b \label{eq:antiDiv} \\
\int_{\T^3} \Om_f^{\a \b}(X) dX &= 0 \label{eq:Omfint0}
}
One can take for instance $\Om_f^{\a \b} = \nb^\a \De^{-1} [ \psi_f f^\b ] - \nb^\b \De^{-1}[\psi_f f^\a]$.  The existence of this choice relies on the fact that $\psi_f(X) f^\ell$ is divergence free and has integral $0$ on $\T^3$.

We now define
\ali{
V_J^\ell &= \la^{-1} \nb_a[(\nb \Ga_I^{-1})^a_\a (\nb \Ga_I^{-1})_\b^\ell e^{1/2}_I(t) \ga_J \Om_f^{\a \b}(\la \Ga_I) ] , \quad J = (I,f) \in \Z \times \F \label{eq:bigDivFormula}
}
The $V_J^\ell$ above is divergence free because it is the divergence of an antisymmetric tensor.  (The tensor within the brackets is antisymmetric in $a, \ell$ because $\Om_f^{\a \b} = -\Om_f^{\b\a}$ is antisymmetric in $\a, \b$.)  Expanding the divergence in \eqref{eq:bigDivFormula} and using \eqref{eq:antiDiv}, one sees that $V_J^\ell$ has the form \eqref{eq:leadingTerm} for
\ali{
\de V_J^\ell &= \de v_{J,\a\b}^\ell \Om_f^{\a \b}(\la \Ga_I)  \label{eq:deVJdevJ}\\
\de v_{J,\a\b}^\ell &= \la^{-1}\nb_a[(\nb \Ga_I^{-1})^a_\a (\nb \Ga_I^{-1})_\b^\ell e_I^{1/2}(t) \ga_J(t,x) ]  \label{eq:devJform}
}

\section{The Error Terms}
Let $(v,p,R)$ be the given Euler-Reynolds flow obeying the assumptions of Lemma~\ref{lem:convexInt}.  The new Euler-Reynolds flow $(v_1, p_1, R_1)$ will have the form $v_1 = v + V$, $p_1 = p + P$, and $R_1^{j\ell}$ that satisfies
\ali{
\nb_j R_1^{j\ell} &= \pr_t V^\ell + \nb_j[v^j V^\ell + V^j v^\ell + V^j V^\ell + P \de^{j\ell} + R^{j\ell}] \notag
}
We will define mollifications $v_\ep$ and $R_\ep$ to approximate $v$ and $R$.  Recall also the decomposition $V^\ell = \sum_J \VR_J^\ell + \de V_J^\ell$.  In terms of these, the new error $R_1^{j\ell}$ will be composed of terms that solve 
\ali{
R_1^{j\ell} &= R_M^{j\ell} + R_T^{j\ell} + R_S^{j\ell} + R_H^{j\ell} \label{eq:R1decomp} \\
R_M^{j\ell} &= (v^j - v_\ep^j) V^\ell + V^j ( v^\ell - v_\ep^\ell) + (R^{j\ell} - R_\ep^{j\ell}) \label{eq:mollifTerm}\\
\nb_j R_T^{j\ell} &=  \pr_t V^\ell + v_\ep^j \nb_j V^\ell + V^j \nb_j v_\ep^\ell \label{eq:transTerm}\\
R_S^{j\ell} &= \sum_{J, K \in \Z \times \F} \de V_J^j \VR_K^\ell + \VR_J^j \de V_K^\ell + \de V_J^j \de V_K^\ell \label{eq:lowOrderProducterm} \\
\nb_j R_H^{j\ell} &= \nb_j \left[ \sum_{J \in \Z \times \F} \VR_J^j \VR_J^\ell + P \de^{j\ell} + R_\ep^{j\ell} \right] \label{eq:highFreqTerm}
}
Note that we used $\nb_j V_J^j = 0$ to obtain \eqref{eq:transTerm}.  Note also that \eqref{eq:lowOrderProducterm} is symmetric in $j,\ell$ due to the double sum over $\Z \times \F$.  In order to obtain \eqref{eq:highFreqTerm}, a key cancellation comes from the fact that $\supp \VR_J \cap \supp \VR_K = \emptyset$ for all $J, K \in \Z \times \F, J \neq K$, which eliminates all the cross terms in the product.  This disjointness of support follows from \eqref{eq:disjtSuppPsif}, \eqref{eq:suppteFnctn} and \eqref{ct:disjointness}.

The amplitudes $v_J^\ell$ in \eqref{eq:leadingTerm}-\eqref{eq:vJAnsatz} and the correction $P$ to the pressure will be chosen such that
\ali{
\sum_{J \in \Z \times \F} v_J^j v_J^\ell + P \de^{j\ell} + R_\ep^{j\ell} &= 0 \label{eq:stressEqn}
}
In this way, the ``low-frequency'' part of \eqref{eq:highFreqTerm}  will cancel out, and \eqref{eq:highFreqTerm} becomes (using \eqref{eq:leadingTerm},\eqref{eq:vJAnsatz})
\ali{
\nb_j R_H^{j\ell} &= \sum_{J \in \Z \times \F} \nb_j[ v_J^j v_J^\ell ( \psi_f^2(\la \Ga_I) - 1) ] , \quad J = (I,f) \label{eq:highFreqReduced}
}

\section{The Algebraic Equation} \label{sec:algEqn}
In this Section, we specify how the $e_I^{1/2}(t)$, $\ga_J(t,x)$ and $P$ above are chosen so that \eqref{eq:stressEqn} is satisfied.

From Lemma~\ref{lem:convexInt}, there is a decomposition $R = \sum_I R_I$ with $\suppt R_I \subseteq [t(I) - \fr{\th}{2}, t(I) + \fr{\th}{2}]$.  We will define $R_\ep^{j\ell} = \eta_\ep \ast R^{j\ell}$ by mollifying only in the spatial variables, and hence we obtain an analogous decomposition $R_\ep^{j\ell} = \sum_I R_{I,\ep}^{j\ell}$, with $R_{I,\ep}^{j\ell} = \eta_\ep \ast R_I^{j\ell}$ supported in the same time intervals $\suppt R_{I,\ep}^{j\ell} \subseteq [t(I) - \fr{\th}{2}, t(I) + \fr{\th}{2}]$.  

Writing $P = \sum_I P_I$, equation \eqref{eq:stressEqn} now reduces to choosing $v_J$ and $P_I$ such that for all $I \in \Z$
\ali{
\sum_{J \in I \times \F} v_J^j v_J^\ell + P_I \de^{j\ell} + R_{I,\ep}^{j\ell} &= 0. \label{eq:localStress}
}
We take $P_I = - e_I(t)$, and \eqref{eq:localStress} reduces to
\ali{
\sum_{J \in I \times \F} v_J^j v_J^\ell = e_I(t) \sum_{J \in I \times \F} \ga_J^2 (\nb \Ga_I^{-1})^j_a (\nb \Ga_I^{-1})_b^\ell f^a f^b &= e(t) \de^{j\ell} - R_{I, \ep}^{j\ell} \notag
}
Assuming that we can divide by $e_I(t)$, this equation will hold if we have for all $I \in \Z$ 
\ali{
\sum_{J \in I \times \F} \ga_J^2 (\nb \Ga_I^{-1})^j_a (\nb \Ga_I^{-1})_b^\ell f^a f^b &= \de^{j\ell} + \varep_I^{j\ell} \label{eq:gaJeqnwithInverse} \\
\varep_I^{j\ell} &= - e_I(t)^{-1} R_{I,\ep}^{j\ell} \notag
}
To ensure that the above division is well-behaved, we choose $e_I^{1/2}(t)$ to have the form
\ali{
e_I^{1/2}(t) &= [ K C_\de \lhxi ~ e_R ]^{1/2} \eta_{\th/8} \ast_t 1_{[t(I) - 3 \th/4, t(I) + 3\th/4 ]}(t) \label{eq:eIhalf}
}
In the above formula, $C_\de$ is the constant in the upper bound \eqref{eq:newRIbdGlue}, $\eta_{\th/8}(\tau)$ is a standard mollifying kernel in the time variable supported in $|\tau| \leq \fr{\th}{8}$, $1_{[t(I) - 3 \th/4, t(I) + 3\th/4 ]}$ is the characteristic function of $[t(I) - 3\th/4, t(I) + 3\th/4]$ and $K$ is a large constant to be determined shortly.  Note that the support restriction \eqref{eq:suppteFnctn} is satisfied by the above formula.

From the support property of $R_{I,\ep}$ and \eqref{eq:newRIbdGlue} (which holds also for $R_{I,\ep}$), we have
\ali{
\co{ \varep_I } &\leq K^{-1} \label{eq:varepC0bd}
}
Applying $(\nb \Ga_I)$ to \eqref{eq:gaJeqnwithInverse}, it suffices have for all $I \in \Z$ and all $(t,x) \in [t(I)-\th, t(I)+\th] \times \T^3$ that
\ali{
\sum_{J \in \{I\}\times \F} \ga_J^2 f^j f^\ell &= (\nb \Ga_I)_a^j (\nb \Ga_I)_b^\ell( \de^{ab} + \varep^{ab} ), \qquad J = (I, f) \in \Z \times \F \label{eq:goodgaJeqn}
}
At each point $(t,x)$, the right hand side belongs to the space $\SS \subseteq \R^3 \otimes \R^3$ of symmetric $(2,0)$ tensors.  From our choice of $\F$ in \eqref{eq:choseFDirects}, the following claims hold
\ali{
\tx{ The tensors } (f^j f^\ell)_{f \in \F} &\tx{ form a basis for } \SS \label{eq:basisclaim}\\
\sum_{f \in \F} \fr{1}{4} f^j f^\ell &= \de^{j\ell} \label{eq:symmetryIdentity}
}

\noindent Viewing the right hand side of \eqref{eq:goodgaJeqn} as a perturbation of $\de^{j\ell}$, we assume $\ga_{(I,f)}^2$ will have the form
\ali{
\ga_{(I,f)}^2(t,x) &= \fr{1}{4} + a_{(I,f)}(t,x) + b_{(I,f)}(t,x) \label{eq:gaIfsqDef} \\
\sum_{f \in \F} a_{(I,f)} f^j f^\ell &= [(\nb \Ga_I)_a^j (\nb \Ga_I)_b^\ell - \mbox{Id}_a^j \mbox{Id}_b^\ell] \de^{ab} \\
\sum_{f \in \F} b_{(I,f)} f^j f^\ell &= (\nb \Ga_I)_a^j (\nb \Ga_I)_b^\ell \varep_I^{ab} \label{eq:bfeqnGaEp}
}
Inverting to solve for $b_{(I,f)}$ and $a_{(I,f)}$ (which is possible by Claim~\eqref{eq:basisclaim}), we have the following bounds
\ali{
\co{a_{(I,f)}} &\leq C_0 (1 + \co{ \nb \Ga_I } )\co{ (\nb \Ga_I) - \mbox{Id}  } \label{eq:aIfbdis} \\
&\stackrel{\eqref{eq:closeToIdGaI}}{\leq} C_0 \th \co{\nb v} \label{eq:nearToIdbd}\\
\co{b_{(I,f)}} &\leq C_0 \co{ \nb \Ga_I }^2 \co{\varep_I} \\
&\stackrel{\eqref{eq:varepC0bd}}{\leq} C_0 K^{-1} \label{eq:bIftermchooseK}
}
We choose $K$ to be an absolute constant such that the last term is bounded by $20^{-1}$.  The right hand side of \eqref{eq:nearToIdbd} is bounded by $C_0 b_0$ where $b_0$ appears in the bound \eqref{eq:b0bd}.  We now choose $b_0$ an absolute constant such that \eqref{eq:nearToIdbd} is at most $20^{-1}$.  With these choices, we can take the positive square root in \eqref{eq:gaIfsqDef} to define $\ga_J$, which then solves \eqref{eq:goodgaJeqn} thanks to \eqref{eq:symmetryIdentity}.  

Note that we have now represented the coefficients  $\ga_J(t,x) = \ga_{(I,f)}(t,x)$ in the form
\ali{
\ga_{(I,f)}(t,x) &= \ga_f(\nb \Ga_I, \varep_I), \label{eq:gaJformula}
}
where $\ga_f : \overline{K} \to \R$ is one of $6$ smooth functions $(\ga_f)_{f \in \F}$ that are defined on an appropriate, compact subset $\overline{K} \subseteq \R^{3 \times 3} \times \SS$ containing the range of $(\nb \Ga_I, \varep_I)$ and that are bounded by $\sup_{\overline{K}} |\ga_f(\cdot)| \leq 1$.  

The construction is now entirely specified except for the definitions of $v_\ep$ and $R_\ep$ and the choice of the large parameter $\la \in \Z$.  The choices of these terms will be governed by the estimates we need to prove.  We start by defining $v_\ep$ and $R_\ep$.

\section{The Coarse Scale Velocity Field and Stress Tensor} \label{sec:mollPrelim}
Following \cite{isett}, the regularization of $R^{j\ell}$ will have a double mollification structure $R_\ep^{j\ell} := \eta_\ep \ast \eta_\ep \ast R^{j\ell}$.  The double-mollification structure will play a role in the advective derivative bounds of Proposition~\ref{prop:DdtRepbds} below.  The mollifying kernel $\eta_\ep$ has support in $|h| \leq \ep$ and satisfies the vanishing moment condition $\int_{\R^3} h^a \eta_\ep(h) dh = 0$ for each co-ordinate $a = 1,2,3$, so that $\co{ R - \eta_\ep \ast R } \leq C_0 \ep^2 \co{\nb^2 R} $ holds\footnote{A proof of this statement, which is well-known, can be found in \cite[Section 14]{isett}.}.  \hltRed{We may also take $\eta$ even and non-negative for convenience.}

The choice of $\ep$ here is dictated by the bound on $R - R_\ep$, which is given by
\ALI{
\co{R - R_\ep} &\leq \co{ R - \eta_\ep \ast R} + \co{ \eta_\ep \ast [ R - \eta_\ep \ast R ] } \\
&\leq C_0 \ep^2 \co{\nb^2 R } \\
\co{R - R_\ep} &\stackrel{\eqref{eq:newRIbdGlue}}{\lsm} \ep^2 \lhxi ~ \Xi^2 e_R
}
(Recall that constants in the $\lsm$ notation can depend on the $C_1$ and $C_\de$ in the hypotheses of Lemma~\ref{lem:convexInt}.)

Take $\ep = \ep_R$ to have the form $\ep_R = c_R N^{-1/2} \Xi^{-1}$, where $c_R$ is a small constant chosen to imply
\ali{
\co{R - R_\ep} &\leq \lhxi \fr{ e_R}{500 N} \label{eq:rminusRepbd}
}
This choice leads to the estimates
\ali{
\co{\nb^k R_\ep} &\lsm_k \lhxi N^{(k - 2)_+/2} \Xi^k e_R, \label{eq:nbkRep}
}
(which are the same as those in \cite{isett} except for the appearance of $\lhxi$).  To prove \eqref{eq:nbkRep}, use \eqref{eq:newRIbdGlue} for $0 \leq k \leq 2$ and write $\nb^k \eta_\ep \ast \eta_\ep \ast R = [\nb^{k-2} \eta_\ep \ast \eta_\ep ] \ast \nb^2 R$ for $k > 2$.

To define $v_\ep$, consider the error term
\ALI{
R_{M, v1}^{j\ell} &= \sum_J (v^j - v_\ep^j) \VR_J^\ell + \VR_J^j (v^\ell - v_\ep^\ell), \\
&= \sum_J [ (v^j - v_\ep^j) v_J^\ell + v_J^j (v^\ell - v_\ep^\ell)]  \psi_f(\la \Ga_I(t,x)), \qquad J = (I, f)
}
which is part of $R_M^{j\ell}$ in the Mollification term \eqref{eq:mollifTerm}.  Note that $v_J^\ell$ is not well-defined until $\nb \Ga_I$ is chosen; however, the term $\ga_J = \ga_f( \nb \Ga_I, \varep_I )$ in $v_J^\ell$ that involves $\nb \Ga_I$ is a priori bounded by a constant $\co{ \ga_J } \leq \sup \ga_f \leq 1$.  From \eqref{eq:eIhalf} and \eqref{eq:vJAnsatz}, we obtain an a priori bound
\ALI{
\sup_J \co{ v_J^\ell } &\lsm \plhxi^{1/2} e_R^{1/2} \\
\co{ R_{M,v1}^{j\ell} } &\lsm \co{ v - v_\ep } \plhxi^{1/2} e_R^{1/2} \\
&\lsm \ep^2 \Xi^2 e_v^{1/2} \plhxi^{1/2} e_R^{1/2},
}
assuming our mollifier satisfies the same vanishing moment condition as in the $R_\ep$ case.  

Take $\ep = \ep_v$ to have the form $\ep_v = c_v N^{-1/2} \Xi^{-1}$, with $c_v$ is a small constant such that
\ali{
\co{ R_{M,v1}^{j\ell} } &\leq \plhxi^{1/2}\fr{e_v^{1/2} e_R^{1/2}}{500 N} \label{eq:RMvepvchoice}
}
This choice of $\ep$ and \eqref{eq:newVelocBdGlue} lead to the bounds
\ali{
\co{\nb^k v_\ep} &\lsm_k N^{(k - 2)_+/2} \Xi^k e_v^{1/2}.  \label{eq:vepBounds}
}
These are the same bounds as those satisfied by the $v_\ep$ in \cite[Section 15]{isett} in the case of frequency energy levels of order $L = 2$ of the Main Lemma in that paper.  This coincidence is due to how we have chosen the same values of $\ep_v$ and $\ep_R$ as in that paper.

With the above choices, we obtain the following estimates for the advective derivative of $R_\ep$.
\begin{prop} \label{prop:DdtRepbds} Let $\Ddt = \pr_t + v_\ep \cdot \nab$ denote the coarse scale advective derivative operator.  Then
\ali{
\co{\nb^k \Ddt R_\ep } &\lsm_k \plhxi^{3}N^{(k - 1)_+/2} \Xi^{k+1} e_v^{1/2} e_R \label{eq:logLossTransRep}
}  
\end{prop}
We deduce this Proposition from \cite[Proposition 18.6]{isett}.  The proof of that Proposition is where the double mollification structure of $R_\ep$ is used to control higher order derivatives.
\begin{proof}  Define $\widetilde{R}^{j\ell} = \plhxi^{-3} R^{j\ell}$.  Then $\wtld{R}^{j\ell}$ satisfies the bounds
\ALI{
\co{\nb^k \wtld{R} } &\lsm \Xi^k e_R, \quad k = 0, 1, 2 \\
\co{\nb^k (\pr_t + v \cdot \nab) \wtld{R}} &\leq \Xi^{k+1} e_v^{1/2} e_R, \quad k = 0, 1
}
These are the same estimates as those assumed on $R^{j\ell}$ in \cite[Definition 10.1]{isett} in the case of frequency energy levels to order $L = 2$.  Moreover, we have chosen the mollification parameters $\ep_v$ and $\ep_R$ to be the same as in \cite[Sections 15, 18.3]{isett} for the case $L=2$.  Thus the estimates of the $L=2$ case of \cite[Proposition 18.6]{isett} (which are the same as \eqref{eq:logLossTransRep} without the logarithmic factor) apply to $\wtld{R}$.  We note that \cite[Proposition 18.6]{isett} involves only the spatial mollification of $R$ (not the mollification in time along the flow) and that the proof does not involve estimates for $\nb p$.
\end{proof}
Having specified the mollification parameters and proven bounds on $v_\ep$ and $R_\ep$, we now turn to estimating the terms in the construction.
\section{Estimates for the Construction}
In Sections~\ref{sec:lowFreqbds}-\ref{sec:concludingProof} below, we prove all the required estimates for Lemma~\ref{lem:convexInt}.  The concluding Section~\ref{sec:concludingProof} reviews where each conclusion of Lemma~\ref{lem:convexInt} has been proven.

\subsection{Estimates for Low-Frequency Terms in the Construction} \label{sec:lowFreqbds}
In this Section we prove estimates for the low frequency terms in the construction, namely $(\nb \Ga_I)$, $(\nb \Ga_I^{-1})$, $\ga_J$, $\varep_I$, and the amplitudes $v_J$ and $\de v_{J,\a\b}$.

\begin{prop} \label{prop:backToLabelsBds}  The following estimates hold for the back-to-labels map
\ali{
\co{\nb^k (\nb \Ga_I) } + \co{\nb^k (\nb \Ga_I^{-1}) } &\lsm_k N^{(k-1)_+/2} \Xi^k, \quad \tx{ for all } k \geq 0 \label{eq:nbknbGaIbd} \\
\co{\nb^k \Ddt(\nb \Ga_I) } + \co{\nb^k \Ddt(\nb \Ga_I^{-1}) } &\lsm_k N^{(k-1)_+/2} \Xi^{k +1} e_v^{1/2} , \quad \tx{ for all } k \geq 0 \label{eq:nbkDdtnbGaIbd}
}
\end{prop}
\begin{proof}
The estimates \eqref{eq:nbknbGaIbd} and \eqref{eq:nbkDdtnbGaIbd} for $\nb \Ga_I$ follow from Propositions 17.3 and 17.5 of \cite{isett}.  There the estimates are performed for a solution $\xi_I$ to
\ali{
(\pr_t + v_\ep \cdot \nab) \xi_I &= 0 \\
\xi_I(t(I),x) &= \hat{\xi}_I(x)
}
where the initial data $\hat{\xi}_I$ is linear with $|\nb \hat{\xi}_I| \leq C_0$.  Each component $\Ga_I^k$ of $\Ga_I$ therefore falls into the framework of those estimates.  Alternatively, one can adapt the proof of those estimates for the system \eqref{eq:evolNabGaI} while modifying the dimensionless energy for $\nb \xi_I$ \cite[Definition 17.1]{isett} to involve Frobenius norms of the matrix $\nb \Ga_I$ and its derivatives.  We note that the proof of these estimates does not require control over $\nb p$, which had been assumed in \cite{isett} for the purpose of controlling second order advective derivatives.

The estimates \eqref{eq:nbknbGaIbd} and \eqref{eq:nbkDdtnbGaIbd} for $(\nb \Ga_I^{-1})$ can be deduced from those for $\nb \Ga_I$ by taking spatial derivatives of the equations
\ALI{
(\nb \Ga_I) (\nb \Ga_I^{-1}) &= \mbox{Id} \\
\Ddt (\nb \Ga_I^{-1}) &= - (\nb \Ga_I^{-1}) [ \Ddt (\nb \Ga_I) ] (\nb \Ga_I^{-1})
}
and using the bound $\co{ \nb \Ga_I^{-1} } \leq C_0$ of Proposition~\ref{prop:backToLabels}.  (The second equation comes from applying $\Ddt$ to the first.)  Alternatively, since the evolution equation \eqref{eq:inverseEqn} for $\nb \Ga_I^{-1}$ has the same form as the equation \eqref{eq:evolNabGaI} for $\nb \Ga_I$ (except for the minus sign and order of matrix multiplication), one obtains \eqref{eq:nbknbGaIbd} and \eqref{eq:nbkDdtnbGaIbd} for $(\nb \Ga_I^{-1})$ by applying the proof of \cite[Propositions 17.3, 17.5]{isett} to equation \eqref{eq:inverseEqn}.
\end{proof}

The remaining low frequency building blocks of the construction can be estimated as follows
\begin{prop} \label{prop:lowFreqEstimates} The following estimates hold for all $k \geq 0$
\ali{
\sup_I \co{\nb^k \varep_I} &\lsm_k N^{(k - 2)_+/2} \Xi^{k} \label{eq:varepIbd1} \\
\sup_I \co{\nb^k \Ddt \varep_I} &\lsm_k \plhxi^2 N^{(k-1)_+/2} \Xi^{k+1} e_v^{1/2} \label{eq:varepIDdtbd} \\
\sup_J \co{\nb^k \ga_J} &\lsm_k N^{(k - 1)_+/2} \Xi^{k} \label{eq:gaJbd} \\
\sup_J \co{\nb^k \Ddt \ga_J} &\lsm_k \plhxi^2 N^{(k-1)_+/2} \Xi^{k+1} e_v^{1/2} \label{eq:DdtgaJbd} \\
\sup_J \co{\nb^k v_J } &\lsm_k \plhxi^{1/2} N^{(k - 1)_+/2} \Xi^{k} e_R^{1/2} \label{eq:nbkAmpbd} \\
\sup_J \co{\nb^k \Ddt v_J } &\lsm_k \plhxi^{5/2} N^{(k - 1)_+/2} \Xi^{k +1} e_v^{1/2} e_R^{1/2} \label{eq:nbDdtAmpbd} \\
\sup_{J, \a\b} \co{\nb^k \de v_{J,\a\b} } &\lsm_k \la^{-1} \plhxi^{1/2} N^{k/2} \Xi^{k+1} e_R^{1/2} \label{eq:nbkdeltavJbd}\\
\sup_{J,\a\b} \co{\nb^k \Ddt \de v_{J,\a\b} } &\lsm_k \la^{-1} \plhxi^{5/2} N^{k/2} \Xi^{k+2} e_v^{1/2} e_R^{1/2} \label{eq:DdtbdDevJ}
}
\end{prop}
In the proof of Proposition~\ref{prop:lowFreqEstimates} below, the implicit constants in the $\lsm$ notation will in general depend on $k$, but we will omit this dependence.
\begin{proof}[Proof of \eqref{eq:varepIbd1}-\eqref{eq:varepIDdtbd}]
Recall that $\varep_I^{j\ell} = - e_I^{-1}(t) R_{I,\ep}^{j\ell}$.  From formula \eqref{eq:eIhalf}, we have that $e_I^{-1}(t)$ is a constant in time on the support of $R_{I,\ep}$, which is contained in $[t(I) - \th/2, t(I)+\th/2]$, and on that domain satisfies a lower bound $\sup_{t \in [t(I) - \th/2, t(I) + \th/2]} |e_I^{-1}(t)| \lsm \plhxi^{-1} e_R^{-1}$.  The bound \eqref{eq:varepIbd1} now follows from \eqref{eq:nbkRep}, and similarly \eqref{eq:varepIDdtbd} follows from \eqref{eq:logLossTransRep} and $\Ddt \varep_I^{j\ell} = - e_I^{-1}(t) \Ddt R_\ep^{j\ell}$.
\end{proof}
\begin{proof}[Proof of \eqref{eq:gaJbd}-\eqref{eq:DdtgaJbd}]
Recall from Section~\ref{sec:algEqn} (in particular Formula~\eqref{eq:gaJformula}) that $\ga_J = \ga_{(I,f)}$ takes the form $\ga_J(t,x) = \ga_f(\nb \Ga_I, \varep_I)$, where $\ga_f$ belongs to a set of six smooth functions whose domains are a compact subset $\overline{K}$ of $\R^{3 \times 3} \times \SS$.  We have already shown in Section~\ref{sec:algEqn} that $\co{\ga_J} \leq \sup_{\overline{K}} \ga_f \leq 1$.  Now let $\nb_{\va}$ be a partial derivative operator of order $|\va| = k \geq 1$.  We will use $\pr_\Ga \ga_f$ to denote a derivative of $\ga_f$ in the $\R^{3 \times 3}$ argument, and $\pr_\varep \ga_f$ to denote a derivative of $\ga_f$ in the $\SS$ argument.  Set $p = (\nb \Ga_I, \varep)$.  Using the chain rule and product rule, we can expand $\nb_{\va} \ga_J = \nb_{\va}[ \ga_f(\nb \Ga_I, \varep) ]$ in the form
\ali{
\nb_{\va}[\ga_f(\nb \Ga_I, \varep_I) ] &= \sum_{m + m' \leq k} \sum_{\vcb, \vcc} \pr_\Ga^{m} \pr_\varep^{m'} \ga_f(p) \prod_{i = 1}^m \nb_{\vcb_i} (\nb \Ga_I) \prod_{j=1}^{m'} \nb_{\vcc_j} \varep \notag
}
The innermost sum is restricted to certain multi-indices indices such that $\sum_{i=1}^m |\vcb_i| + \sum_{j=1}^{m'} |\vcc_j| = k$.  We now estimate this term by 
\ALI{
\co{\nb_{\va} \ga_J} &\lsm \sum_{m + m' \leq k} \left(\prod_{i = 1}^m N^{(|\vcb_i| - 1)_+/2} \Xi^{|\vcb_i|} \right) \left( \prod_{j=1}^{m'} N^{(|\vcc_j| - 1)_+/2} \Xi^{|\vcc_j|} \right) \\
&\stackrel{\eqref{ineq:counting}}{\lsm} N^{(k - 1)_+/2} \Xi^k
}
Similarly, for $\Ddt \ga_J = \pr_\Ga \ga_f \Ddt (\nb \Ga_I) + \pr_\varep \ga_f \Ddt \varep$, we can express
\ali{
\nb_{\va} \Ddt \ga_J &= \sum_{m + m' \leq k} \sum_{\vcb, \vcc, \vce} \pr_\Ga^{m + 1} \pr_\varep^{m'} \ga_f(p) \left(\prod_{i = 1}^m \nb_{\vcb_i} (\nb \Ga_I) \prod_{j=1}^{m'} \nb_{\vcc_j} \varep \right) \nb_{\vce} \Ddt (\nb \Ga_I) \label{eq:nbvaDtgaJGaItrm}\\
&+ \sum_{m + m' \leq k} \sum_{\vcb, \vcc, \vce} \pr_\Ga^{m} \pr_\varep^{m'+1} \ga_f(p) \left(\prod_{i = 1}^m \nb_{\vcb_i} (\nb \Ga_I) \prod_{j=1}^{m'} \nb_{\vcc_j} \varep \right) \nb_{\vce} \Ddt \varep,\label{eq:nbvaDtgaJvreptrm}
}
where the summation runs over certain multi-indices with $\sum_{i=1}^m |\vcb_i| + \sum_{j=1}^{m'} |\vcc_j| + |\vce| = k$ and empty products are equal to $1$.  These terms can be bounded using \eqref{eq:nbknbGaIbd}-\eqref{eq:nbkDdtnbGaIbd} and \eqref{eq:varepIbd1}-\eqref{eq:varepIDdtbd} by
\ALI{
\co{\eqref{eq:nbvaDtgaJGaItrm}} &\lsm \sum_{m + m' \leq k} \sum_{\vcb, \vcc, \vce} \left(\prod_{i = 1}^m [N^{(|\vcb_i| - 1)_+/2} \Xi^{|\vcb_i|}]  \prod_{j=1}^{m'} [N^{(|\vcc_j| -1)_+/2} \Xi^{|\vcc_j|} ] \right) N^{(|\vce| - 1)_+/2} \Xi^{|\vce| + 1} e_v^{1/2} \\
&\stackrel{\eqref{ineq:counting}}{\lsm} N^{(k -1)_+} \Xi^{k+1} e_v^{1/2} \\
\co{\eqref{eq:nbvaDtgaJvreptrm}} &\lsm \sum_{m + m' \leq k} \sum_{\vcb, \vcc, \vce} \left(\prod_{i = 1}^m [N^{(|\vcb_i| - 1)_+/2} \Xi^{|\vcb_i|}]  \prod_{j=1}^{m'} [N^{(|\vcc_j| -1)_+/2} \Xi^{|\vcc_j|} ] \right) N^{(|\vce| - 1)_+/2} \Xi^{|\vce| + 1} e_v^{1/2} \plhxi^2  \\
&\stackrel{\eqref{ineq:counting}}{\lsm} \plhxi^2 N^{(k -1)_+/2} \Xi^{k+1} e_v^{1/2}
}
This estimate concludes the proof of \eqref{eq:gaJbd}-\eqref{eq:DdtgaJbd}.
\end{proof}
\begin{proof}[Proof of \eqref{eq:nbkAmpbd}-\eqref{eq:nbDdtAmpbd}]
Recall from \eqref{eq:vJAnsatz} that $v_J^\ell = e_I^{1/2}(t) \ga_J(t,x) (\nb \Ga_I^{-1})_a^\ell f^a$.  Let $\nb_{\va}$ be a partial derivative of order $k$.  Then
\ali{
\nb_{\va} v_J^\ell &= \sum_{|\vcb| + |\vcc| = k} e_I^{1/2}(t) c_{\va, \vcb, \vcc} \nb_{\vcb} \ga_J \nb_{\vcc} (\nb \Ga_I^{-1})_a^\ell f^a \notag \\
\nb_{\va} \Ddt v_J^\ell &= \sum_{|\vcb| + |\vcc| = k} \pr_t e_I^{1/2}(t) c_{\va, \vcb, \vcc} \nb_{\vcb} \ga_J \nb_{\vcc} (\nb \Ga_I^{-1})_a^\ell f^a \label{eq:ddteItermvJ} \\
&+ \sum_{|\vcb| + |\vcc| = k}  e_I^{1/2}(t) c_{\va, \vcb, \vcc} \nb_{\vcb} \Ddt \ga_J \nb_{\vcc} (\nb \Ga_I^{-1})_a^\ell f^a \label{eq:DdtgaJtermvJ}\\
&+ \sum_{|\vcb| + |\vcc| = k} e_I^{1/2}(t) c_{\va, \vcb, \vcc} \nb_{\vcb} \ga_J \nb_{\vcc} \Ddt (\nb \Ga_I^{-1})_a^\ell f^a \label{eq:DdtGaIinvtermvJ}
}
From formula \eqref{eq:eIhalf} and \eqref{ineq:thBound}, we have the following bounds on $e_I^{1/2}(t)$:
\ali{
\begin{split}
\co{ e_I^{1/2}(t)} &\lsm \plhxi^{1/2} e_R^{1/2} \\
\co{ \pr_t e_I^{1/2}(t) } \lsm \th^{-1} &\plhxi^{1/2} e_R^{1/2} \lsm \plhxi^{5/2} \Xi e_v^{1/2} e_R^{1/2} \label{eq:prtbdeI}
\end{split}
}
Applying these bounds and those of \eqref{eq:varepIbd1}-\eqref{eq:DdtgaJbd} gives
\ALI{
\co{\nb_{\va} v_J^\ell} &\lsm \co{e_I^{1/2}} \sum_{|\vcb| + |\vcc| = k} \co{\nb_{\vcb} \ga_J} \co{\nb_{\vcc} (\nb \Ga_I^{-1})_a^\ell} \\
&\lsm \co{e_I^{1/2}} \sum_{|\vcb| + |\vcc| = k} [N^{(|\vcb| - 1)_+/2} \Xi^{|\vcb|}] [N^{(|\vcc| - 1)_+/2} \Xi^{|\vcc|}] \\
&\lsm \plhxi^{1/2} e_R^{1/2} N^{(k -1)_+/2} \Xi^k  
}
We similarly obtain the advective derivative estimates using \eqref{eq:nbknbGaIbd}-\eqref{eq:nbkDdtnbGaIbd}, \eqref{eq:gaJbd}-\eqref{eq:DdtgaJbd} and \eqref{eq:prtbdeI}.
\ALI{
\co{\nb_{\va} \eqref{eq:ddteItermvJ}} &\lsm \sum_{|\vcb| + |\vcc| = k} \co{ \pr_t e_I^{1/2}} \co{\nb_{\vcb} \ga_J} \co{\nb_{\vcc} (\nb \Ga_I^{-1})_a^\ell} \\
&\lsm \plhxi^{5/2} N^{(k -1)_+/2} \Xi^k e_R^{1/2} \\
\co{\nb_{\va} \eqref{eq:DdtgaJtermvJ}} &\lsm \sum_{|\vcb| + |\vcc| = k} \co{e_I^{1/2}} \co{\nb_{\vcb} \Ddt \ga_J} \co{\nb_{\vcc}(\nb \Ga_I^{-1})} \\
\co{\nb_{\va}\eqref{eq:DdtGaIinvtermvJ}}&\lsm \sum_{|\vcb| + |\vcc| = k} \co{e_I^{1/2}} \co{\nb_{\vcb} \Ddt (\nb \Ga_I^{-1}) } \co{\nb_{\vcc}\ga_J}  \\
\co{\nb_{\va} \eqref{eq:DdtgaJtermvJ}} + \co{\nb_{\va}\eqref{eq:DdtGaIinvtermvJ}}&\lsm \plhxi^{1/2} e_R^{1/2} \sum_{|\vcb| + |\vcc| = k} \plhxi^2 [N^{(|\vcb| - 1)_+/2} \Xi^{|\vcb|}] [N^{(|\vcc|-1)_+/2} \Xi^{|\vcc|}] \\
&\lsm \plhxi^{5/2} N^{(k - 1)_+/2} \Xi^k e_R^{1/2}
}
\end{proof}
\begin{proof}[Proof of \eqref{eq:nbkdeltavJbd}-\eqref{eq:DdtbdDevJ}]  Recalling formula \eqref{eq:devJform} and commuting in the advective derivative, we have
\ali{
\de v_{J,\a\b}^\ell &= \la^{-1}\nb_a[(\nb \Ga_I^{-1})^a_\a (\nb \Ga_I^{-1})_\b^\ell e_I^{1/2}(t) \ga_J(t,x) ] \notag\\
\Ddt \de v_{J,\a\b}^\ell &= \la^{-1}\nb_a[ \Ddt[(\nb \Ga_I^{-1})^a_\a (\nb \Ga_I^{-1})_\b^\ell \ga_J(t,x) ] e_I^{1/2}(t)   ] \label{eq:DdtnbGaIGaIdevJ} \\
&+ \la^{-1}\nb_a [ (\nb \Ga_I^{-1})^a_\a (\nb \Ga_I^{-1})_\b^\ell \ga_J(t,x) \pr_t e_I^{1/2}(t) ] \label{eq:prteItermdevJ} \\
&- \la^{-1} \nb_a v_\ep^i \nb_i[ (\nb \Ga_I^{-1})^a_\a (\nb \Ga_I^{-1})_\b^\ell e_I^{1/2}(t) \ga_J(t,x) ] \label{eq:commutermdevJDdt}
}
Let $\nb_{\va}$ be a partial derivative of order $|\va| = k$.  Then from the product rule
\ali{
\co{ \nb_{\va} \de v_{J,\a\b}^\ell } &\lsm \la^{-1} \co{e_I^{1/2}} \sum_{|\va_1| + |\va_2| + |\va_3| = k +1 } \co{ \nb_{\va_1} (\nb \Ga_I^{-1})} \co{ \nb_{\va_2} (\nb \Ga_I^{-1})} \co{\nb_{\va_{3}} \ga_J} \notag \\
&\lsm \la^{-1} \co{e_I^{1/2}} \sum_{|\va_1| + |\va_2| + |\va_3| = k +1 }  \prod_{i=1}^3 [N^{(|\va_i| - 1)_+/2} \Xi^{|\va_i|}] \notag \\
&\stackrel{\eqref{ineq:counting}}{\lsm} \la^{-1} \co{e_I^{1/2}} N^{(k + 1 - 1)_+/2} \Xi^{k+1} \notag \\
&\lsm \la^{-1} \plhxi^{1/2} N^{k/2} \Xi^{k+1} e_R^{1/2}  \notag
}
We similarly estimate the terms \eqref{eq:DdtnbGaIGaIdevJ}-\eqref{eq:commutermdevJDdt} by applying Proposition~\ref{prop:backToLabelsBds} and the bounds \eqref{eq:gaJbd}-\eqref{eq:DdtgaJbd} 
\ALI{
\co{\nb_{\va} \eqref{eq:prteItermdevJ}} &\lsm \la^{-1} \co{\pr_t e_I^{1/2}} \sum_{|\va_1| + |\va_2| + |\va_3| = k +1 } \co{ \nb_{\va_1} (\nb \Ga_I^{-1})} \co{ \nb_{\va_2} (\nb \Ga_I^{-1})} \co{\nb_{\va_{3}} \ga_J} \\
&\stackrel{\eqref{eq:prtbdeI}}{\lsm} \la^{-1} \plhxi^{5/2} \Xi e_v^{1/2} e_R^{1/2}  N^{k/2} \Xi^{k} \\
\co{\nb_{\va} \eqref{eq:DdtnbGaIGaIdevJ} } &\lsm \la^{-1} \co{e_I^{1/2}}  \sum_{|\va_1| + |\va_2| + |\va_3| = k +1 } \co{\nb_{\va_1} \Ddt (\nb \Ga_I^{-1})} \co{ \nb_{\va_2} (\nb \Ga_I^{-1})} \co{\nb_{\va_{3}} \ga_J} \\
&+\la^{-1} \co{e_I^{1/2}}  \sum_{|\va_1| + |\va_2| + |\va_3| = k +1 } \co{\nb_{\va_1}(\nb \Ga_I^{-1})} \co{ \nb_{\va_2} (\nb \Ga_I^{-1})} \co{\nb_{\va_{3}} \Ddt \ga_J}  \\
&\lsm \la^{-1} \co{e_I^{1/2}} \plhxi^2 \Xi e_v^{1/2} \prod_{i=1}^3 N^{(|\va_i| -1)_+/2} \Xi^{|\va_i|}  \\
&\lsm \la^{-1} \plhxi^{1/2 + 2} e_R^{1/2} \Xi e_v^{1/2} N^{(k + 1 - 1)_+/2} \Xi^{k+1} \\
&\lsm \la^{-1} \plhxi^{5/2} N^{k/2} \Xi^{k+2} e_v^{1/2} e_R^{1/2}
}
For the commutator term, we sum over multi-indices with $|\va_0| + \ldots + |\va_3| = k+1$ and $|\va_0| \leq k$
\ALI{
\co{\nb_{\va} \eqref{eq:commutermdevJDdt} } &\lsm \la^{-1} \co{e_I^{1/2}} \sum_{\va_0, \ldots, \va_3 } \co{\nb_{\va_0} \nb_a v_\ep} \co{ \nb_{\va_1} (\nb \Ga_I^{-1})} \co{ \nb_{\va_2} (\nb \Ga_I^{-1})} \co{\nb_{\va_{3}} \ga_J}  \notag \\
\co{\nb_{\va} \eqref{eq:commutermdevJDdt} } &\lsm \la^{-1} \co{e_I^{1/2}} \Xi e_v^{1/2} \sum_{\va_0, \ldots, \va_3 } \prod_{i=0}^3 N^{(|\va_i| - 1)_+/2} \Xi^{|\va_i|} \\
&\lsm \la^{-1} \plhxi^{1/2} e_R^{1/2} \Xi e_v^{1/2} N^{k/2} \Xi^{k+1}
}
This bound completes the proof of \eqref{eq:nbkdeltavJbd}-\eqref{eq:DdtbdDevJ}.
\end{proof}

Our next task will be to estimate high frequency terms, including the correction $V_J$.

\subsection{Bounds on the Correction} \label{sec:correctionBounds}
We now proceed to estimate the components of the high frequency correction $V^\ell = \sum_J \VR_J^\ell + \de V_J^\ell$ defined in Section~\ref{sec:Ansatz}.  In the process, we prove the estimate \eqref{eq:cobdcorrect}, and verify the estimates implied by \eqref{eq:newFrEnLvls} and \eqref{eq:frEnVelocbds} for the new velocity field.

The bounds for high frequency term involve the choice of the parameter $\la$, which we now describe.  Consistent with the frequency level in \eqref{eq:newFrEnLvls}, we assume that $\la$ will take the form
\ali{
\la &= B_\la N \Xi \notag
}
The parameter $B_\la$ is the last parameter that remains to be chosen.  Unlike all of the constants chosen previously, the choice of $B_\la$ will depend on the parameter $\eta$ in the assumptions of Lemma~\ref{lem:convexInt}. 

To be more precise, there will be a large constant $\overline{B}_\la$ that remains to be chosen depending on $(C_0, C_1, \de, \eta)$, and $B_\la$ will be chosen from the interval $B_\la \in [\overline{B}_\la, 2\overline{B}_\la]$ in order to ensure that $\la \in \Z$ is an integer.  Since $B_\la$ and $\overline{B}_\la$ are equal to within a factor of $2$, we can ignore the distinction between them and think of $B_\la$ as the last constant parameter that remains to be chosen.


The bounds we obtain for the correction are as follows
\begin{prop}[Correction Estimates] \label{correctEsts}  The following bounds hold for $0 \leq k \leq 3$
\ali{
\sup_J \co{\nb^k \VR_J} &\lsm (B_\la N \Xi)^k \plhxi^{1/2} e_R^{1/2} \label{eq:mainCorrectionBound}\\
\sup_J \co{\nb^k \de V_J} &\lsm (B_\la N \Xi)^{k-1} \Xi \plhxi^{1/2} e_R^{1/2} \label{eq:lowerCorrectionBound} \\
\co{\nb^k V^\ell } &\lsm (B_\la N \Xi)^k \plhxi^{1/2} e_R^{1/2} \label{eq:correctionBoundToPass} \\
\suppt V &\subseteq \bigcup_I [t(I) - \th, t(I) + \th] \label{eq:suppPropV}
}
Furthermore, the bound \eqref{eq:cobdcorrect} holds, and the estimates implied by \eqref{eq:newFrEnLvls} and \eqref{eq:frEnVelocbds} hold for $v_1 = v + V$.  
\end{prop}
\begin{proof}
Let $\nb_{\va}$ be a partial derivative of order $0 \leq |\va| \leq 3$.  Recall from \eqref{eq:leadingTerm} that $\VR_J^\ell = v_J^\ell \psi_f(\la \Ga_I)$.  First observe that
\ali{
\nb_{\va} \VR_J &= \sum_{0 \leq m \leq |\va|} \sum_{\vcb, \vcc} \nb_{\vcb} v_J \pr^m\psi_f(\la \Ga_I) \la^m \prod_{i=1}^m \nb_{\vcc_i}(\nb \Ga_I) \notag
}
where the sum ranges over a set of multi-indices such that $|\vcb| + m + \sum_{i=1}^m |\vcc_i| = |\va|$, and the empty product equals $1$ in the case $m = 0$.  Using \eqref{eq:nbknbGaIbd} and \eqref{eq:nbkAmpbd} we obtain
\ALI{
\co{\nb_{\va} \VR_J} &\lsm \plhxi^{1/2} e_R^{1/2}\sum_{0 \leq m \leq |\va|} \sum_{\vcb, \vcc} N^{(|\vcb| - 1)_+/2} \Xi^{|\vcb|} \la^m \prod_{i=1}^m [N^{(|\vcc_i|-1)_+/2} \Xi^{|\vcc_i|}] \\
&\lsm \plhxi^{1/2} e_R^{1/2}\sum_{0 \leq m \leq |\va|} \sum_{\vcb, \vcc} N^{(|\va| - m - 1)_+/2} \Xi^{|\va| - m} \la^m \\
&\lsm \plhxi^{1/2} e_R^{1/2}\sum_{0 \leq m \leq |\va|} \sum_{\vcb, \vcc} N^{(|\va| - m - 1)_+/2} \Xi^{|\va| - m} (B_\la N \Xi)^m \\
&\lsm \plhxi^{1/2} e_R^{1/2} (B_\la N \Xi)^{|\va|}
}
Note that the worst terms occur when all of the derivatives in $\nb_{\va}[ v_J^\ell \psi_f(\la \Ga_I) ]$ fall on the high frequency function $\psi_f(\la \Ga_I)$, in which case each derivative costs a factor of $B_\la N \Xi$.  This case corresponds to $m = |\va|$ and $|\vcc_i| = 0$ for all $i = 1, \ldots, m$ in the above estimate.  

Recalling formula \eqref{eq:deVJdevJ}, we treat $\de V_J^\ell = \de v_{J,\a\b}^\ell \Om_f^{\a\b}(\la \Ga_I) $ similarly
\ALI{
\nb_{\va} \de V_J &= \sum_{0 \leq m \leq |\va|} \sum_{\vcb, \vcc} \nb_{\vcb} \de v_{J,\a\b} \pr^m\Om_f^{\a\b}(\la \Ga_I) \la^m \prod_{i=1}^m \nb_{\vcc_i}(\nb \Ga_I) \\
\co{\nb_{\va} \de V_J} &\lsm \la^{-1} \Xi e_R^{1/2} \sum_{0 \leq m \leq |\va|} \sum_{\vcb, \vcc} [N^{|\vcb|/2} \Xi^{|\vcb|} ] \la^m \prod_{i=1}^m [N^{(|\vcc_i|-1)_+/2} \Xi^{|\vcc_i|}] \\
&\lsm \la^{-1} \Xi e_R^{1/2} \sum_{0 \leq m \leq |\va|} (B_\la N \Xi)^{m} N^{(|\va|-m)/2} \Xi^{(|\va| - m)} \\
&\lsm \la^{-1} \Xi e_R^{1/2} (B_\la N \Xi)^{|\va|} \\
&\lsm (B_\la N \Xi)^{|\va| - 1} \Xi e_R^{1/2}
}
As before, the worst terms in $\nb_{\va}[\de v_{J,\a\b}^\ell \Om_f^{\a\b}(\la \Ga_I)]$ occur when every derivative hits the high-frequency function $\Om_f^{\a \b}(\la \Ga_I)$, each time costing a factor of $\la$.  

The bound \eqref{eq:correctionBoundToPass} follows by adding \eqref{eq:mainCorrectionBound}-\eqref{eq:lowerCorrectionBound}, summing over $V = \sum_J \VR_J + \de V_J$ and noting that at most $|\F| = 6$ of the $\VR_J$ and $\de V_J$ are nonzero at any given time, and that the $\VR_J$ contribute the dominant term.  

The support property \eqref{eq:suppPropV} is clear from the formula \eqref{eq:bigDivFormula} using \eqref{eq:eIhalf}, which implies \eqref{eq:suppteFnctn}.

To check that the new bounds \eqref{eq:frEnVelocbds} are satisfied for \eqref{eq:newFrEnLvls}, observe that for all $1 \leq |\va| \leq 3$ 
\ali{
\nb_{\va} v_1 &= \nb_{\va} v + \nb_{\va} V \notag \\
\co{ \nb_{\va} v_1} &\lsm \Xi^{|\va|} e_v^{1/2} + \plhxi^{1/2} (B_\la N \Xi)^{|\va|} e_R^{1/2} \notag \\
&\lsm  (B_\la N \Xi)^{|\va|} (\lhxi \, \, e_R)^{1/2} , \label{eq:newBdv1FrEn}
}
where we used that $N \geq (e_v/e_R)^{1/2}$.  Note that \eqref{eq:newBdv1FrEn} coincides with the bounds required in \eqref{eq:mainCorrectionBound}-\eqref{eq:lowerCorrectionBound}.

To check that \eqref{eq:cobdcorrect} holds, note that \eqref{eq:cobdcorrect} is equivalent to the $k = 0$ case of \eqref{eq:correctionBoundToPass}.  The bound here is independent of $B_\la$ (and independent of $\eta$).
\end{proof}
We now begin estimating the error terms, beginning with the terms that do not involve solving the divergence equation.

\subsection{Stress Terms not Involving the Divergence Equation }
In this Section, we begin estimating the terms in the new stress $R_1$ determined by \eqref{eq:R1decomp}.  We start with the terms \eqref{eq:mollifTerm} and \eqref{eq:lowOrderProducterm}, which do not require solving the divergence equation.  

\begin{prop} \label{prop:noInvDivbds} There exists a constant $\overline{B}_\la$ such that for all $B_\la \geq \overline{B}_\la$ the following bounds hold
\ali{
\co{R_M} + \co{R_S} &\leq \lhxi \fr{e_v^{1/2} e_R^{1/2}}{10 N} \label{eq:mollStressNewBd} \\
\co{\nb_{\va} R_M} + \co{\nb_{\va} R_S} &\lsm (B_\la N \Xi)^{|\va|} \lhxi \fr{e_v^{1/2}e_R^{1/2}}{N}, \quad 1 \leq |\va| \leq 3 \label{eq:mollStressNewBdNba} \\
\suppt R_M \cup \suppt R_S &\subseteq \bigcup_I [t(I) - \th, t(I)+\th] \label{eq:suppRMRS}
}
\end{prop} 
\begin{proof}[Proof of \eqref{eq:mollStressNewBd}-\eqref{eq:suppRMRS} for $R_S$] 
From the formula~\eqref{eq:lowOrderProducterm} and Proposition~\ref{correctEsts} we have 
\ALI{
R_S^{j\ell} &= \sum_{J, K \in \Z \times \F} \de V_J^j \VR_K^\ell + \VR_J^j \de V_K^\ell + \de V_J^j \de V_K^\ell \\
\co{R_S} &\lsm [ (B_\la N \Xi)^{-1} \plhxi^{1/2} \Xi e_R^{1/2} ] [\plhxi^{1/2} e_R^{1/2} ] + [ (B_\la N \Xi)^{-1} \plhxi^{1/2} \Xi e_R^{1/2} ]^2 \\
&\lsm (B_\la N \Xi)^{-1} \Xi \lhxi  e_R
}
Here we used that the number of nonzero terms is bounded by $|\F|^2 = 36$ at any given time.  For $B_\la$ a sufficiently large constant, this term is bounded by $\lhxi \fr{e_R}{1000 N}$.  The term $R_M^{j\ell}$ will obey a similar bound, from which \eqref{eq:mollStressNewBd} will follow.  As for the derivatives of $R_S$, again the term involving $\VR_J$ dominates the term quadratic in $\de V_J$, and we have
\ALI{
\co{\nb_{\va} R_S^{j\ell}} &\lsm \sum_{|\vcb| + |\vcc| = |\va| } \sup_J \co{\nb_{\vcb} \de V_J} \sup_J ( \co{\nb_{\vcc} \VR_J} + \co{\nb_{\vcc} \de V_J} ) \\
&\lsm \sum_{|\vcb| + |\vcc| = |\va| } [ (B_\la N \Xi)^{|\vcb|-1} \Xi \plhxi^{1/2} e_R^{1/2}] [(B_\la N \Xi)^{|\vcc|} \plhxi^{1/2} e_R^{1/2}] \\ 
&\lsm \lhxi ~ (B_\la N \Xi)^{|\va| - 1} \Xi e_R.
}
This bound suffices for \eqref{eq:mollStressNewBdNba}.  The support property \eqref{eq:suppRMRS} follows from \eqref{eq:deVJdevJ}-\eqref{eq:devJform} and \eqref{eq:suppteFnctn}.
\end{proof}
\begin{proof}[Proof of \eqref{eq:mollStressNewBd}-\eqref{eq:suppRMRS} for $R_M$]
We recall the following formula from \eqref{eq:mollifTerm}
\ali{
R_M^{j\ell} &= (v^j - v_\ep^j) V^\ell + V^j ( v^\ell - v_\ep^\ell) + (R^{j\ell} - R_\ep^{j\ell}) \label{eq:mollTerms}
}
In Lines~\eqref{eq:rminusRepbd} and \eqref{eq:RMvepvchoice}, the parameters $\ep_R$ and $\ep_v$ for $R_\ep$ and $v_\ep$ were chosen such that
\ali{
\co{R - R_\ep} &\leq \lhxi \fr{ e_R}{500 N} \notag \\
\sum_J \co{(v^j - v_\ep^j) \VR_J^\ell + \VR_J^j ( v^\ell - v_\ep^\ell) } &\leq \plhxi^{1/2}\fr{e_v^{1/2} e_R^{1/2}}{500 N} \label{eq:1over500bd} \\
\co{v^\ell - v_\ep^\ell} &\lsm \fr{e_v^{1/2}}{N} \label{eq:vMinusvEpbd}
}
To complete the proof of \eqref{eq:mollStressNewBd}, we use \eqref{eq:lowerCorrectionBound} to bound the remaining lower order term by 
\ALI{
R_{M, v2}^{j\ell} &= \sum_{J \in \Z \times \F} (v^j - v_\ep^j) \de V_J^\ell + \de V_J^j ( v^\ell - v_\ep^\ell) \\
\co{R_{M, v2}^{j\ell}} &\lsm \fr{e_v^{1/2}}{N} [(B_\la N \Xi)^{-1} \Xi \plhxi^{1/2} e_R^{1/2}]  
}
For $B_\la$ suffiently large, this term also satisfies the estimate \eqref{eq:1over500bd}.  Thus the $C^0$ estimate \eqref{eq:mollStressNewBd} holds using also our previous bound $\co{R_S} \leq \lhxi \fr{e_R}{1000 N}$.  We now move on to proving \eqref{eq:mollStressNewBdNba}.   

Let $\nb_{\va}$ be a partial derivative operator of order $1 \leq |\va| \leq 3$.  We start by estimating 
\ALI{
\co{\nb_{\va} (R^{j\ell} - R_\ep^{j\ell} )} &\leq \co{\nb_{\va} R^{j\ell}} + \co{\nb_{\va} R_\ep^{j\ell}} \\
&\stackrel{\eqref{eq:R1Est}}{\lsm} \lhxi ~ (e_v/e_R)^{(|\va| - 2)_+/2} \Xi^{|\va|} e_R + \lhxi ~ N^{(|\va| - 2)_+/2} \Xi^{|\va|} e_R \\
N \geq (e_v/e_R)^{1/2} \Rightarrow \quad &\lsm \lhxi ~ N^{(|\va| - 2)_+} \Xi^{|\va|} e_R \\
|\va| \geq 1 \Rightarrow \quad&\lsm \lhxi N^{|\va|} \Xi^{|\va|} \fr{e_R}{N}
}
Let $R_{M,v}^{j\ell} = R_{M,v1}^{j\ell} + R_{M,v2}^{j\ell}$ denote the term in \eqref{eq:mollTerms} involving $(v - v_\ep)$.  By \eqref{eq:vMinusvEpbd} and \eqref{eq:correctionBoundToPass}, we obtain
\ali{
\co{ \nb_{\va} R_{M,v}^{j\ell}} &\leq \co{v - v_\ep} \co{\nb_{\va} V} + \sum_{|\vcb| + |\vcc| = |\va|} ( \co{\nb_{\vcb} v} + \co{\nb_{\vcb} v_\ep} ) \co{\nb_{\vcc} V} 1_{1 \leq |\vcb| \leq 3} \notag \\
&\lsm \fr{e_v^{1/2}}{N} \plhxi^{1/2} (B_\la N \Xi)^{|\va|} e_R^{1/2} \label{eq:vminusvepnbVterm} \\
&+ \sum_{|\vcb| + |\vcc| = |\va|} ( \Xi^{|\vcb|} e_v^{1/2} + N^{(|\vcb| - 2)_+/2} \Xi^{|\vcb|} e_v^{1/2} ) [ (B_\la N \Xi)^{|\vcc|} \plhxi^{1/2} e_R^{1/2} ] 1_{|\vcb| \geq 1}  \label{eq:summTerm} \\
\eqref{eq:summTerm} &\leq B_\la^{|\va|} \Xi^{|\va|} \plhxi^{1/2} e_v^{1/2} e_R^{1/2} \sum_{|\vcb| + |\vcc| = |\va|} N^{(|\vcb| - 2)_+/2 + |\vcc|} 1_{|\vcb| \geq 1} \label{eq:vminvepSumVtrmRmv} \\
\eqref{eq:summTerm} &\lsm (B_\la N \Xi)^{|\va|} \plhxi^{1/2} \fr{e_v^{1/2} e_R^{1/2}}{N} \label{eq:nonvvepnbVtermRmv}
}
In the last line, we note that the largest term in \eqref{eq:vminvepSumVtrmRmv} occurs when $|\vcb| = 1$ and $|\vcc| = |\va| - 1 \geq 0$.  Combining \eqref{eq:vminusvepnbVterm} and \eqref{eq:nonvvepnbVtermRmv} we obtain \eqref{eq:mollStressNewBdNba} for $R_{M,v}$, which finishes the proof of \eqref{eq:mollStressNewBdNba} for $R_M$.

\end{proof}

\subsection{Stress Terms Involving the Divergence Equation } \label{sec:divEqnTerms}
In this Section, we bound the terms \eqref{eq:transTerm} and \eqref{eq:highFreqTerm} that compose the remaining part of the new stress $R_1$ defined in \eqref{eq:R1decomp}.  The bound we obtain is the following:
\begin{prop} \label{prop:divStressBds} There exists a constant $\overline{B}_\la$ (depending on $C_0, C_1, \de, \eta$) such that for all $B_\la \geq \overline{B}_\la$, there exist symmetric tensors $R_T^{j\ell}$ and $R_H^{j\ell}$ that solve
\ali{
\nb_j R_T^{j\ell} &=  \pr_t V^\ell + v_\ep^j \nb_j V^\ell + V^j \nb_j v_\ep^\ell \label{eq:transTerm2}\\
\nb_j R_H^{j\ell} &= \nb_j \left[ \sum_{J \in \Z \times \F} \VR_J^j \VR_J^\ell + P \de^{j\ell} + R_\ep^{j\ell} \right] \label{eq:highFreqTerm2}
}
and satisfy the following bounds
\ali{
\co{R_T} + \co{R_H} &\leq \plhxi^{5/2} \fr{e_v^{1/2} e_R^{1/2}}{20N} \label{eq:needc0bdRTRH} \\
\co{\nb^k R_T} + \co{\nb^k R_H} &\lsm (B_\la N \Xi)^k \plhxi^{5/2} \fr{e_v^{1/2} e_R^{1/2}}{N}, \quad 1 \leq k \leq 3 \label{eq:nbknewRTRH} \\
\suppt R_T \cup \suppt R_H &\subseteq \bigcup_{I} [t(I) - \th, t(I)+\th] \notag
}
\end{prop}
The key to the estimates will be a set of crucial cancellations that arise thanks to the Ansatz of Section~\ref{sec:Ansatz} combined with Proposition~\ref{prop:nonstatPhase} below (which is inspired by calculations in \cite{danSze}), which gains cancellation while inverting the divergence equation for high frequency right hand sides.
\begin{proof}  
From Section~\ref{sec:Ansatz}, we can expand the correction $V^\ell$ in the form
\ali{
V^\ell &= \sum_J \left( v_J^\ell \psi_f(\la \Ga_I) + \de v_{J,\a\b}^\ell \Om_f^{\a\b}(\la \Ga_I) \right) \notag
} 
Substituting into the Transport term \eqref{eq:transTerm2} and using
\ALI{
(\pr_t + v_\ep^j \nb_j) [\psi_f(\la \Ga_I)] &= 0 \\
(\pr_t + v_\ep^j \nb_j) [\Om_f^{\a \b}(\la \Ga_I)] &= 0,
}
we can write the Transport term in the form
\ali{
\eqref{eq:transTerm2} &= \sum_J \left(u_{T,J}^\ell \psi_f(\la \Ga_I) + u_{T,J\a\b}^\ell \Om_f^{\a\b}(\la \Ga_I) \right)\label{eq:transTermtwoParts}\\
u_{T,J}^\ell &= (\Ddt v_J^\ell + v_J^j \nb_j v_\ep^\ell) \label{eq:utJelldef} \\
u_{T,J\a\b}^\ell &= (\Ddt \de v_{J,\a \b}^\ell + \de v_{J,\a\b}^j \nb_j v_\ep^\ell) \label{eq:utJalbeelldef}
}
For the term \eqref{eq:highFreqTerm2}, we recall \eqref{eq:highFreqReduced} and use the fact in \eqref{eq:keyCancel} that $v_J^j \nb_j[ \psi_f(\la \Ga_I) ] = 0$ to obtain
\ali{
\eqref{eq:highFreqTerm2} &= \sum_{J \in \Z \times \F} \nb_j[ v_J^j v_J^\ell ( \psi_f^2(\la \Ga_I) - 1) ]  \notag \\
&= \sum_J \nb_j[v_J^j v_J^\ell] ( \psi_f^2(\la \Ga_I) - 1) \label{eq:highFreqIntReady}
}
Oberserve that the terms \eqref{eq:transTermtwoParts} and \eqref{eq:highFreqIntReady} all have the form $u^\ell \om(\la \Ga_I)$ for some smooth $u^\ell$ and some smooth $\om : \T^3 \to \R$ that has integral $0$ (\eqref{eq:int0psif}, \eqref{eq:int1psifsq} and \eqref{eq:Omfint0}).  
Applying Proposition~\ref{prop:nonstatPhase} below, we have that for any $D \in \Z_+$ 
there exist tensors $R_{T,J}^{j\ell}$, $R_{T,J\a\b}^{j\ell}$ and $R_{H,J}^{j\ell}$ that solve
\ali{
\nb_j R_{T,J}^{j\ell} &= (1 - \Pi_0) [u_{T,J}^\ell \psi_f(\la \Ga_I)] \notag \\
\nb_j R_{T,J\a \b}^{j\ell} &= (1 - \Pi_0) [ u_{T,J\a\b}^\ell \Om_f^{\a \b}(\la \Ga_I) ] \notag \\
\nb_j R_{H,J}^{j\ell} &= (1 - \Pi_0) [ u_{H,J}^\ell (\psi_f^2(\la \Ga_I) - 1) ] \notag \\
u_{H,J}^\ell &:= \nb_j[v_J^j v_J^\ell] \label{eq:uHjelldef}
}
and that obey the estimates 
\ali{
\sup_{0 \leq k \leq 3} \la^{-k} \co{\nb^k R_{T,J}^{j\ell} } &\lsm_D (\la^{-1} + B_\la^{-1} N^{-D/2}) \sup_{0 \leq |\va| \leq D + 5} \fr{\co{\nb_{\va} u_{T,J} }}{N^{|\va|/2} \Xi^{|\va|}} \label{eq:longRTJboundD} \\
\sup_{0 \leq k \leq 3} \la^{-k} \co{\nb^k R_{T,J\a\b}^{j\ell} } &\lsm_D (\la^{-1} + B_\la^{-1} N^{-D/2}) \sup_{0 \leq |\va| \leq D + 5} \fr{\co{\nb_{\va} u_{T,J\a\b} }}{N^{|\va|/2} \Xi^{|\va|}} \notag \\
\sup_{0 \leq k \leq 3} \la^{-k} \co{\nb^k R_{H,J}^{j\ell} } &\lsm_D (\la^{-1} + B_\la^{-1} N^{-D/2}) \sup_{0 \leq |\va| \leq D + 5} \fr{\co{\nb_{\va} u_{H,J} }}{N^{|\va|/2} \Xi^{|\va|}} \label{eq:longRTHboundD} \\
\suppt R_{T,J} \cup_{\a,\b} \suppt R_{T,J\a\b}  \cup \suppt R_{H,J} &\subseteq [t(I) - \th, t(I) + \th], \quad  J \in \{ I \} \times \F \notag
}
Our goal for these estimates is to gain a factor of $\la^{-1}$ in the bound for each stress term.  We choose $D \in \Z_+$ such that $N^{-D/2} \leq N^{-1} \Xi^{-1}$.  By the assumption $N \geq \Xi^\eta$ in Lemma~\ref{lem:convexInt}, it suffices to take $D > 2(1+\eta^{-1})$.  With this choice, the $B_\la^{-1} N^{-D/2}$ term above can be absorbed into the $\la^{-1}$ term.    

We now set $R_T^{j\ell} = \sum_J R_{T,J}^{j\ell} + \sum_{J,\a\b} R_{T,J\a\b}^{j\ell}$ and $R_H^{j\ell} = \sum_J R_{H,J}^{j\ell}$.  Then $R_T$ and $R_H$ solve \eqref{eq:transTerm2} and \eqref{eq:highFreqTerm2} respectively, since (using that $V^\ell$ in \eqref{eq:bigDivFormula} is the divergence of an antisymmetric tensor)
\ALI{
\nb_j R_T^{j\ell} &\stackrel{\eqref{eq:transTermtwoParts}}{=} (1 - \Pi_0) [ \pr_t V^\ell + v_\ep^j \nb_j V^\ell + V^j \nb_j v_\ep^\ell ] \\
&= (1-\Pi_0)[\pr_t V^\ell + \nb_j(v_\ep^j V^\ell + V^j v_\ep^\ell) ] ]  \\
&\stackrel{\eqref{eq:bigDivFormula}}{=} \pr_t V^\ell + \nb_j(v_\ep^j V^\ell + V^j v_\ep^\ell) = \eqref{eq:transTerm2} \\
\nb_j R_H^{j\ell} &\stackrel{\eqref{eq:highFreqIntReady}}{=} (1 - \Pi_0) \sum_{J \in \Z \times \F} \nb_j[ v_J^j v_J^\ell ( \psi_f^2(\la \Ga_I) - 1) ] \\
&= \sum_{J \in \Z \times \F} \nb_j[ v_J^j v_J^\ell ( \psi_f^2(\la \Ga_I) - 1) ] \stackrel{\eqref{eq:highFreqReduced}}{=} \eqref{eq:highFreqTerm2}
}
To prepare to bound $R_T$ and $R_H$, we first bound $u_{T,J}, u_{T,J\a\b}$ and $u_{H,J}$ using \eqref{eq:nbkAmpbd} -\eqref{eq:DdtbdDevJ} 
\ali{
\co{\nb_{\va} u_{T,J}} &\lsm_{|\va|} \co{\nb_{\va} \Ddt v_J^\ell} + \sum_{|\vcb| + |\vcc| = |\va|} \co{\nb_{\vcb} v_J^j} \co{\nb_{\vcc} \nb_j v_\ep^\ell } \notag \\
&\lsm_{|\va|} \plhxi^{5/2} N^{(|\va| - 1)_+/2} \Xi^{|\va| + 1} e_v^{1/2} e_R^{1/2} \notag \\
&+  \sum_{|\vcb| + |\vcc| = |\va|} [\plhxi^{1/2} N^{(|\vcb| - 1)_+/2} \Xi^{\vcb}] [ \plhxi^{1/2} N^{(|\vcc| - 1)_+/2} \Xi^{|\vcc|}] \notag \\
\co{\nb_{\va} u_{T,J\a\b}} &\lsm_{|\va|} \co{\nb_{\va} \Ddt \de v_{J,\a\b}^\ell} + \sum_{|\vcb| + |\vcc| = |\va|} \co{\nb_{\vcb} \de v_{J,\a\b}^j} \co{\nb_{\vcc} \nb_j v_\ep^\ell } \notag\\
&\lsm_{|\va|} \la^{-1} \plhxi^{5/2} N^{|\va|/2} \Xi^{|\va| + 2} e_v^{1/2} e_R^{1/2} \notag \\
&+  \la^{-2}\sum_{|\vcb| + |\vcc| = |\va|} [\plhxi^{1/2} N^{|\vcb|/2} \Xi^{|\vcb|+1}] [ \plhxi^{1/2} N^{|\vcc|/2} \Xi^{|\vcc|+1}] \notag \\
\Rightarrow \co{\nb_{\va} u_{T,J}} + \co{\nb_{\va} u_{T,J\a\b}} &\lsm_{|\va|} N^{|\va|/2} \Xi^{|\va|} [ \plhxi^{5/2} \Xi e_v^{1/2} e_R^{1/2}] \label{eq:conbvautjutjab} }
\ali{
\co{\nb_{\va} u_{H,J}} &\lsm \sum_{|\vcb| + |\vcc| = |\va| + 1} \co{\nb_{\vcb} v_J} \co{\nb_{\vcc} v_J} \notag \\
&\lsm_{|\va|} \lhxi \sum_{|\vcb| + |\vcc| = |\va| + 1} [N^{(|\vcb|-1)_+/2} \Xi^{|\vcb|} e_R^{1/2}] [N^{(|\vcc|-1)_+/2} \Xi^{|\vcc|} e_R^{1/2} ] \notag \\
&\lsm_{|\va|} \lhxi ~ N^{|\va|/2} \Xi^{|\va|+1} e_R \label{eq:nbvauHJbd}
}
From \eqref{eq:conbvautjutjab} and \eqref{eq:nbvauHJbd} we obtain
\ali{
\sup_{0 \leq |\va| \leq D + 5} \fr{\co{\nb_{\va} u_{T,J} }}{N^{|\va|/2} \Xi^{|\va|}} + \sup_{0 \leq |\va| \leq D + 5} \fr{\co{\nb_{\va} u_{T,J\a\b} }}{N^{|\va|/2} \Xi^{|\va|}} &\lsm_\eta \plhxi^{5/2} \Xi e_v^{1/2} e_R^{1/2} \label{eq:boundsToDutJ} \\
\sup_{0 \leq |\va| \leq D + 5} \fr{\co{\nb_{\va} u_{H,J} }}{N^{|\va|/2} \Xi^{|\va|}} &\lsm_\eta \lhxi \Xi e_R \label{eq:boundsToDuHJ}
}

Combining \eqref{eq:longRTJboundD}-\eqref{eq:longRTHboundD} with \eqref{eq:boundsToDutJ}-\eqref{eq:boundsToDuHJ} and using our choice of $D$ above, $\la = B_\la N \Xi$, and the universal bound on the number of $v_J$ and $\de v_{J,\a\b}$ that are nonzero at any given time, we obtain
\ali{
\sup_{0 \leq k \leq 3} \la^{-k} \co{\nb^k R_{T}^{j\ell} } &\lsm_\eta (B_\la N \Xi)^{-1} \plhxi^{5/2} \Xi e_v^{1/2} e_R^{1/2} \label{eq:almostChooseBLaT} \\
\sup_{0 \leq k \leq 3} \la^{-k} \co{\nb^k R_{T}^{j\ell} } &\lsm_\eta (B_\la N \Xi)^{-1} \lhxi \Xi e_R \label{eq:almostChooseBLaH}
}
We finally choose the parameter $B_\la$ sufficiently large so that \eqref{eq:needc0bdRTRH} holds (and such that $\la = B_\la N \Xi$ is an integer).  The other estimates in \eqref{eq:nbknewRTRH} now follow from \eqref{eq:almostChooseBLaT}-\eqref{eq:almostChooseBLaH} and $\la = B_\la N \Xi$.
\end{proof}

\subsection{Proof of Proposition~\ref{prop:nonstatPhase}}
In this Section we prove Proposition~\ref{prop:nonstatPhase} below, which was used to bound the error terms in Section~\ref{sec:divEqnTerms}.  The implicit constants in this section will depend on the $D$ introduced below.

\begin{prop}\label{prop:nonstatPhase}  For every integer $D \geq 1$ and for any smooth $\om : \T^3 \to \R$ with $\int_{\T^3} \om(X) dX = 0$ there exists $C = C_{D,\om}$ such that if $u^\ell$ is smooth and satisfies $\suppt u^\ell \subseteq [t(I) - \th, t(I) + \th]$ and 
\ali{
\sup_{0 \leq |\va| \leq D + 5} \fr{\co{\nb_{\va} u}}{N^{|\va|/2} \Xi^{|\va|}} &\leq H \label{eq:HdimensionnormDef}
}
then there exists a symmetric tensor field $Q^{j\ell}$ in the class $C_t C_x^3(\R \times \T^3)$ such that 
\ali{
\nb_j Q^{j\ell} &= (1 - \Pi_0)[u^\ell \om(\la \Ga_I) ] \notag \\
\sup_{0 \leq k \leq 3} \la^{-k} \co{\nb^k Q} &\leq C_{D,\om} ( \la^{-1} + B_\la^{-1} N^{-D/2} ) H \label{eq:boundOnQgood} \\
\suppt Q^{j\ell} &\subseteq [t(I) - \th, t(I)+\th] \label{eq:containingQGood}
}
Moreover, one can arrange that $Q$ depends bilinearly on $u$ and $\om$.
\end{prop}
To prove the Proposition, as in calculations of \cite{danSze}, we will expand $\om$ in a Fourier series to reduce to the case where $\om(X) = e^{i m \cdot X}$ and then sum over $m \neq 0$.  The case of $\om(X) = e^{i m \cdot X}$ will be handled as in \cite{isett} using a nonstationary phase argument with nonlinear phase functions.  We remark that the nonstationary phase technique in \cite{danSze} based on the earlier \cite{deLSzeCts} is different in that it proves estimates directly for $Q^{j\ell} = \RR^{j\ell}[u~ \om(\la\Ga_I)]$ that gain a factor of $\la^{-1+\ep}$ rather than $\la^{-1}$.  

To prepare for the proof, we start by stating estimates for the phase functions $m \cdot \Ga_I(t,x)$.
\begin{prop}  Let $m \in \Z^3 \setminus \{ 0 \}$ and $\xi_m(t,x) := m \cdot \Ga_I$.  Then for $t \in [t(I) - \th, t(I) + \th]$, we have
\ali{
\co{\nb_{\va} \nb \xi_m} &\lsm_{|\va|} |m| N^{(|\va| - 1)_+} \Xi^{|\va|} \label{eq:phaseGradBound} \\
\co{ ~|\nb \xi_m|^{-1} } &\lsm |m|^{-1}  \label{eq:phaseGradInvbd}
}
\end{prop}
\begin{proof}  If we view $\xi_m(t,x)$ as a map from $\R \times \R^3 \to \R^3$, then $\xi_m(t,x)$ is the unique solution to
\ALI{
(\pr_t + v_\ep^j \nb_j) \xi_m &= 0 \\
\xi_m(t(I),x) &= m \cdot x
}
The gradient of $\xi_m(t,x)$ then satisfies the equation 
\ali{
(\pr_t + v_\ep^j \nb_j) \nb_i \xi_m &= - \nb_i v_\ep^j \nb_j \xi_m \label{phaseTransport} \\
\nb_i \xi_m(t(I), x) &= m_i \notag
}
Then the vector field $\nb \xi_m$ obeys the same transport equations as the functions $\nb \xi_I$ used in \cite{isett} in the case of frequency-energy levels of order $L = 2$ .  From \cite[Proposition 17.4]{isett}, we obtain
\ali{
E_M[\xi_m](\Phi_s(t(I), x)) &\leq e^{C_M \Xi e_v^{1/2} |s|} E_M[\xi_m](t(I), x) \label{eq:expGrowthEMxim}\\
E_M[\xi_m](t,x) &:= \sum_{0 \leq |\va| \leq M } \Xi^{-2|\va|} N^{-(|\va| -1)_+} |\nb_{\va} \nb \xi_m|^2, \label{eq:EMdefxim}
}
where $\Phi_s$ is the coarse scale flow map defined in Section~\ref{sec:csfBTLmaps}.  For $|s| \leq \th$, we have $\Xi e_v^{1/2} |s| \lsm 1$ and the initial data satisfies $E_M[\xi_m](t(I),x) \lsm |m|$, from which the bound \eqref{eq:phaseGradBound} follows using \eqref{eq:expGrowthEMxim}-\eqref{eq:EMdefxim}.  

To obtain \eqref{eq:phaseGradInvbd}, let $p_s$ denote the point $p_s = \Phi_s(t(I),x)$.  Then for $|s| \leq \th$
\ALI{
\fr{d}{ds} |\nb \xi_m|^{-2}(\Phi_s(t(I),x)) &= - 2 |\nb \xi_m|^{-4}(p_s) \nb^i \xi_m(p_s) [\Ddt\nb_i \xi_m](p_s) \\
\left| \fr{d}{ds} |\nb \xi_m|^{-2}(\Phi_s(t(I),x)) \right| &\stackrel{\eqref{phaseTransport}}{\lsm} \co{\nb v_\ep } |\nb \xi_m|^{-2}(\Phi_s(t(I),x)) \\
\tx{Gronwall } \Rightarrow \quad |\nb \xi_m|^{-2}(\Phi_s(t(I),x)) &\leq e^{C_0 \co{\nb v_\ep}|s| } |\nb \xi_m|^{-2}(t(I), x)
}
From $\co{\nb v_\ep} |s| \, {\lsm} 1$ and $|\nb \xi_m|^{-2}(t(I), x) \, {=} \, |m|^{-2}$, we have $\co{\,|\nb \xi_m|^{-2} } \lsm |m|^{-2}$ implying \eqref{eq:phaseGradInvbd}.
\end{proof}

As a step towards proving Proposition~\ref{prop:nonstatPhase}, we now state
\begin{prop} \label{prop:linNonstPhaseProp} Under the assumptions of Proposition~\ref{prop:nonstatPhase}, if $m \in \Z^3 \setminus \{ 0 \}$,  
then there exists a (complex-valued) symmetric tensor field $Q^{j\ell} = Q_m^{j\ell} : \R \times \T^3 \to \SS \otimes \C$ such that 
\ali{
\nb_j Q_m^{j\ell} &= (1 - \Pi_0)[u^\ell e^{2\pi i m \cdot \la \Ga_I} ] \label{eq:divQmeqn} \\
\sup_{0 \leq |\va| \leq 3} (|m|\la)^{-{|\va|}} \co{\nb_{\va} Q_m} &\leq C_{D,\om} ( \la^{-1} + B_\la^{-1} N^{-D/2} ) H \label{eq:desiredQmbd} \\
\suppt Q_m^{j\ell} &\subseteq [t(I) - \th, t(I)+\th] \notag
}
Moreover, one can take $Q_m^{j\ell}$ to depend linearly on $u^\ell$. 
\end{prop}
\begin{proof}  Following \cite[Section 26]{isett}, we write our solution using a parametrix expansion of the form
\ali{
Q^{j\ell} &= Q_{(D)}^{j\ell} + \wtld{Q}_{(D)}^{j\ell}, \qquad Q_{(D)}^{j\ell} = (2 \pi \la)^{-1} \sum_{k=1}^D e^{2\pi i \la \xi_m} q_{(k)}^{j\ell} \label{eq:highorderExpand}
}
We explain first the case $D = 1$, where the method reduces to writing 
\ali{
Q^{j\ell} = (2 \pi \la)^{-1} e^{2 \pi i \la \xi_m} q_{(1)}^{j\ell} + \wtld{Q}_{(1)}^{j\ell} \label{eq:firstOrderExpand}
}
After we choose a smooth symmetric tensor $q_{(1)}^{j\ell}$ that solves $i \nb_j \xi_m q_{(1)}^{j\ell} = u^\ell$ pointwise, the first term in \eqref{eq:firstOrderExpand} will be a good approximate solution to $\nb_j Q^{j\ell} = e^{i \la \xi_m} u^\ell$.  The remainder term $\wtld{Q}^{j\ell}_{(1)}$ is then chosen to eliminate the error by solving $\nb_j \wtld{Q}^{j\ell}_{(1)} = - (2 \pi \la)^{-1} e^{2 \pi i \la \xi_m} \nb_j q_{(1)}^{j\ell} = e^{2 \pi i \la \xi_m} u^\ell - \nb_j Q_{(1)}^{j\ell}$.  This last equation can only be solved exactly when the original $e^{2 \pi i \la \xi_m} u^\ell$ has integral $0$; otherwise, one obtains a solution to \eqref{eq:divQmeqn} involving the projection $(1 - \Pi_0)$.

For $D \geq 1$, we repeat this process to determine a sequence of $q_{(k)}^{j\ell}$ and $u_{(k)}^\ell$ such that $u_{(0)}^\ell = u^\ell$ and 
\ali{
i \nb_j \xi_m q_{(k)}^{j\ell} &= u_{(k-1)}^\ell \label{eq:undetLin} \\ 
u_{(k)}^\ell &= - (2 \pi \la)^{-1} \nb_j q_{(k)}^{j\ell} \notag
}
on all of $\R \times \T^3$ for all $1 \leq k \leq D$.  To construct a good solution to the underdetermined equation \eqref{eq:undetLin}, we set $q_{(k)}^{j\ell} = \bar{q}^{j\ell}(\nb \xi_m)[u_{(k-1)}]$, where $\bar{q}^{j\ell} = \bar{q}^{j\ell}(p)[u]$ is a map with the following properties
\ali{
\bar{q}^{j\ell} \in C^\infty(\R^3 \setminus \{ 0 \} \times \R^3)  & \tx{ as a map taking values in } \SS \otimes \C \label{eq:cinfinityqbar} \\
 \bar{q}^{j\ell}(p)[u] \tx{ is linear in } &u \tx{ and homogeneous of degree -1 in } p \label{eq:homogpropqbar}\\
ip_j \bar{q}^{j\ell}(p)[u] &= u^\ell , \quad \tx{ for all } (p, u) \in \R^3 \setminus \{ 0 \} \times \R^3 \label{eq:solvelinqbar}
}
One can construct such a map $\bar{q}$ by decomposing $u^\ell = u_\perp^\ell + u_{\parallel}^\ell$, $u_\parallel^\ell = |p|^{-2}(u \cdot p) p^\ell$ and then setting $\bar{q}^{j\ell}(p)[u] = -i (q_\perp^{j\ell} + q_\parallel^{j\ell} )$, where $q_\parallel^{j\ell} = |p|^{-2}(u \cdot p) \de^{j\ell}$ and $q_\perp^{j\ell} = |p|^{-2} (p^j u_\perp^\ell + u_\perp^j p^\ell)$.  With this construction, $\bar{q}^{j\ell}$ is symmetric, one has $p_j q_\perp^{j\ell} = u_\perp^\ell$ and $p_j q_\parallel^{j\ell} = u_\parallel^{\ell}$, and properties \eqref{eq:cinfinityqbar}-\eqref{eq:solvelinqbar} all hold.

We now begin estimating the above parametrix.  By \eqref{eq:homogpropqbar}, the map $\bar{q}^{j\ell}(p)[u]$ can be written in the form $\bar{q}^{j\ell}(p)[u] = \bar{q}^{j\ell}_a(p) u^a$ where the $\bar{q}^{j\ell}_a(p)$ are homogeneous of degree $-1$ in $p$ and smooth away from $0$.  Then $q_{(k)}^{j\ell} = \bar{q}_a^{j\ell}(\nb \xi_m) u_{(k-1)}^a$.  We use the homogeneity of the derivatives of $\bar{q}_a^{j\ell}(\nb \xi_m)$ to write
\ALI{
\nb_{\va} [\bar{q}_a^{j\ell}(\nb \xi_m) ] &= \sum_{r=0}^{|\va|} \sum_{\va_i} \pr^{r} \bar{q}_a^{j\ell}(\nb \xi_m) \prod_{i=1}^r \nb_{\va_i} \nb \xi_m \\
&= \sum_{r=0}^{|\va|} \sum_{\va_i} |\nb \xi_m|^{-(1+r)} \pr^r \bar{q}_a^{j\ell}\left( \fr{\nb \xi_m}{|\nb \xi_m|} \right) \prod_{i=1}^r \nb_{\va_i} \nb \xi_m,
}
where the sum ranges over a family of multi-indices with $\sum_i |\va_i| = |\va|$.  Using \eqref{eq:phaseGradBound}-\eqref{eq:phaseGradInvbd} we have
\ali{
\co{\nb_{\va} [\bar{q}_a^{j\ell}(\nb \xi_m) ]} &\lsm_{|\va|} \sum_{r=0}^{|\va|} \sum_{\va_i} |m|^{-(1+r)} \sup_{|p| =1} |\pr^r \bar{q}_a^{j\ell}(p)| ~ \prod_{i=1}^r [N^{(|\va_i| - 1)_+/2} \Xi^{|\va_i|} |m| ] \notag \\
\co{\nb_{\va} [\bar{q}_a^{j\ell}(\nb \xi_m) ]}&\lsm_{|\va|} |m|^{-1} N^{(|\va| - 1)_+/2} \Xi^{|\va|} \lsm N^{|\va|/2} \Xi^{|\va|} \label{eq:bqnbximBound}
}
By induction, we now prove the following estimates for $q_{(k)}^{j\ell} = \bar{q}_a^{j\ell}(\nb \xi_m) u_{(k-1)}^a$ and $u_{(k)}^\ell = - \la^{-1} \nb_j q_{(k)}^{j\ell}$
\ali{
\co{\nb_{\va} q_{(k)}^{j\ell} } &\lsm N^{-(k-1)/2} N^{|\va|/2} \Xi^{|\va|} H, \quad \tx{ for all } 0 \leq |\va| \leq D - k + 4, \quad 1 \leq k \leq D \label{eq:qkparambdsiterate}\\
\co{\nb_{\va} u_{(k)}^\ell } &\lsm B_\la^{-1} N^{-k/2} N^{|\va|/2} \Xi^{|\va|} H \quad \tx{ for all } 0 \leq |\va| \leq D - k + 3, \quad 0 \leq k \leq D \label{eq:ampseqParambds} 
}
As a base case, note that \eqref{eq:ampseqParambds} holds for $k = 0$ (without the $B_\la^{-1}$ factor) because $\co{\nb_{\va} u_{(0)} } = \co{ \nb_{\va} u} \leq N^{|\va|/2} \Xi^{|\va|} H$ for all $0 \leq |\va| \leq D + 5$ by definition of $H$ in \eqref{eq:HdimensionnormDef}.  Now for $k \geq 1$, suppose \eqref{eq:ampseqParambds} holds for $k - 1$ (without the $B_\la^{-1}$ if $k - 1 = 0$).  Then if $0 \leq |\va| \leq D - k + 4 = D - (k-1) + 3$
\ALI{
\co{\nb_{\va} q_{(k)}^{j\ell} } &\lsm \sum_{|\vcb| + |\vcc| = |\va|} \co{ \nb_{\vcb} [ \bar{q}_a^{j\ell}(\nb \xi_m) ] } \co{\nb_{\vcc} u_{(k-1)}} \\
&\lsm \sum_{|\vcb| + |\vcc| = |\va|} [ N^{|\vcb|/2} \Xi^{|\vcb|} ] [N^{-(k-1)/2} N^{|\vcc|/2} \Xi^{|\vcc|} H ] \\
&\lsm N^{-(k-1)/2} N^{|\va|/2} \Xi^{|\va|} H
}
Then for $u_{(k)}^\ell = - (2\pi\la)^{-1} \nb_j q_{(k)}^{j\ell}$ and $0 \leq |\va| \leq D - k + 3$ we have
\ALI{
\co{\nb_{\va} u_{(k)} } &\lsm \la^{-1} \co{\nb_{\va} \nb_j q_{(k)}^{j\ell} } \\
&\lsm (B_\la N \Xi)^{-1} N^{-(k-1)/2} N^{(|\va| + 1)/2} \Xi^{(|\va| + 1)} H \\
&\lsm B_\la^{-1} N^{-k/2} N^{|\va|/2} \Xi^{|\va|} H
}
To estimate the parametrix in \eqref{eq:highorderExpand} we write, for $0 \leq |\va| \leq 3$, 
\ali{
\nb_{\va} Q_{(D)}^{j\ell} &= \sum_{k=1}^D \sum_{r = 0}^{|\va|} (2\pi\la)^{-1+r}\sum_{\vcb_i, \vcc} e^{2\pi i \la \xi_m} \nb_{\vcc} q_{(k)}^{j\ell} \prod_{i=1}^r [ \nb_{\vcb_i} \nb \xi_m ] \label{eq:bigFormula}
}
where the multi-indices in the summation satisfy $r + |\vcc| + \sum_i |\vcb_i| = |\va|$.  In the extreme case where all the derivatives hit $q_{(k)}$, note that even for $k = D$, \eqref{eq:qkparambdsiterate} provides bounds on at least $|\va| \leq 3$ derivatives.  The worst term is the case $r = |\va|$ where all the derivatives hit the phase function, each costing a factor of $\la |\nb \xi_m| \lsm \la |m|$ in the estimate.  The bound we attain (using $N^{-(k-2)/2} \leq 1$) is
\ali{
\co{\nb_{\va} Q_{(D)}^{j\ell}} &\lsm \sum_{k=1}^D \sum_{r=0}^{|\va|} \la^{-1+r} [N^{-(k-1)/2} N^{|\vcc|/2} \Xi^{|\vcc|} H] \prod_{i=1}^r [ |m| N^{|\vcb_i|/2} \Xi^{|\vcb_i|}  ] \notag \\
&\lsm  \la^{-1}\sum_{r=0}^{|\va|} \la^{r} N^{(|\va| - r)/2} \Xi^{|\va| - r } |m|^r H \notag \\
&\lsm  \la^{-1} \la^{|\va|} |m|^{|\va|} \sum_{r=0}^{|\va|} (B_\la N \Xi)^{r - |\va|} N^{(|\va| - r)/2} \Xi^{|\va| - r } H \notag \\
&\lsm \la^{-1} (|m| \la)^{|\va|} H \label{eq:finalParametrixBound}
}
The remainder term $\wtld{Q}_{(D)}^{j\ell}$ in \eqref{eq:highorderExpand} is defined to be
\ali{
\wtld{Q}_{(D)}^{j\ell} &= \RR^{j\ell}[ e^{2\pi i \la \xi_m} u_{(D)} ] \label{eq:qtldQDform}
}
Since $e^{2 \pi i \la \xi_m} u_{(D)}^\ell = e^{2 \pi i \la \xi_m} u^\ell - \nb_j Q_{(D)}^{j\ell}$ (which can be seen by induction on $D$) we see that $Q^{j\ell}$ in \eqref{eq:highorderExpand} solves the divergence equation \eqref{eq:divQmeqn} using \eqref{eq:invertModIntegral}.  To estimate $\wtld{Q}_{(D)}$, we use that $\RR^{j\ell}$ is a bounded operator on $C^0(\T^3)$.  This boundedness can be proven as in the bounds of Section~\ref{sec:rhoIests} by estimating
\ALI{
\| \RR^{j\ell}[U] \|_{C^0(\T^3)} &\leq \| \RR^{j\ell} \|~ \| U \|_{C^0(\T^3)} \\
\| \RR^{j\ell} \| \leq \sum_{q=0}^\infty \| \RR^{j\ell} P_q \|  &\lsm \sum_{q = 0}^\infty 2^{-q} \lsm 1 
}
Thus, as before in \eqref{eq:bigFormula}-\eqref{eq:finalParametrixBound}, if $0 \leq |\va| \leq 3$, using \eqref{eq:ampseqParambds} we have that
\ali{
\nb_{\va} \wtld{Q}_{(D)}^{j\ell} &= \RR^{j\ell}\left[\sum_{r = 0}^{|\va|} \la^{r}\sum_{\vcc, \vcb_i} e^{i \la \xi_m} \nb_{\vcc} u_{(D)} \prod_{i=1}^r [ \nb_{\vcb_i} \nb \xi_m ] \right] \notag \\
\co{\nb_{\va} \wtld{Q}_{(D)}^{j\ell} }&\lsm \sum_{r=0}^{|\va|} \sum_{\vcc, \vcb_i} \la^{r} [B_\la^{-1} N^{-D/2} N^{|\vcc|/2} \Xi^{|\vcc|} H] \prod_{i=1}^r [ |m| N^{|\vcb_i|/2} \Xi^{|\vcb_i|}  ] \notag \\
&\lsm B_\la^{-1} N^{-D/2} \la^{|\va|} |m|^{|\va|}  \sum_{r=0}^{|\va|} \sum_{\vcc, \vcb_i} (B_\la N \Xi)^{r - |\va|}  N^{(|\va|-r)/2} \Xi^{|\va|-r} H \notag \\
\co{\nb_{\va} \wtld{Q}_{(D)}^{j\ell} }&\lsm (|m| \la)^{|\va|} B_\la^{-1} N^{-D/2} H \label{eq:remainderParamBded}
}
Combining \eqref{eq:finalParametrixBound} and \eqref{eq:remainderParamBded} gives \eqref{eq:desiredQmbd}.  It is also clear from \eqref{eq:qtldQDform} and the construction of $Q_{(D)}$ that $\suppt Q_{(D)} \cup \suppt \wtld{Q}_{(D)} \subseteq \suppt u \subseteq [t(I) - \th, t(I) + \th]$.  This containment together with the previous discussion of \eqref{eq:divQmeqn} concludes the proof of Proposition~\ref{prop:linNonstPhaseProp}.
\end{proof}
We are now ready to prove Proposition~\ref{prop:nonstatPhase}.
\begin{proof}[Proof of Proposition~\ref{prop:nonstatPhase}]
Let $\om : \T^3 \to \R$ be smooth and have integral $0$ as in Proposition~\ref{prop:nonstatPhase} above.  Then the Fourier series
\ali{
\om(X) &= \sum_{m \neq 0} \hat{\om}(m) e^{2 \pi i m \cdot X}
}
converges absolutely in $C^0(\T^3)$, and, since $\om$ is real-valued and smooth, the coefficients obey
\ali{
\hat{\om}(-m) = \overline{\hat{\om}(m)} , \quad |\hat{\om}(m)| \lsm_\om |m|^{-40} \label{eq:fCoeffProps}
}
For each $m \in \Z^3$, choose a solution $Q_m^{j\ell}$ to $\nb_j Q_m^{j\ell} = (1 - \Pi_0) [ e^{2 \pi i \la m \cdot \Ga_I } u^\ell] $ that obeys the conclusions of Proposition~\ref{prop:linNonstPhaseProp} and set
\ali{
Q^{j\ell} &:= \fr{1}{2}\sum_{m \in \Z^3 \setminus \{ 0 \}} \left(\hat{\om}(m)Q_m^{j\ell} + \hat{\om}(-m) \, \overline{Q}_m^{j\ell} \right) \label{eq:Qjlsumdef}
}
Then $Q^{j\ell}$ is real-valued by \eqref{eq:fCoeffProps} and belongs to $C_t C_x^3(\R \times \T^3)$ by the following estimate 
\ali{
\sup_{0 \leq |\va| \leq 3} \la^{-|\va|} \co{ \nb_{\va} Q } &\lsm \sum_{m \in \Z^3\setminus \{ 0 \}} |\hat{\om}(m)| |m|^{3} \sup_{0 \leq |\va| \leq 3} (|m|\la)^{-|\va|} \co{ \nb_{\va} Q_m } \\
&\lsm \sum_{m \in \Z} |\hat{\om}(m)| |m|^3 (\la^{-1} + B_\la^{-1} N^{-D/2}) H \\
&\stackrel{\eqref{eq:fCoeffProps}}{\lsm} (\la^{-1} + B_\la^{-1} N^{-D/2}) H
}
Thus $Q^{j\ell}$ satisfies \eqref{eq:boundOnQgood}-\eqref{eq:containingQGood}.  Taking the divergence of $Q^{j\ell}$ in \eqref{eq:Qjlsumdef} and using that $u^\ell$ is real-valued, 
\ali{
\nb_j Q^{j\ell} &= \fr{1}{2} \sum_{m \in \Z^3} \Big(\hat{\om}(m)(1 - \Pi_0) [ e^{2 \pi i \la m \cdot \Ga_I } u^\ell ] + \hat{\om}(-m)(1 - \Pi_0) [ e^{-2 \pi i \la m \cdot \Ga_I } u^\ell ] \Big) \\
\nb_j Q^{j\ell} &= (1 - \Pi_0) \Big( \sum_{m \in \Z^3} \hat{\om}(m) e^{2 \pi i m \cdot (\la \Ga_I) } u^\ell \Big) = (1 - \Pi_0) [  u^\ell \om(\la \Ga_I)] 
}
This calculation concludes the proof of Proposition~\ref{prop:nonstatPhase}.  \end{proof}

\subsection{Concluding the Proof of the Convex Integration Lemma} \label{sec:concludingProof}
In this Section we conclude the proof of Lemma~\ref{lem:convexInt} by indicating where in the course of the proof the various conclusions of the Lemma have been shown.  The constant $b_0$ whose existence is asserted by the Lemma was chosen following Lines~\eqref{eq:aIfbdis}-\eqref{eq:bIftermchooseK}.  The choices of $b_0$ and $K$ there assure that the square root used to define the coefficients $\ga_J$ is well-defined and bounded from below.  The bounds implied by \eqref{eq:newFrEnLvls} and Definition~\ref{defn:frenlvls} for the new velocity $v_1 = v + V$ were proven in Proposition~\ref{correctEsts}.  Inequality \eqref{eq:cobdcorrect} for the velocity correction was also proven in Proposition~\ref{correctEsts}.  The bounds implied by \eqref{eq:newFrEnLvls} and Definition~\ref{defn:frenlvls} for the new stress $R_1 = R_M + R_S + R_T + R_H$ follow from the bounds in Propositions~\ref{prop:noInvDivbds} and \ref{prop:divStressBds}.

To check the statement~\eqref{ct:suppCvxInt} regarding the growth of support, observe that \eqref{ct:growth} and \eqref{eq:supptRI} imply 
\ali{
\suppt R &\subseteq N(J;3^{-1} \Xi^{-1} e_v^{-1/2}) \cap \bigcup_I \left[ t(I) - 2^{-1} \th, t(I) + 2^{-1} \th \right] \notag
}
Technically, statement \eqref{ct:suppCvxInt} may not hold if we define $e_I^{1/2}(t)$ by formula \eqref{eq:eIhalf} for all $I \in \Z$.   However, we can replace $e_I^{1/2}(t)$ by $0$ for $I$ such that $R_I$ is equal to $0$ without affecting the proof.  Modifying the construction in this way, recalling from Lemma~\ref{lem:glueLem} and \eqref{ineq:thBound} that $\th \leq \de_0 \Xi^{-1} e_v^{-1/2} \leq 25^{-1} \Xi^{-1} e_v^{-1/2}$, and letting $\II$ be the subset of $\Z$ such that $R_I \neq 0$, we have
\ALI{
\suppt V \cup \suppt R_1 &\subseteq \bigcup_{I \in \II} \left[t(I) - \th, t(I) + \th \right] \subseteq N(\suppt R; 2^{-1} \th) \\
 \subseteq N(\suppt R; 50^{-1} \Xi^{-1} e_v^{-1/2}) &\subseteq N\Big(N(J; 3^{-1} \Xi^{-1} e_v^{-1/2}); 50^{-1} \Xi^{-1} e_v^{-1/2}\Big) \\
\Rightarrow \quad \suppt V \cup \suppt R_1 &\subseteq N(J; \Xi^{-1} e_v^{-1/2})
}
Since $\suppt v \subseteq N(J; 3^{-1} \Xi^{-1} e_v^{-1/2})$, we also have $\suppt v_1 = \suppt (v + V) \subseteq N(J; \Xi^{-1} e_v^{-1/2})$, which confirms the containment~\eqref{ct:suppCvxInt} and hence concludes the proof of Lemma~\ref{lem:convexInt}.

\section{Proof of the Main Theorem} \label{proofMainThm}
In this Section we give a proof of Theorem~\ref{thm:main} based on Lemma~\ref{lem:mainLem}.  We follow the algorithm for computing regularity in the presence of double exponential frequency growth developed in \cite{isett}. 

For the base case of the iteration, we will use the previous convex integration result of \cite{isett}, since Lemma~\ref{lem:mainLem} of the present paper does not include any inputs that would guarantee the nontriviality of the solution.  The final solution is then constructed by iteratively applying Lemma~\ref{lem:mainLem}.


Let $\a^* < 1/3$ be given.  We introduce a parameter $\de$, $0 < \de < 1/4$, that will be chosen close to zero depending on $\a^*$.  Our proof will lead to the following result, which immediately implies Theorem~\ref{thm:main}:
\begin{thm} \label{thm:messyReg}  For any $0 < \de < 1/4$ there exists a weak solution $(v, p)$ to the incompressible Euler equations with nonempty, compact support in time in $\R \times \T^3$ such that $v \in C_{t,x}^\a, p \in C_{t,x}^{2\a}$ whenever
\ali{
- \left( \fr{1}{2} - \de \a \right)\left(1 + \fr{\de}{2} \right) + \fr{\de}{2} + \a\left(\fr{3}{2} + \de \right) < 0 \label{ineq:messyReg}
}
\end{thm}
Note that the left hand side of \eqref{ineq:messyReg} is bounded by $-\fr{1}{2} + \fr{3\a}{2} + O(\de)$.  Thus, given $\a^* < 1/3$, we can always choose $\de > 0$ so that \eqref{ineq:messyReg} is satisfied for $\a = \a^*$, and Theorem~\ref{thm:main} now follows.

\subsection{Regularity parameters} \label{sec:regParams}
We start by introducing a few parameters.  The parameter $\de \in (0, 1/4)$ is fixed.  We introduce a parameter $\ep$, which we set equal to $\ep := \fr{\de}{2}$.  We introduce a third parameter $\eta$, which we set equal to $\eta := \fr{\ep}{16} = \fr{\de}{32}$.  We also define a parameter $r := \fr{20}{\de^2}$.  
The parameters have already been chosen in such a way that $\eta \ll \ep \ll \de$, 
and $r$ is large enough depending on $\de$ and $\ep$.  
We define the constant $C_\eta$ to be the constant from the conclusion of Lemma~\ref{lem:mainLem} using the above choice of $\eta > 0$.  

There will be two large constant parameters.  One parameter is called $N_{(-1)}$.  The largest parameter will be called $Z$; it depends on the $\de, \ep, \eta$ above and will be large compared to $N_{(-1)}$.

There will also be a sequence of parameters $(\Xi_{(k)}, e_{v,(k)}, e_{R,(k)})$ that will represent the frequency-energy level bounds on our approximate solutions.  In terms of these, we define $\hxi_{(k)} := \Xi_{(k)} \left( \fr{e_v}{e_R} \right)^{1/2}_{(k)}$.

\subsection{The base case: k = -1} \label{sec:baseCase}

The base case will rely on the Main Lemma in \cite{isett}, since this Lemma gives information that will be crucial for proving nontriviality of the solution, and also controls higher derivatives of the Euler-Reynolds flow to make it compatible with the scheme of the current paper.

Consider first the zero solution to the Euler-Reynolds system $(v, p, R)_{-1} = (0, 0, 0)$.  It has frequency energy levels (in the sense of \cite[Definition 10.1]{isett}) to order $3$ in $C^0$ below $(\Xi_{(-1)}, e_{v,(-1)}, e_{R,(-1)}) = (3, 1, 1)$.  Let $e^{1/2}(t) \geq 0$ be a smooth function with compact support in $\R$ such that $e^{1/2}(0) = 1$ and
\ALI{
\sup_t \left| \fr{d^a}{dt^a} e^{1/2}(t) \right| &\leq 10 (\Xi_{(-1)} e_{v,(-1)}^{1/2})^{r} e_{R,(-1)}^{1/2}, \quad a = 0, 1, 2.
}
Let $\check{C}$ be the constant in the Main Lemma (Lemma 10.1) of \cite{isett}, where we take $L = 3$ and $N_{(-1)} \geq \Xi_{(-1)}^{1/2}$.  Applying that lemma with the function $e^{1/2}(t)$ above and $N_{(-1)}$ to be chosen, we obtain an Euler-Reynolds flow $(v, p, R)_{(0)}$, also with compact support in time, with Frequency-Energy levels to order $3$ in $C^0$ (in the sense of Definition~\ref{defn:frenlvls} above) bounded by
\ali{
(\Xi_{(0)}, e_{v,(0)}, e_{R,(0)}) &= \left( 3\check{C} N_{(-1)} , 1, \fr{1}{N_{(-1)}^{1/2}}\right)  \notag
}
For $N_{(-1)}$ sufficiently large, the following inequalities are satisfied in the stage $k = 0$:
\ali{
\log \hxi_{(k)} &\leq \left( \fr{e_v}{e_R} \right)_{(k)}^\ep \label{ineq:logReq}\\
\hxi_{(k)} &\leq \fr{1}{r^r} \ever^{r \ep}_{(k)} \label{ineq:secondReq} \\
e_{R,(k)}^{\fr{\de r \ep}{4}} \Xi_{(k)} &\leq 1 \label{ineq:thirdReq}\\
\Xi_{(k)}^\eta \ever_{(k)}^{-1/2} e_{R,(k)}^\de &\leq 1 \label{ineq:etaReq}
}
This choice is possible thanks to our choice of $r$ being large relative to $\de$ and $\ep$; for example, when $k = 0$ the left hand side of \eqref{ineq:thirdReq} is bounded above by 
\ALI{
e_{R,(0)}^{\fr{\de r \ep}{4}} \Xi_{(0)} &\leq \check{C} N_{(-1)}^{1 - \fr{\de r \ep}{8}} 3.
}
One obtains \eqref{ineq:secondReq} and \eqref{ineq:etaReq} similarly for sufficiently large $N_{(-1)}$.  

Note that inequality \eqref{ineq:logReq} follows immediately from \eqref{ineq:secondReq} by taking $u$ to be the right hand side of \eqref{ineq:logReq} in the following elementary inequality
\ALI{ \left(1 + \fr{u}{r}\right)^r \leq e^u \quad \tx{ for all } u \geq 0, r > 0  }

The last conditions we require on $N_{(-1)}$ are the ones that will ultimately guarantee that we construct a nontrivial solution.  Observe that Lemma 10.1 of~\cite{isett} guarantees a bound of
\ALI{
\sup_t \left| \int_{\T^3} |v_{(0)}(t,x)|^2 dx - e(t) \right| &\leq \check{C} \fr{e_{R,(-1)}}{N_{(-1)}}
}
(Note that, since we start with $v_{(-1)} = 0$, the correction $V = V_{(-1)}$ in this case is equal to our solution $v_{(0)}$.)  For $N_{(-1)}$ sufficiently large, we can guarantee that 
\ali{
.81 = e(0) - .19 &\leq \int_{\T^3} |v_{(0)}(0,x)|^2 dx \leq \co{ v_{(0)} }^2, \qquad \notag \\
\Rightarrow \quad .9 &\leq \co{ v_{(0)} } \label{eq:nontrivC0}
}
Finally, let $\check{C}_0$ be the constant that in the statement of inequality \eqref{ineq:coBdV} in Lemma~\ref{lem:mainLem} (which is the upper bound on the $C^0$ norm of the correction).  A sufficiently large choice of $N_{(-1)}$ guarantees that
\ali{
5000 \check{C}_0 ( \log \hxi_{(0)} )^{1/2} e_{R,(0)}^{1/2} \leq \fr{1}{400} \label{eq:0thCorrectSmall}
}
We now fix $N_{(-1)}$ to satisfy the above conditions together with \eqref{ineq:logReq}-\eqref{ineq:etaReq}.

\subsection{The Sequence of Parameters}

The goal of the present section is establish the main properties of the parameters $(\Xi, e_v, e_R)_{(k)}$ that will ultimately be the frequency energy levels of our sequence of Euler-Reynolds flows.  The value of $(\Xi_{(0)}, e_{v,(0)}, e_{R,(0)})$ is already determined.  The remaining values of the sequence are governed by the parameters $\de, \ep, \eta$ and a parameter $Z$ according to the following rules:
\ali{
\Xi_{(k+1)} &= C_\eta Z \ever_{(k)}^{\fr{1}{2} + \fr{5 \ep}{2}} e_{R,(k)}^{-\de} \Xi_{(k)} \label{eq:Xievol} \\
e_{v,(k+1)} &= \ever_{(k)}^\ep e_{R,(k)} \label{eq:evEvol} \\
e_{R,(k+1)} &= \fr{e_{R,(k)}^{1+\de}}{Z} \label{eq:eREvol}
}
The constant $C_\eta$ above is the constant associated to the parameter $\eta$ by Lemma~\ref{lem:mainLem}, with $\eta = \fr{\de}{32}$ specified above.  We also define the sequence
\ali{
N_{(k)} &:= Z (\log \hxi_{(k)})^{5/2} \ever_{(k)}^{1/2} e_{R,(k)}^{-\de} \label{eq:Nkdef}
}
Our choice of $Z$ will be specified later in line \eqref{ineq:smallCorrections}.

The following Proposition will ensure that the iteration proceeds in a well defined way for sufficiently large choices of $Z$.
\begin{prop} \label{prop:paramProp} Let $(\Xi_{(0)}, e_{v,(0)}, e_{R,(0)})$ be parameters that satisfy the conditions \eqref{ineq:logReq}-\eqref{ineq:etaReq} of Section~\ref{sec:baseCase}, with $e_{v,(0)} \geq e_{R,(0)}$ and $e_{R,(0)} < 1 < \Xi_{(0)}$.  There exists $Z_0$ such that for all $Z \geq Z_0$, the sequence determined by \eqref{eq:Xievol}-\eqref{eq:eREvol} and \eqref{eq:Nkdef} satisfies conditions \eqref{ineq:logReq}-\eqref{ineq:etaReq} for all $k \geq 0$ and also
\ali{
\ever_{(k)}^{\fr{1}{2}} &\leq Z^{1/2} e_{R,(k)}^{-\fr{\de}{2}} \label{eq:ratioGrowth} \\
N_{(k)} &\geq \Xi_{(k)}^\eta \label{eq:Nkbound}
}
\end{prop}
\begin{proof}[The case $k = 0$]
For $Z$ sufficiently large depending on $e_{v,(0)}, e_{R,(0)}$, we can ensure that \eqref{eq:ratioGrowth} holds for $k = 0$.  Also, \eqref{eq:Nkbound} is an immediate consequence of \eqref{ineq:etaReq} and the definition \eqref{eq:Nkdef} of $N_{(k)}$.  Thus we may assume the Proposition holds for $k = 0$.  We now prove the Proposition for stage $k+1$.
\end{proof}
\begin{proof}[Proof of \eqref{ineq:secondReq} and \eqref{ineq:logReq} for $k +1$]
Observe that
\ALI{
\hxi_{(k+1)} \ever_{(k+1)}^{-r\ep} &= \Xi_{(k+1)} \ever_{(k+1)}^{\fr{1}{2} - r \ep} \leq \Xi_{(k+1)} \ever_{(k+1)}^{-\fr{r\ep}{2}} \notag \\
\Xi_{(k+1)} \ever_{(k+1)}^{-\fr{r\ep}{2}} &= \left[ C_\eta Z e_{R,(k)}^{-\de} \ever_{(k)}^{\fr{1 + 5 \ep}{2}} \Xi_{(k)} \right] \left[\ever_{(k)}^\ep Z e_{R,(k)}^{-\de} \right]^{-\fr{r \ep}{2}} \notag \\
&\leq C_\eta Z^{1 - \fr{r \ep}{2}} \ever_{(k)}^{\fr{1 + 5 \ep}{2} - \fr{r \ep^2}{2}} [ e_{R,(k)}^{\fr{\de r \ep}{4}} \Xi_{(k)} ]
}
By our choice of $r$, the power to which $Z$ is raised in the last line is negative and likewise the energy ratio $(e_v/e_R)_{(k)}$ is raised to a positive power.  The energy ratio term is therefore at most $1$ and the constant term is bounded by $r^{-r}$ for sufficiently large $Z$.  The last term in brackets is at most $1$ by the induction hypothesis on \eqref{ineq:thirdReq}.  Finally, as noted previously in Section~\ref{sec:baseCase}, inequality \eqref{ineq:logReq} follows from \eqref{ineq:secondReq}. 
\end{proof}
\begin{proof}[Proof of \eqref{ineq:thirdReq} for $k + 1$]
Observe that
\ALI{
e_{R,(k+1)}^{\fr{\de r \ep}{4}} \Xi_{(k+1)} &= C_\eta Z^{1 - \fr{\de r \ep}{4}} \left[ \ever_{(k)}^{\fr{1 + 5 \ep}{2}} e_{R,(k)}^{-\de + \fr{\de^2 r \ep}{4}} \right] e_{R,(k)}^{\fr{\de r \ep}{4}} \Xi_{(k)} 
}
The last product involving $e_{R,(k)}$ and $\Xi_{(k)}$ is at most $1$ by the induction hypothesis on \eqref{ineq:thirdReq}.  Applying the induction hypothesis of inequality \eqref{eq:ratioGrowth} and $5 \ep < 1$ gives
\ALI{
e_{R,(k+1)}^{\fr{\de r \ep}{4}} \Xi_{(k+1)} &\leq C_\eta Z^{2 - \fr{\de r \ep}{4}} e_{R,(k)}^{-2 \de + \fr{\de^2 r \ep}{4}}
}
For $Z$ large enough, the constant term is bounded by $1$, and the power to which $e_{R,(k)}$ is raised above is positive by the choice of $r$, which gives the inequality.  
\end{proof}
\begin{proof}[Proof of \eqref{ineq:etaReq} and \eqref{eq:Nkbound} for $k + 1$]
As before, \eqref{eq:Nkbound} follows from \eqref{ineq:etaReq}.  To prove \eqref{ineq:etaReq}, observe that
\ALI{
\Xi_{(k+1)}^\eta &\ever_{(k+1)}^{-1/2} e_{R,(k)}^\de = \left[ C_\eta Z \ever_{(k)}^{\fr{1 + 5 \ep}{2}} e_{R,(k)}^{-\de} \Xi_{(k)} \right]^\eta \left[ Z^{-1/2} \ever_{(k)}^{-\ep/2} e_{R,(k)}^{\fr{\de}{2}} \right] \fr{e_{R,(k)}^{\de(1+\de)}}{Z^\de} \\
&\leq (C_\eta)^\eta Z^{\eta - \fr{1}{2} - \de} \ever_{(k)}^{\left(\fr{1 + 5}{2}\right) \eta - \fr{\ep}{2}} e_{R,(k)}^{\de(- \eta + \fr{1}{2} + \de)} [ e_{R,(k)}^\de \Xi_{(k)}^\eta]
}
By our induction hypotheses on \eqref{ineq:etaReq} and \eqref{eq:ratioGrowth}, the last term in brackets is at most
\ALI{
 e_{R,(k)}^\de \Xi_{(k)}^\eta \leq \ever_{(k)}^{1/2} \leq Z^{1/2} e_{R,(k)}^{-\fr{\de}{2}}
}
Using this bound and noting that $(e_v/e_R)_{(k)}$ is raised to a negative power by the choice of $\eta$ gives
\ALI{
\Xi_{(k+1)}^\eta \ever_{(k+1)}^{-1/2} e_{R,(k)}^\de &\leq (C_\eta)^\eta Z^{\eta - \de}  e_{R,(k)}^{\de(-\eta + \de)} 
}
Since we chose $\eta < \de$, we see that $Z$ is raised to a negative power and $e_{R,(k)}$ is raised to a positive power in the above estimate.  For large $Z$, the constant term is at most $1$ and the estimate follows.
\end{proof}
\begin{proof}[Proof of \eqref{eq:ratioGrowth} for $k + 1$]  By the induction hypothesis on \eqref{eq:ratioGrowth} and using \eqref{eq:eREvol} and $\ep = \fr{\de}{2}$, we obtain
\ALI{
\ever_{(k+1)}^{1/2} &=Z^{1/2} \ever_{(k)}^{\fr{\ep}{2}} e_{R,(k)}^{-\fr{\de}{2}} \leq Z^{\fr{1}{2} + \fr{\ep}{2}} e_{R,(k)}^{- \fr{\de \ep}{2} - \fr{\de}{2}} \\
&= Z^{\fr{1}{2}} Z^{\fr{\ep}{2} -\fr{\de}{2} \left(\fr{1 + \ep}{1 + \de} \right) } e_{R,(k+1)}^{-\fr{\de}{2} \left( \fr{1 + \ep}{1 + \de}  \right)} \leq Z^{1/2} e_{R,(k+1)}^{-\fr{\de}{2}}
}
\end{proof}

\subsection{Iteration of the Main Lemma} \label{sec:iterateMainLem}

We now prove Theorem~\ref{thm:main} by repeated iteration of Lemma~\ref{lem:mainLem}.

Let $(v_{(0)}, p_{(0)}, R_{(0)})$ be the Euler-Reynolds flow constructed in the base case of Section~\ref{sec:baseCase}, and $(\Xi_{(0)}, e_{v,(0)}, e_{R,(0)})$ be its associated frequency-energy levels, which satisfy the assumptions of Proposition~\ref{prop:paramProp}.  Let $I_{(0)}$ be a bounded, closed interval containing $\suppt v_{(0)} \cup \suppt R_{(0)}$.  

We construct a sequence of Euler-Reynolds flows with supporting time intervals $J_{(k)}$ and frequency-energy levels bounded by $(\Xi_{(k)}, e_{v,(k)}, e_{R,(k)})$ as follows.  For $k \geq 0$, apply Lemma~\ref{lem:mainLem} with the parameter $\eta$ chosen in Section~\ref{sec:regParams} and taking $N$ to be the $N_{(k)}$ defined in \eqref{eq:Nkdef}.  Note that the parameter $N_{(k)}$ satisfies the admissibility conditions $N_{(k)} \geq (e_v/e_{R,(k)})^{1/2}$ and $N_{(k)} \geq \Xi_{(k)}^\eta$ by Proposition~\ref{prop:paramProp}.  Lemma~\ref{lem:mainLem} then yields  an Euler-Reynolds flow $(v_{(k+1)}, p_{(k+1)}, R_{(k+1)})$ such that
\ali{
\suppt v_{(k+1)} \cup \suppt R_{(k+1)} &\subseteq J_{(k+1)} := N(J_{(k)}; \Xi_{(k)}^{-1} e_{v,(k)}^{-1/2}) \label{eq:supportGrowth}
}
and such that the frequency energy levels of $(v_{(k+1)}, p_{(k+1)}, R_{(k+1)})$ are bounded by
\ALI{
(\Xi_{(k)}', e_{v,(k)}', e_{R,(k)}') = \left( C_\eta Z (\log \hxi_{(k)})^{5/2} e_{R,(k)}^{-\de} \Xi_{(k)}, (\log \hxi_{(k)}) e_{R,(k)}, \fr{e_{R,(k)}^{1+\de}}{Z} \right)
}
By inequality \eqref{ineq:logReq} of Proposition~\ref{prop:paramProp}, we have $\Xi_{(k)}' \leq \Xi_{(k+1)}$ and $e_{v,(k)}' \leq e_{v,(k+1)}$, and we also have $e_{R,(k)}' = e_{R,(k+1)}$.  We may therefore regard $(v_{(k+1)}, p_{(k+1)}, R_{(k+1)})$ as an Euler-Reynolds flow with frequency energy levels below $(\Xi_{(k+1)}, e_{v,(k+1)}, e_{R,(k+1)})$, allowing the induction to continue in a well-defined way.  Lemma~\ref{lem:mainLem} also gives the following estimates for $V_{(k)} := v_{(k+1)} - v_{(k)}$
\ali{
\co{ V_{(k)} } &\leq \check{C}_0 (\log \hxi_{(k)})^{1/2} e_{R,(k)}^{1/2} \label{eq:c0bound} \\
\co{ \nab V_{(k)} } &\leq C_\eta N_{(k)} \Xi_{(k)} (\log \hxi_{(k)})^{1/2} e_{R,(k)}^{1/2} \label{eq:c1boundCorrect}
}

\subsection{Continuity and Nontriviality of the Solution}
We claim that the sequence of velocity fields $v_{(k)}$ converges uniformly to a limit that is a nontrivial, continuous weak solution to the Euler equations.  Indeed, from \eqref{eq:c0bound} we have that for all $k \geq 0$,
\ali{
\co{ V_{(k+1)} } \stackrel{\eqref{ineq:logReq}}{\leq} \check{C}_0 \ever_{(k+1)}^{\fr{\ep}{2}} & e_{R,(k+1)}^{1/2} \stackrel{\eqref{eq:ratioGrowth}}{\leq} \check{C}_0 Z^{\fr{\ep}{2}} e_{R,(k+1)}^{\fr{1}{2} - \fr{\de}{2}} \leq \check{C}_0 Z^{\fr{\ep + \de - 1}{2}} e_{R,(k)}^{(1 + \de)\left(\fr{1 - \de}{2}\right)} \notag \\
\co{ V_{(k+1)} } &\leq \check{C}_0 Z^{-\fr{1}{3}} e_{R,(k)}^{\fr{1}{4}} \label{ineq:c0bdkplus1}
}
Using $e_{R,(k+1)} \leq \fr{1}{2} e_{R,(k)}$, we can at this point choose $Z$ large enough depending on $\check{C}_0$ (= the constant in inequality \eqref{ineq:coBdV}) and $e_{R,(0)}$ such that
\ali{
\sum_{k=0}^\infty  \check{C}_0 Z^{-\fr{1}{3}} e_{R,(k)}^{\fr{1}{4}} &\leq \fr{3}{400}. \label{ineq:smallCorrections}
}
It follows that $v_{(k)}$ converges uniformly to a continuous velocity field $v$.  

If we choose the integral $0$ normalization for $p_{(k)} = \De^{-1} \nb_j \nb_\ell R^{j\ell}_{(k)} - \De^{-1} \nb_j \nb_\ell(v_{(k)}^j v_{(k)}^\ell) $, then since $R_{(k)} \to 0$ uniformly, we have that  $p_{(k)}$ converges weakly in $\DD'$ to the pressure $p = -\De^{-1} \nab_j \nab_\ell (v^j v^\ell)$.  One sees by testing the Euler-Reynolds equations for $(v_{(k)}, p_{(k)}, R_{(k)})$ against smooth test functions that the pair $(v,p)$ form a weak solution to the incompressible Euler equations.


To see that $v$ is not the $0$ solution, compare the lower bound \eqref{eq:nontrivC0} on $\co{v_{(0)}}$ to the upper bound
\ALI{
\co{ v - v_{(0)} } &\leq \co{ V_{(0)} } + \sum_{k = 0}^\infty \co{V_{(k+1)}} \leq \fr{1}{400} + \fr{3}{400} = .01 ,
}
which follows from \eqref{eq:0thCorrectSmall} and \eqref{ineq:c0bdkplus1}-\eqref{ineq:smallCorrections}.  
It now remains to show that $v$ has compact support in time and satisfies the regularity stated in Theorem~\ref{thm:messyReg}.

\subsection{Regularity and Compact Support in Time of the Solution}

We now show that the incompressible Euler flow $v$ defined in Section~\ref{sec:iterateMainLem} above belongs to the class $v \in C_{t,x}^\a, p \in C_{t,x}^{2 \a}$ for all $\a$ that satisfy inequality \eqref{ineq:messyReg}.  The time regularity theory of \cite{isett2} shows that if $0 < \a < 1$ and $(v,p)$ is an incompressible Euler flow in the class $v \in C_t C_x^\a$, then $v \in C_{t,x}^\a$ and $p \in C_{t,x}^{2\b}$ for all $0 \leq \b < \a$.  It therefore suffices to show that $v \in C_t C_x^\a$ for the stated range of $\a$.

Interpolating \eqref{ineq:coBdV} and \eqref{ineq:coNbvCorrect} gives the following bound on the $C_t C_x^\a$ norm of each correction
\ali{
\| V_{(k)} \|_{C_t C_x^\a} &\leq C (N_{(k)} \Xi_{(k)})^\a (\log \hxi_{(k)})^{1/2} e_{R,(k)}^{1/2} \notag \\
&\leq C Z^\a [ (\log \hxi_{(k)})^{\fr{5\a}{2}} \ever_{(k)}^{\fr{\a}{2}} e_{R,(k)}^{-\de \a} \Xi_{(k)}^\a ] (\log \hxi_{(k)})^{1/2} e_{R,(k)}^{1/2} \notag \\
&= C Z^\a (\log \hxi_{(k)})^{\fr{5\a + 1}{2}} \ever_{(k)}^{\fr{\a}{2}} e_{R,(k)}^{\fr{1}{2}-\de \a} \Xi_{(k)}^\a  \notag \\ 
\fr{\| V_{(k)} \|_{C_t C_x^\a}}{C Z^\a} &\stackrel{\eqref{ineq:logReq}}{\leq}  e_{R,(k)}^{\fr{1}{2} - \de \a} \ever_{(k)}^{1/2} \Xi_{(k)}^\a  \label{eq:ctxabdCorrect}
}
Next define the following sequence of parameter vectors
\ALI{
\psi_{(k)} &:= \vect{ \log e_{R} \\ \log (e_v/e_R) \\ \log \Xi }_{(k)}
}
The evolution rules \eqref{eq:Xievol}-\eqref{eq:eREvol} can be rephrased as
\ali{
\psi_{(k+1)} &= \vect{ - \log Z \\ \log Z \\ \log(C_\eta Z) } + \mat{ccc}{1 + \de & 0 & 0 \\ -\de & \ep & 0 \\ -\de & \fr{1}{2} + \ep & 1 } \psi_{(k)} \label{eq:logEvolutionRule}
}
We call the $3 \times 3$ matrix appearing on the right hand side of \eqref{eq:logEvolutionRule} the ``{\bf parameter evolution matrix}'' as in \cite{isett}, and we denote this matrix by $T_\de$.  Since $T_\de$ is lower triangular, the eigenvalues of $T_\de$ are the diagonal entries $(1+\de, \ep = \fr{\de}{2}, 1)$.  For large $k$, the $\psi_{(k)}$  are (projectively) concentrated near the eigenline corresponding to the largest eigenvalue, $(1 + \de)$, which is spanned by the eigenvector 
\ali{
\psi_+ &:= \vect{ - (1 + \fr{\de}{2}) \\ \de \\ \fr{3}{2} + \de} \in \mbox{NS}[T_\de - (1{+} \de) I ] 
= \mbox{NS} \mat{ccc}{ 0 & 0 & 0 \\ -\de & -1 - \fr{\de}{2} & 0 \\  -\de & \fr{1}{2} + \fr{\de}{2} & -\de } \label{eq:dominantEvect}
}
More precisely, let $(\psi_+, \psi_\ep, \psi_1)$ be an eigenbasis for $T_\de$ corresponding to the eigenvalues $(1+\de, \ep, 1)$.  In terms of this basis, 
\ali{
\psi_{(k)} &= c_{+,(k)} \psi_+ + c_{\ep, (k)} \psi_\ep + c_{1,(k)} \psi_{1} \label{eq:expressEigen}
} 
and $[ - \log Z , \log Z , \log(C_\eta Z) ]^t = u_+ \psi_+ + u_\ep \psi_\ep + u_1 \psi_1$.  In this basis, \eqref{eq:logEvolutionRule} transforms to
\ali{
c_{+,(k+1)} &= u_+ + (1+\de) c_{+,(k)} \label{eq:dominantEqnDiag} \\
c_{\ep,(k+1)} &= u_\ep + \ep c_{\ep,(k)} \label{eq:epEqndiag} \\
c_{1,(k+1)} &= u_1 + c_{1,(k)} \label{eq:1eqndiag}
}
From \eqref{eq:epEqndiag} and \eqref{eq:1eqndiag}, one obtains by induction a linear upper bound of
\ali{
|c_{\ep,(k)}| + |c_{1,(k)}| \leq |k|(|u_\ep| + |u_1|) + |c_{\ep,(0)}| + |c_{1,(0)}| \label{eq:minorParamBounds}
}
As for \eqref{eq:dominantEqnDiag}, we claim that $u_+ > 0$ and
\ali{
c_{+,(k)} \geq (1+\de)^k c_{+, (0)} > 0 \quad \tx{ for all } k \label{eq:dominantEqn}
}
To prove the claim, we use the fact that $[1,0,0]$ is a $(1 + \de)$-row-eigenvector: $[1, 0, 0] [T_\de - (1+ \de) I ] = 0$.  
It is therefore invariant under the projection to the $1 + \de$ eigenspace:
\ali{
[1, 0, 0] &= [1, 0, 0]\left[ \fr{(T_ \de - \ep I)}{(1+\de - \ep)} \fr{(T_\de - I)}{\de} \right] \label{eq:rowProject}
}
Applying $[1,0,0]$ to \eqref{eq:expressEigen} and using \eqref{eq:dominantEvect} and \eqref{eq:rowProject}, one obtains $u_+ = \fr{\log Z}{(1 + \fr{\de}{2})}$ and $c_{+,(k)} = - \fr{\log e_{R,(k)}}{(1 + \fr{\de}{2})}$.  From this calculation, the claim \eqref{eq:dominantEqn} follows from \eqref{eq:dominantEqnDiag} by induction.

We now turn to the estimate \eqref{eq:ctxabdCorrect}.  Let $E_{\a,(k)}$ denote the right hand side of the upper bound of \eqref{eq:ctxabdCorrect}.  Then using \eqref{eq:expressEigen} and \eqref{eq:minorParamBounds}
\ali{
\log E_{\a,(k)} &= [1/2 - \de \a, 1/2, \a ] \psi_{(k)} \label{eq:logBoundEalk}\\
\log E_{\a,(k)} &= c_{+,(k)} [1/2 - \de \a, 1/2, \a ] \psi_+ + O_{\a,Z}(|k|) \label{eq:mainLogBoundAlp}
}
The $O(\cdot)$ term above grows linearly in $k$ with an implied constant that depends on $\a, Z, C_\eta$.  The assumption \eqref{ineq:messyReg} on $\a$ in Theorem~\ref{thm:messyReg} is exactly the condition that
\ali{
[1/2 - \de \a, 1/2, \a ] \psi_+ &= - \left( \fr{1}{2} - \de \a \right)\left(1 + \fr{\de}{2} \right) + \fr{\de}{2} + \a\left(\fr{3}{2} + \de \right) < 0 \label{ineq:negativeIneq}
} 
Using \eqref{eq:dominantEqn}, \eqref{eq:ctxabdCorrect} and \eqref{eq:mainLogBoundAlp} we obtain a double-exponential decay for the $C_t C_x^\a$ norms of $V_{(k)}$: 
\ALI{
\fr{\| V_{(k)} \|_{C_t C_x^\a}}{C Z^\a} \leq E_{\a, (k)} &\leq e^{ - c_\a ( 1 + \de)^k + O_{\a,Z}(|k|) }
}
The constant $c_\a$ above is a positive number that depends on $\a, Z$ and the initial $(\Xi, e_v, e_R)_{(0)}$.  With this bound, we obtain the desired regularity $v \in C_t C_x^\a$ for our solution.

We now prove the compact support in time for the solution.  By \eqref{eq:supportGrowth}, is suffices to show that the series $\sum_{k=0}^\infty (\Xi_{(k)} e_{v,(k)}^{1/2})^{-1} = \sum_{k = 0} ^\infty e_{R,(k)}^{-1/2} (e_v/e_R)_{(k)}^{-1/2} \Xi_{(k)}^{-1}$ converges to a finite value.  As in the analysis from \eqref{eq:logBoundEalk} to \eqref{ineq:negativeIneq}, it suffices to check that
\ALI{
[-1/2, -1/2, -1] \psi_+ &= \fr{1}{2} \left(1 + \fr{\de}{2} \right) - \fr{\de}{2} - \left(\fr{3}{2} + \de \right) < 0.
}
This calculation concludes the proof of Theorem~\ref{thm:messyReg}.


\appendix
\section{Appendix} \label{sec:appendix}
In this Appendix, we gather several general analysis facts that have been used throughout the proofs of Lemmas~\ref{lem:regSublem}-\ref{lem:convexInt}.  We start by proving Proposition~\ref{prop:LPhold}, which is the well-known Littlewood-Paley characterization of the $\dot{C}^\a$ seminorm.  We refer to Section~\ref{sec:gluingPrelims} for notation.

\begin{proof}[Proof of Proposition~\ref{prop:LPhold}] Let $\| f \|_{\dot{B}^\a_{\infty, \infty}} := \sup_q 2^{\a q} \co{P_q f}$ denote the Littlewood-Paley version of the seminorm.  To see that $\| f \|_{\dot{B}_{\infty, \infty}^\a} \lsm_\a \cda{f}$, we use that $\int_{\R^n} \chi_q(h) dh = 0$ to write 
\ALI{
P_q f(x) &= \int_{\R^n} f(x - h) \chi_q(h) dh \\
&= \int_{\R^n} ( f(x - h) - f(x) ) \chi_q(h) dh \\
| P_q f(x) | &\leq \cda{f} \int_{\R^n} |h|^\a |\chi_q(h)| dh \lsm_\a 2^{-\a q} \cda{f}
} 
Multiplying by $2^{\a q}$ and taking a supremum over $x$ and $q$ gives $\| f \|_{\dot{B}^\a_{\infty, \infty}} \lsm_\a \cda{f}$.

  To prove $\cda{f} \lsm \| f \|_{\dot{B}^\a_{\infty, \infty}}$,	let $x \in \T^n$, $h \in \R^n$, $h \neq 0$.  Choose $\bar{q} \in \Z$ such that $2^{\bar{q} - 1} < |h| \leq 2^{\bar{q}}$.  
Using the decomposition \eqref{eq:LPdecomp} and that $\Pi_0 f$ is a constant, we have
\ALI{
f(x + h) - f(x) &= \sum_{q \leq \bar{q}} [P_q f(x + h) - P_q f(x)] + \sum_{q > \bar{q} } [P_q f(x + h) - P_q f(x)] \\
\left|\sum_{q > \bar{q} } [P_q f(x + h) - P_q f(x)] \right| &\leq \sum_{q > \bar{q} } 2\co{ P_q f } \\
&\leq 2 \sum_{q > \bar{q}} 2^{-\a q} \| f \|_{\dot{B}^\a_{\infty,\infty}} \\
(0 < \a) \Rightarrow \qquad&\lsm_\a 2^{-\a \bar{q}} \| f \|_{\dot{B}^\a_{\infty,\infty}} \lsm |h|^\a  \| f \|_{\dot{B}^\a_{\infty,\infty}} 
}
For the low-frequency part, apply the Mean Value Theorem and $P_q = P_{\approx q} P_q$ to obtain
\ALI{
\left|\sum_{q \leq \bar{q}} [P_q f(x + h) - P_q f(x)] \right| &\leq \sum_{q \leq \bar{q}} |h| \co{ \nb P_q f } \\
&\leq |h| \sum_{q \leq \bar{q}} \co{ \nb P_{\approx q} P_q f } \\
&\leq |h| \sum_{q \leq \bar{q}} \| \nb P_{\approx q} \|~\co{ P_q f } \\
&\lsm_\a |h| \sum_{q \leq \bar{q}} 2^{q} [2^{-\a q } \| f \|_{\dot{B}^\a_{\infty,\infty}}] \\
(\a < 1) \Rightarrow \qquad &\lsm_\a |h| 2^{(1 - \a) \bar{q}} \|  f \|_{\dot{B}^\a_{\infty,\infty}} \leq |h|^\a \|  f \|_{\dot{B}^\a_{\infty,\infty}}
}
\end{proof}
We next prove the commutator estimate of Proposition~\ref{prop:LPcommutunab}.
\begin{proof}[Proof of Proposition~\ref{prop:LPcommutunab}]  Let $u \in L^\infty(\R^n)$ be a smooth vector field and $f \in L^\infty(\R^n)$ be a smooth function.  Then for all $x \in \R^n$
\ALI{
 u\cdot \nb P_q f(x) - P_q[u \cdot \nb f](x) &= u^i(x) \fr{\pr}{\pr x^i} \int_{\R^n} f(x + h) \chi_q(h) dh - \int_{\R^n} u^i(x + h) \fr{\pr}{\pr x^i} f( x + h) \chi_q(h) dh \\
&= \int_{\R^n} ( u^i(x) - u^i(x+h) ) \fr{\pr}{\pr h^i} f(x + h) \chi_q(h) dh \\
&= \int_{\R^n} ( u^i(x) - u^i(x+h) ) \fr{\pr}{\pr h^i}[ f(x+h) - f(x) ] \chi_q(h) dh
}
Using that $u \in L^\infty(\R^n)$, $f \in L^\infty(\R^n)$ and $\chi_q$ is Schwartz, we may integrate by parts in $h$ to obtain
\ali{
u\cdot \nb P_q f(x) - P_q[u \cdot \nb f](x) &= - \int_{\R^n} ( u^i(x) - u^i(x+h) ) (f(x + h) - f(x)) \nb_i \chi_q(h) dh \label{eq:doublediffcomterm}\\
&+ \int_{\R^n} \nb_i u^i(x+h) (f(x+h) - f(x)) \chi_q(h) dh \label{eq:nbiuicomterm}
}
We estimate the terms on the right hand side by
\ALI{
|\eqref{eq:nbiuicomterm}| &\leq \co{\nb u} \cda{f} \int_{\R^n} |h|^\a |\chi_q(h)| dh \\
&\lsm 2^{-\a q} \co{\nb u} \cda{f} \\
\eqref{eq:doublediffcomterm} &=  \int_{\R^n} \left[\int_0^1 \nb_a u^i( x + \si h ) h^a d\si \right] (f(x + h) - f(x) ) \nb_i \chi_q(h) dh \\
|\eqref{eq:doublediffcomterm}| &\leq \co{ \nb u } \cda{f} \int_{\R^n} |h|^{1 + \a} |\nb \chi_q(h)| dh \\
&\lsm 2^{-\a q} \co{ \nb u } \cda{f}
}
\end{proof}

Proposition~\ref{prop:maxTransport} below lists some standard facts about nonsingular linear transport equations.
\begin{prop}  \label{prop:maxTransport} Let $t_0 \in \R$ and $J$ be an open interval in $\R$ containing $t_0$.  Let $u : J \times \T^n \to \R^n$ be a smooth vector field and $g : J \times \T^n \to \R$ be a smooth function, i.e. $u \in C_t C_x^k$, $g \in C_t C_x^k$ for all $k \geq 0$.  Let $f_0 : \T^n \to \R$ be smooth.  Then there exists $f : J \times \T^n \to \R$ such that $f \in C^1(J \times \T^n)$, $f \in \bigcap_{k \geq 0} C_t C_x^k$ is smooth in the spatial variables on $J \times \T^n$, and $f$ satisfies
\ali{
\label{eq:transpEqn}
\begin{split}
(\pr_t + u \cdot \nb) f &= g, \qquad \qquad \tx{ on } J \times \T^n \\
f(t_0, x) &= f_0(x), \qquad \tx{ for } x \in \T^n
\end{split}
}
Furthermore, $f$ is unique among solutions to \eqref{eq:transpEqn} in the class $f \in C^1(J \times \T^n)$, and $f$ satisfies
\ali{
\co{f(t)} &\leq \co{f_0} + \left| \int_{t_0}^t \co{g(\tau)} d\tau \right|, \quad \tx{ for all } t \in J \label{eq:estimateTrans}
}   
\end{prop}
\begin{proof} We only sketch the proof of \eqref{eq:estimateTrans}.  Let $(t, x) \in J \times \T^n$.  Let $\ga(\cdot) : J \to \T^n$ be the unique solution to the ODE $\fr{d \ga}{d\tau} (\tau) = u(\tau, \ga(\tau) )$ with $\ga(t) = x$.  Then for any $f \in C^1(J \times \T^n)$ solving \eqref{eq:transpEqn} and $\tau \in J$, set $\psi(\tau) := f(\tau, \ga(\tau))$.  Then $\fr{d}{d\tau} \psi(\tau) = g(\tau, \ga(\tau))$ for all $\tau \in J$ and $\psi(t_0) = f_0(\ga(t_0))$.  Integrating from $t_0$ to $t$, we have that $|\psi(t)| = |f(t,x)|$ is bounded by the right hand side of \eqref{eq:estimateTrans}.  
\end{proof}

Our proof of the Gluing Approximation Lemma relies on the following Proposition concerning existence of regular solutions to the transport-elliptic equation \eqref{eq:transportElliptic}.  Recall that $\SS \subseteq \R^3 \otimes \R^3$ denotes the space of real symmetric $(2,0)$ tensors.  
\begin{thm} \label{prop:transElliptWPthry}  Let $J$ be an open subinterval of $\R$ and $t_0 \in J$.  Let $v : J \times \T^3 \to \R^3$ be a smooth vector field $v \in \bigcap_{k \geq 0} C_t C_x^k$ that is divergence free $\nb_i v^i = 0$.  Let $\rho_0 : \T^3 \to \SS$ be smooth and $Z : J \times \T^3 \to \R^3$ be a smooth vector field $Z \in \bigcap_{k\geq 0} C_t C_x^k$.  
Then there exists $\rho : J \times \T^n \to \SS$ such that $\rho \in C^1(J \times \T^3)$, $\rho \in \bigcap_{k \geq 0} C_t C_x^k$ is smooth in the spatial variables, and
\ali{
\begin{split} \label{eq:transportElliptic2}
(\pr_t + v \cdot \nab) \rho^{j\ell} &= \RR^{j\ell}[\nb_a v^i \nb_i( \rho^{ab} ) + Z^b ] \\
\rho^{j\ell}(t_0, x) &= \rho_{0}^{j\ell}(x) 
\end{split}
}
\end{thm}
\begin{proof}  The proof of Theorem~\ref{prop:transElliptWPthry} proceeds by modifying the work in \cite[Sections 27.1-27.3]{isett}.  There the analysis specialized to the case of $\rho_0^{j\ell} = \RR^{j\ell}[U]$ and $Z^b = (\pr_t + v \cdot \nb) U$ for some vector field $U$ with spatial integral $0$ (which suffices for the applications of the present paper).  The proof in \cite[Sections 27.1-27.3]{isett} assumes estimates on $v$ and $Z$ that are uniform in time and emphasizes the a priori estimates on the solution.  Here we outline how to adapt the proof to an arbitrary open time interval, and focus on the proof of existence for the solution.

Set $\rho_{(0)}^{j\ell}(t,x) = \rho_0^{j\ell}(x)$, and define $\rho_{(k+1)}^{j\ell}$ to be the unique solution to 
\ali{
\begin{split} \label{eq:transportElliptic3}
(\pr_t + v \cdot \nab) \rho_{(k+1)}^{j\ell} &= \RR^{j\ell}[\nb_a v^i \nb_i( \rho_{(k)}^{ab} ) + Z^b ] \\
\rho^{j\ell}_{(k+1)}(t_0, x) &= \rho_{0}^{j\ell}(x) 
\end{split}
}
Observe that the functions $\rho_{(k)}^{j\ell} : J \times \T^3 \to \SS$ are well defined on all of $J \times \T^3$ and are smooth in the spatial variables: $\rho_{(k)}, \pr_t \rho_{(k)} \in \bigcap_{k \geq 0} C_t C_x^k$.  We claim that for every compact subinterval $\overline{J} \subseteq J$ and every $L \in \Z_+$, the sequence $\rho_{(k)}$ is Cauchy in $C_t C_x^L(\overline{J} \times \T^3)$.

Let $\overline{J}$ be an compact subinterval of $J$.  We claim that for all $L \geq 0$ the sequence $\rho_{(k)}$ is Cauchy in $C_t C_x^L(\overline{J} \times \T^3)$.  Let $L \geq 1$ be given.  Choose parameters $\bar{\tau}^{-1}, \La >0 $ such that
\ali{
\co{\nb^{k+1} v} &\leq \La^k \bar{\tau}^{-1}, \quad \tx{ for all } 0 \leq k \leq L \label{eq:vbdstauinvla}
}
For any $\rho : J \times \T^n \to \SS$ in the class $\bigcap_{k\geq 0} C_t C_x^k$, define the weighted energy
\ALI{
E_L[\rho(t)] &= \sum_{K=1}^L \sum_{j, \ell = 1}^3 \sum_{|\va| = K} \int_{\T^n} \fr{| \nb_{\va} \rho^{j\ell}(t,x)|^4}{\La^{4K}} dx 
}
For $B \geq 1$ to be chosen later and smooth $\rho : \overline{J} \times \T^3 \to \SS$, define the seminorm
\ali{
\| \rho \|_X &= \sup_{t \in \overline{J}} e^{- B \bar{\tau}^{-1} | t - t_0 |} E_L[\rho(t)]^{1/4} \notag
}
Let $T$ be the map such that $\rho_{(k+1)} = T[\rho_{(k)}]$, which is defined by solving \eqref{eq:transportElliptic3}.  We want to show that $T$ is a contraction on $C_t W_x^{L,4}$ if one takes the appropriate norm.  Let $\rho, \tilde{\rho} : J \times \T^3 \to \SS$ be smooth $(2,0)$ tensor fields.  Then differentiating equation \eqref{eq:transportElliptic3} for the difference $T[\rho] - T[\tilde{\rho}]$ and commuting using \eqref{eq:vbdstauinvla}, one obtains that for all $1 \leq |\va| \leq L$ and all $t \in \overline{J}$,
\ali{
\| (\pr_t + v \cdot \nab) \nb_{\va}[T[\rho] - T[\tilde{\rho}] ] \|_{L^4(\T^n)} &\leq C_{|\va|} \La^{|\va|} \bar{\tau}^{-1} ( E_L[\rho - \tilde{\rho}]^{1/4} + E_L[T[\rho] - T[\tilde{\rho}]]^{1/4} ) \label{eq:advecDeriv}
}
The computation follows as in the a priori estimate in \cite[Propositions 27.1-27.2]{isett}.  A key input in this estimate is the fact that $\nab \RR^{j\ell}$ acts as a bounded operator on $L^4(\T^3)$, which follows from the Calderon-Zygmund theory on $\R^3$ as discussed in \cite[Proposition 6.2]{isett}.  

Applying \eqref{eq:advecDeriv} and using (in a non-essential way) that $\nb_i v^i = 0$, 
\ali{
\fr{d}{dt} E_L\big[ T[\rho](t) - T[\tilde{\rho}](t) \big] &= \sum_{K=1}^L \sum_{j, \ell = 1}^3 \sum_{|\va| = K} \La^{-4K} \int_{\T^n} (\pr_t + v \cdot \nab)  | \nb_{\va}\big[T[\rho]^{j\ell} - T[\tilde{\rho}]^{j\ell}\big]|^4(t,x) dx \notag \\
\left| \fr{d}{dt} E_L\big[T[\rho](t) - T[\tilde{\rho}](t) \big] \right| &\stackrel{\eqref{eq:advecDeriv}}{\leq} C_L \bar{\tau}^{-1} ( E_L[\rho - \tilde{\rho}]^{1/4} + E_L\big[T[\rho] - T[\tilde{\rho}]\big]^{1/4} ) E_L\big[T[\rho] - T[\tilde{\rho}] \big]^{\frac{3}{4}} \notag \\
&\leq C_L \bar{\tau}^{-1} (E_L[\rho(t) - \tilde{\rho}(t)] +E_L\big[T[\rho](t) - T[\tilde{\rho}](t)\big] ) \label{eq:afterYoung} \\
\left| \fr{d}{dt} E_L\big[T[\rho](t) - T[\tilde{\rho}](t) \big] \right| &\leq C_L \bar{\tau}^{-1} e^{4 B \bar{\tau}^{-1} |t - t_0|} (\| \rho - \tilde{\rho} \|_X^4 + \| T[\rho] - T[\tilde{\rho}] \|_X^4 )
}
In line \eqref{eq:afterYoung}, we applied Young's inequality with the exponents $\fr{1}{4} + \fr{3}{4} = 1$.  Integrating the above estimate from $t = t_0$ and observing that $E_L[T[\rho](t_0) - T[\tilde{\rho}](t_0)] = 0$, we obtain
\ali{
E_L[T[\rho](t) - T[\tilde{\rho}](t)] &\leq \fr{C_L}{4B} e^{4 B \bar{\tau}^{-1} |t - t_0|} (\|\rho - \tilde{\rho} \|_X^4 + \| T[\rho] - T[\tilde{\rho}] \|_X^4 ) \notag \\
\| T[\rho] - T[\tilde{\rho}] \|_X^4 &\leq \fr{C_L}{4B}\|\rho - \tilde{\rho} \|_X^4 +  \fr{C_L}{4B}\|T[\rho] - T[\tilde{\rho}] \|_X^4 \label{eq:readyToAbsorb}
}
Choosing $B$ large enough, the last term in \eqref{eq:readyToAbsorb} can be subtracted from both sides and we obtain that 
\ali{
\| T[\rho] - T[\tilde{\rho}] \|_X \leq \fr{1}{2} \| \rho - \tilde{\rho} \|_X \label{eq:contraction}
}
It is also true that for all $t \in \overline{J}$
\ali{
 \int_{\T^3} ( T[\rho]^{j\ell}(t,x) - T[\tilde{\rho}]^{j\ell}(t,x) ) dx &= 0. \label{eq:intZero}
}
Equation \eqref{eq:intZero} follows from the following conservation law, which uses $\nb_i v^i = 0$ and \eqref{eq:transportElliptic3}:
\ali{
\fr{d}{dt} \int_{\T^n} T[\rho]^{j\ell}(t,x) dx &= \int_{\T^n} (\pr_t + v\cdot \nab) T[\rho]^{j\ell}(t,x) dx = 0. \notag
}
Combining \eqref{eq:contraction} and \eqref{eq:intZero}, we see that $T$ is a contraction on the space of $C_t W_x^{L,4}$ tensor fields $\rho : \overline{J} \times \T^3 \to \SS$ when this space is endowed with the norm $\| \rho \| = \sup_{t\in \overline{J}} | \int_{\T^n} \rho(t,x) dx| + \| \rho \|_X$.  In particular, the sequence $\rho_{(k)}^{j\ell}$ is Cauchy in $C_t W_x^{L,4}$ and hence Cauchy in $C_t C_x^{L - 1}$ on $\overline{J} \times \T^3$ by Sobolev embedding.  Since $\overline{J}$ was an arbitrary compact subinterval of $J$ containing $t_0$ and $L$ was also arbitrary, we conclude that $\rho_{(k)}^{j\ell}$ converges to a limit $\rho$ that exists on all of $J \times \T^3$ and is smooth $\rho \in \bigcap_{k \geq 0} C_t C_x^k$.  It also follows that $\rho$ solves the initial value problem \eqref{eq:transportElliptic2} (as a distribution), from which we also have that $\rho \in C^1(J \times \T^3)$ and $\pr_t \rho \in \bigcap_{k \geq 0} C_t C_x^k$.
\end{proof}

\bibliographystyle{alpha}
\bibliography{eulerOnRn}

\end{document}